\documentclass[11pt]{article}
\usepackage{amssymb,amsmath,accents}
\usepackage{amscd}
\usepackage{amsfonts,amsthm,mathrsfs}
\usepackage{setspace} 
\usepackage{graphics}
\usepackage{graphicx}
\usepackage{latexsym}
\usepackage{color}
\usepackage{xcolor}
\usepackage{authblk}
\usepackage[title]{appendix}

\usepackage{subfigure}
\usepackage{placeins}
\usepackage[normalem]{ulem}
\usepackage{hyperref}
\hypersetup{
    colorlinks = true,
    urlcolor   = blue,
    citecolor  = black,
}

\newtheorem{theorem}{Theorem}
\newtheorem{lemma}[theorem]{Lemma}
\newtheorem{proposition}[theorem]{Proposition}
\newtheorem{corollary}[theorem]{Corollary}

\newcommand{\vertiii}[1]{{\left\vert\kern-0.25ex\left\vert\kern-0.25ex\left\vert #1 \right\vert\kern-0.25ex\right\vert\kern-0.25ex\right\vert}}
\theoremstyle{definition}
\newtheorem{definition}[theorem]{Definition}
\newtheorem{remark}[theorem]{Remark}

\title{\textbf{The incompressible Navier-Stokes limit from the lattice BGK Boltzmann equation}}

\author[1]{Zhongyang Gu\thanks{zgu@ms.u-tokyo.ac.jp}}  
\author[1]{Xin Hu\thanks{k99poaaron2@gmail.com}}
\author[2]{Pritpal Matharu\thanks{pritpal@kth.se}}
\author[3]{Bartosz Protas\thanks{bprotas@mcmaster.ca}}
\author[1]{Makiko Sasada\thanks{sasada@ms.u-tokyo.ac.jp}}
\author[4]{Tsuyoshi Yoneda\thanks{t.yoneda@r.hit-u.ac.jp}}
\affil[1]{Graduate School of Mathematical Sciences, The University of Tokyo, 3-8-1 Komaba, Meguro-ku, Tokyo 153-8914, Japan}
\affil[2]{Department of Mathematics, KTH Royal Institute of Technology, Lindstedtsv{\"a}gen 25, 114 28 Stockholm, Sweden}
\affil[3]{Department of Mathematics and Statistics, McMaster University, 1280 Main Street West, Hamilton, Ontario L8S 4K1, Canada}
\affil[4]{Graduate School of Economics, Hitotsubashi University, 2-1 Naka, Kunitachi, Tokyo 186-8601, Japan}

\date{}							

\begin{document}
\maketitle

\begin{abstract}
In this paper, we prove that a local weak solution to the $d$-dimensional incompressible Navier-Stokes equations ($d \geq 2$) can be constructed by taking the hydrodynamic limit of a velocity-discretized Boltzmann equation with a simplified BGK collision operator. Moreover, in the case when the dimension is $d=2,3$, we characterize the combinations of finitely many particle velocities and probabilities that lead to the incompressible Navier-Stokes equations in the hydrodynamic limit. Numerical computations conducted in 2D provide information about the rate with which this hydrodynamic limit is achieved when the Knudsen number tends to zero.
\end{abstract}

\begin{center}
Keywords: Incompressible Navier-Stokes equations, Lattice BGK Boltzmann equation, Hydrodynamic limit.
\end{center}

\section{Introduction}
\label{Sec:Intro}

The lattice Boltzmann method (LBM for short), originating from lattice gas automata, is a successful and promising numerical scheme for simulating fluid flows. 
Different from conventional schemes which are based on macroscopic continuum equations, the LBM is based on microscopic kinetic equations.
The approach of using microscopic kinetic equations has advantages over other computational fluid dynamics methods in the sense of easy implementation of boundary conditions and parallel algorithms.
The underlying reason for the effectiveness of the LBM is because of the fact that macroscopic fluid phenomena are collective behaviors of microscopic interactions between particles.
The LBM is characterized by the numerical implementation of the lattice Boltzmann equation (LBE for short)
\begin{align} \label{LBE:CDis}
f_i(x + c_i \Delta x, t + \Delta t) = f_i(x,t) + \Omega_i(x,t), \quad i = 0,1, ..., M,
\end{align}
where $M$ represents the number of non-zero particle velocities that we want to consider, $c_i$ represents the $i$-th local particle velocity, $f_i$ represents the particle velocity distribution function along $c_i$ and $\Omega_i$ represents the collision operator which characterizes the rate of change of $f_i$ resulting from collisions.
There are many different choices for the collision operator $\Omega_i$ in (\ref{LBE:CDis}).
Among which, the one that is most commonly used for Navier-Stokes
simulations is the lattice Bhatnagar–Gross–Krook (BGK)
collision operator \cite{BGK}
\begin{align} \label{Col:LBGK}
\Omega_i(f) = \frac{\Delta t}{\tau_R} (f_{i,eq} - f_i), \quad f_{i,eq} = w_i \varrho \Big( 1 + 3 c_i \cdot U + \frac{9}{2} (c_i \cdot U)^2 - \frac{3}{2} |U|^2 \Big)
\end{align}
where $\tau_R$ is the relaxation time, $w_i$ represents certain weight for velocity $c_i$,
\begin{align*}
\varrho = \sum_i f_i \quad \text{and} \quad \varrho U = \sum_i c_i f_i.
\end{align*}
In particular, $\varrho$ is the macroscopic fluid density and $U$ is the macroscopic fluid velocity.
Throughout this paper, we use $d \in \mathbf{N}$ to denote the space dimension and we work with the case where the dimension $d \geq 2$.
The $f_{i,eq}$ term in the lattice BGK collision operator (\ref{Col:LBGK}) is derived from the expansion of the Maxwell-Boltzmann equilibrium distribution 
\begin{align} \label{Col:BGK}
\mathcal{M}_f(x,v,t) = \frac{\varrho}{(2 \pi \theta)^{d/2}} \mathrm{exp} \Big( - \frac{|v - U|^2}{2 \theta} \Big)
\end{align}
with respect to the velocity variable $v$ in Hermite polynomials up to the third moment, i.e., polynomials in $v$ of power $2$; see e.g. \cite[Chapter 3]{LBM6}.
In the standard BGK Boltzmann equation, the factor $\theta$ in (\ref{Col:BGK}) is the macroscopic fluid temperature defined by
\begin{align*}
\varrho U^2 + d \varrho \theta = \int_{\mathbf{R}^d} f(x,v,t) |v|^2 \, dv.
\end{align*}
In addition, the $f_{i,eq}$ term in (\ref{Col:LBGK}) is obtained under the isothermal assumption that $\theta=1$.

It is well-known that the Navier-Stokes equations can be derived formally from the LBE (\ref{LBE:CDis}) by the Taylor expansion and the Chapman-Enskog expansion, see e.g. \cite{ChenDoolen}.
This is the mathematical foundation for the usage of the lattice Boltzmann equation (\ref{LBE:CDis}) to model fluid dynamics.
However, this derivation is not rigorous from the perspective of
mathematical analysis since not only how (\ref{LBE:CDis})  relates to
the discretization of the Navier-Stokes equations is not known, but
also the justification of
how the continuous equation  (\ref{LBE:CDis}) relates to the Navier-Stokes equations is missing.
In order to justify the relation between the continuous equation (\ref{LBE:CDis}) and the Navier-Stokes equations in the most rigorous way, one should investigate the hydrodynamic limit of the continuous lattice Boltzmann equation when the Knudsen number is converging to zero.

The study of the hydrodynamic limit from the Boltzmann equation was initiated by the work of Bardos, Golse and Levermore \cite{BGL93}, where they derived Leray solutions to the incompressible Navier-Stokes equations from DiPerna's-Lions' renormalized solutions \cite{DipLio} of the Boltzmann equation with Grad's cutoff kernel.
Since then, the hydrodynamic limit of the Boltzmann equation has become one of the major research topics in fluid mechanics, even until today.
Many significant improvements were established subsequently.
For example, Lions and Masmoudi \cite{LioMas} extended the work of Bardos, Golse and Levermore \cite{BGL93} to a more general time-continuous case. 
Later on, Golse and Saint-Raymond \cite{GSR09} proved the convergence of DiPerna-Lions' renormalized solutions \cite{DipLio} to Leray solutions in the case for hard cutoff potentials.
Same convergence for the case of soft potentials was established by Levermore and Masmoudi \cite{LevMas}.
Furthermore, Arsenio \cite{Arse} extended this convergence to the case for non-cutoff potentials.
Saint-Raymond \cite{SaRay} considered the Boltzmann equation with BGK collision operator and showed that solutions to the Boltzmann equation with BGK collision operator that are fluctuations near the Maxwellian converge in hydrodynamic limit to Leray solutions of the incompressible Navier-Stokes-Fourier equations.
It should be remarked that for all these fascinating results that are mentioned above, the solvability of the Boltzmann equation is proved using an entropy method approach.
In order to obtain solutions with better regularity than Leray solutions in the hydrodynamic limit, Jiang, Xu and Zhao \cite{JXZ} employed an energy method approach to solve the Boltzmann equation in cases for both non-cutoff and cutoff collision operators.
They derived a global energy estimate for the Boltzmann equation which holds when the initial data is sufficiently small.
This global energy estimate guarantees the existence of a global-in-time solution to the Boltzmann equation with Sobolev $H^N (N \geq 3)$ regularity in spatial variable $x$ when the initial data is sufficiently small.
Hence, by taking the hydrodynamic limit for this family of global solutions, a global classical solution to the incompressible Navier-Stokes-Fourier equations can be obtained provided that the initial fluid velocity is sufficiently small.

The purpose of this paper is to give a completely rigorous justification for the hydrodynamic limit of a velocity-discretized Boltzmann equation with simplified BGK collision operator.
We shall start by introducing our model, i.e., the lattice BGK (LBGK for short) Boltzmann equation, in an explicit way.
Let $n \in \mathbf{N}$ where $n$ denotes the number of non-zero velocities that we want to consider in a lattice structure.
Let $\mathcal{V} := \{ v_0 = 0, v_1, ..., v_n \} \subset \mathbf{R}^d$ be a set of velocities for particles such that $v_i \neq v_j$ for any $i,j \in \{ 0, 1, ..., n \}$ satisfying $i \neq j$.
Let $w = (w_0, w_1, ..., w_n) \in \mathbf{R}_+^{n+1}$ be a weight that we are going to impose on $\mathcal{V}$ where $\mathbf{R}_+ := \{ a \in \mathbf{R} \bigm| a > 0 \}$. 
We call the pair $(\mathcal{V}, w)$ a lattice.
\begin{definition} \label{Iso:Lattice}
We call a lattice $(\mathcal{V}, w)$ to be an isotropic lattice
associated with the speed of sound $c_s > 0$ if they satisfy
\begin{equation} \label{Sum:wv}
\left\{
\begin{aligned} 
\sum_{i=0}^n w_i &=1,& \\
\sum_{i=0}^n w_i v_{i,\alpha} &=0,& \\
\sum_{i=0}^n w_i v_{i,\alpha} v_{i,\beta} &= c_s^2 \delta_{\alpha \beta},& \\
\sum_{i=0}^n w_i v_{i,\alpha} v_{i,\beta} v_{i,\gamma} &=0,& \\
\sum_{i=0}^n w_i v_{i,\alpha} v_{i,\beta} v_{i,\gamma} v_{i,\zeta} &= c_s^4 ( \delta_{\alpha \beta} \delta_{\gamma \zeta} + \delta_{\alpha \gamma} \delta_{\beta \zeta} + \delta_{\alpha \zeta} \delta_{\beta \gamma} ),&
\end{aligned}
\right.
\end{equation}
where the notation $v_{i,\eta}$ ($1 \leq \eta \leq d$) represents the $\eta$-component of particle velocity $v_i$.
\end{definition}

The notion of an ``isotropic lattice'' follows from the standard terminology describing lattice symmetries in the classical lattice Boltzmann models, see, e.g., \cite{Wolf}.
Let $\varepsilon>0$ be the Knudsen number and $(\mathcal{V}, w)$ be an isotropic lattice. 
We define the LBGK Boltzmann equation on $(\mathcal{V}, w)$ in $\mathbf{R}^d$ to be the vector system
\begin{equation} \label{LBE}
\varepsilon \partial_t g^\varepsilon + v \cdot \nabla_x g^\varepsilon = \frac{1}{\varepsilon \nu} (g_{eq}^\varepsilon - g^\varepsilon), \quad g^\varepsilon \bigm|_{t=0} = g^{0, \varepsilon},
\end{equation}
where $\nu$ denotes the relaxation time,
\[
g^\varepsilon := (g_0^\varepsilon, g_1^\varepsilon, ..., g_n^\varepsilon), \quad g_{eq}^\varepsilon := (g_{0,eq}^\varepsilon, g_{1,eq}^\varepsilon, ..., g_{n,eq}^\varepsilon), \quad g^{0, \varepsilon} := (g_0^{0, \varepsilon}, g_1^{0, \varepsilon}, ..., g_n^{0, \varepsilon})
\]
and
\[
v := (v_0, v_1, ..., v_n)^\mathrm{T}, \quad \nabla_x g^\varepsilon = (\nabla_x g_0^\varepsilon, \nabla_x g_1^\varepsilon, ..., \nabla_x g_n^\varepsilon).
\]
For each $i$ which takes integer value from $0$ to $n$, the $v_i$ in velocity matrix $v$ is the $i$-th velocity defined in $\mathcal{V}$ and the $i$-th component of the LBGK Boltzmann equation (\ref{LBE}) reads as
\begin{equation*} 
\varepsilon \partial_t g^\varepsilon_i + v_i \cdot \nabla_x g_i^\varepsilon =\frac{1}{\varepsilon \nu} ( g_{i,eq}^\varepsilon-g^\varepsilon_i ), \quad g^{\varepsilon}_i \bigm|_{t=0} = g_i^{0, \varepsilon},
\end{equation*}
where 
\begin{equation} \label{Ex:gieq}
g^\varepsilon_{i,eq} := \rho^\varepsilon + \frac{v_i \cdot u^\varepsilon}{c_s^2} + \frac{\varepsilon}{2 c_s^4} \sum_{\alpha,\beta = 1}^d u^\varepsilon_\alpha u^\varepsilon_\beta ( v_{i,\alpha}v_{i,\beta} - c_s^2 \delta_{\alpha \beta} )
\end{equation}
with 
\begin{equation} \label{Def:rhou}
\rho^\varepsilon := \sum_{i=0}^n w_i g_i^\varepsilon, \quad u^\varepsilon := \sum_{i=0}^n w_i v_i g_i^\varepsilon.
\end{equation}
Following the standard terminology, we call the numerical implementation of the LBGK Boltzmann equation (\ref{LBE}) on an isotropic lattice to be the D$d$Q$n$ scheme, see e.g. \cite{LBM6}.
As an important fact, we emphasize that the summation condition (\ref{Sum:wv}) implies that
\begin{equation} \label{Alt:rhou}
\rho^\varepsilon = \sum_{i=0}^n w_i g_{i,eq}^\varepsilon, \quad u^\varepsilon = \sum_{i=0}^n w_i v_i g_{i,eq}^\varepsilon.
\end{equation}

In order to achieve our goal, we follow the philosophy of Jiang, Xu
and Zhao \cite{JXZ} in using the energy method approach to solve the
Boltzmann equation, and the fundamental idea of Bardos, Golse and
Levermore \cite{BGL93} in taking the hydrodynamic limits.  Apart from
the rigorous analysis which constructs a local solution to the
incompressible Navier-Stokes equations from the LBGK Boltzmann
equation, our work also contains a numerical part, namely, we conduct computations in 2D to provide information
  about the rate with which the hydrodynamic limit is achieved when
  the Knudsen number tends to zero.  More specifically, we work with
the D$2$Q$9$ scheme to solve the LBGK Boltzmann equation (\ref{LBE})
by considering $9$ weights
\begin{equation} \label{D2Q9w}
w_i=\left\{
 \begin{aligned}
 \frac{4}{9} \quad &\text{for} \quad i=0,& \\
 \frac{1}{9} \quad &\text{for} \quad i=1, 2, 3, 4,& \\
 \frac{1}{36} \quad &\text{for}  \quad i=5, 6, 7, 8&
 \end{aligned}
\right.
\end{equation}
and $9$ velocities
\begin{equation} \label{D2Q9v}
\left\{
 \begin{aligned}
 v_0 &= (0,0),& \\
 v_1 &= (1,0),& \quad v_2 &=(0,1),& \quad v_3 &=(-1,0),& \quad v_4 &=(0,-1),& \\
 v_5 &= (1,1),& \quad v_6 &=(-1,1),& \quad v_7 &=(-1,-1),& \quad v_8 &=(1,-1)&
 \end{aligned}
\right.
\end{equation}
with sound speed $c_s = \frac{1}{\sqrt{3}}$.

\subsection{Rigorous justification of the hydrodynamic limit: Analysis part}
\label{Sub:InAna}

In order to establish the solvability of the LBGK Boltzmann equation in a simple way, we follow the philosophy in \cite{GHYpv} to do cutoff to equation (\ref{LBE}) in Fourier space, i.e., we work with the approximate equation of (\ref{LBE}) instead of (\ref{LBE}) itself.
We shall see by the end of Section \ref{Sec:DerNS} that this consideration would not affect the incompressible Navier-Stokes limit.
To summarize the underlying reason in a philosophical sentence, the hydrodynamic limit is not sensitive to small changes on the form of the Boltzmann equation.

Suppose that $(\mathcal{V}, w)$ is an isotropic lattice. 
We next define the approximate equation of (\ref{LBE}) on $(\mathcal{V}, w)$.
Let $\varepsilon \in (0,1)$. For $h \in L^2(\mathbf{R}^d)$, we define that
\begin{align} \label{Ctf:op}
\Lambda_\varepsilon(h) := \int_{\mathbf{R}^d} 1_{|\xi| < \frac{1}{\varepsilon}}(\xi) \widehat{h}(\xi) \mathrm{e}^{2 \pi i x \cdot \xi} \, d\xi, \quad \widehat{h}(\xi) := \int_{\mathbf{R}^d} h(x) \mathrm{e}^{- 2 \pi i x \cdot \xi} \, dx,
\end{align}
where $1_{|\xi| < \varepsilon^{-1}}$ represents the indicator function for the open ball $\{ \xi \in \mathbf{R}^d \bigm| |\xi| < \varepsilon^{-1} \}$, i.e., $\Lambda_\varepsilon(h)$ is the cutoff in Fourier space for $h$ with respect to spatial $x$-variable.
Since the LBGK Boltzmann equation (\ref{LBE}) is already discretized in velocity variable $v$, there is no need to do cutoff to it in Fourier space with respect to $v$.
Hence, we define the approximate LBGK Boltzmann equation to be the vector equation
\begin{equation} \label{ALBE}
\varepsilon \partial_t \Lambda_\varepsilon(g^\varepsilon) + v \cdot \nabla_x \Lambda_\varepsilon(g^\varepsilon) = \frac{1}{\varepsilon\nu} \big( \mathcal{G}_{eq}^\varepsilon - \Lambda_\varepsilon(g^{\varepsilon}) \big), \quad \Lambda_\varepsilon(g^\varepsilon) \bigm|_{t=0} = \Lambda_\varepsilon(g^{0, \varepsilon})
\end{equation}
where
\[
\Lambda_\varepsilon(g^\varepsilon) = \big( \Lambda_\varepsilon(g_0^\varepsilon), ..., \Lambda_\varepsilon(g_n^\varepsilon) \big), \quad \mathcal{G}_{eq}^\varepsilon := (\mathcal{G}_{0,eq}^\varepsilon, ..., \mathcal{G}_{n,eq}^\varepsilon), \quad \Lambda_\varepsilon(g^{0, \varepsilon}) := \big( \Lambda_\varepsilon(g_0^{0, \varepsilon}), ..., \Lambda_\varepsilon(g_n^{0, \varepsilon}) \big).
\]
For each $i$ which takes integer value from $0$ to $n$, the $i$-th component of the approximate LBGK Boltzmann equation (\ref{ALBE}) reads as
\begin{equation*} 
\varepsilon \partial_t \Lambda_\varepsilon(g_i^\varepsilon) + v_i \cdot \nabla_x \Lambda_\varepsilon(g_i^\varepsilon) = \frac{1}{\varepsilon\nu} \big( \mathcal{G}_{i,eq}^\varepsilon - \Lambda_\varepsilon(g^{\varepsilon}_i) \big), \quad \Lambda_\varepsilon(g^\varepsilon_i) \bigm|_{t=0} = \Lambda_\varepsilon(g_i^{0, \varepsilon})
\end{equation*}
where
\begin{equation} \label{Ctf:gieq}
\begin{split}
&\mathcal{G}_{i,eq}^\varepsilon := \Lambda_\varepsilon(\rho^\varepsilon) + \frac{v_i \cdot \Lambda_\varepsilon(u^\varepsilon)}{c_s^2} + \frac{\varepsilon}{2 c_s^4} \sum_{\alpha, \beta = 1}^d (v_{i,\alpha} v_{i,\beta} - c_s^2 \delta_{\alpha \beta}) \Lambda_\varepsilon\big( \Lambda_\varepsilon(u^\varepsilon)_\alpha \Lambda_\varepsilon(u^\varepsilon)_\beta \big), \\
&\Lambda_\varepsilon(\rho^\varepsilon) := \sum_{i=0}^n w_i \Lambda_\varepsilon(g_i^\varepsilon), \quad \Lambda_\varepsilon(u^\varepsilon) := \sum_{i=0}^n w_i v_i \Lambda_\varepsilon(g_i^\varepsilon)
\end{split}
\end{equation}
and $\Lambda_\varepsilon(u^\varepsilon)_\alpha$ denotes the $\alpha$-th component of $\Lambda_\varepsilon(u^\varepsilon)$.
Analogous to (\ref{Alt:rhou}), we deduce by (\ref{Sum:wv}) that
\begin{align} \label{Alt:Crhu}
\Lambda_\varepsilon(\rho^\varepsilon) = \sum_{i=0}^n w_i \mathcal{G}_{i,eq}^\varepsilon, \quad \Lambda_\varepsilon(u^\varepsilon) = \sum_{i=0}^n w_i v_i \mathcal{G}_{i,eq}^\varepsilon.
\end{align}

Before stating our main results, we would like to define several notations and concepts that will repeatedly appear in this paper.
For $f = \big( f(v_0), ..., f(v_n) \big), \, h = \big( h(v_0), ..., h(v_n) \big) \in \mathbf{R}^{n+1}$, we define inner products
\begin{align*} 
\langle f, h \rangle_{L_v^2} := \sum_{i=0}^n f(v_i) h(v_i), \quad \langle f, h \rangle_{L_{v,w}^2} := \sum_{i=0}^n w_i f(v_i) h(v_i).
\end{align*}
For $f = \big( f(x, v_0), ..., f(x, v_n) \big), \, h = \big( h(x, v_0), ..., h(x, v_n) \big) \in L^2(\mathbf{R}_x^d)$, we define inner products
\begin{align*} 
\langle f, h \rangle_{L_x^2 L_v^2} := \sum_{i=0}^n \int_{\mathbf{R}^d} f(x, v_i) h(x, v_i) \, dx, \quad \langle f, h \rangle_{L_x^2 L_{v,w}^2} := \sum_{i=0}^n w_i \int_{\mathbf{R}^d} f(x, v_i) h(x, v_i) \, dx.
\end{align*}
Let $\mathbf{N}_0 := \mathbf{N} \cup \{0\}$ and $\tau = (\tau_1, \tau_2, ..., \tau_d) \in \mathbf{N}_0^d$. 
We define
\begin{align*}
|\tau|_\mathrm{s} := \sum_{j=1}^d \tau_j
\end{align*}
and use the notation $\partial_x^\tau$ to denote the differentiation $\partial_{x_1}^{\tau_1} \partial_{x_2}^{\tau_2} ... \partial_{x_d}^{\tau_d}$.
For $f = \big( f(x,v_0), ..., f(x,v_n) \big) \in H^\eta(\mathbf{R}_x^d)$ with $\eta \in \mathbf{R}$, we define the norm
\begin{align*}
\| f \|_{H_x^\eta L_{v,w}^2}^2 := \sum_{i=0}^n w_i \| f(\cdot, v_i) \|_{H^\eta(\mathbf{R}_x^d)}^2.
\end{align*}
We set $H_x^\eta L_{v,w}^2$ to be the Sobolev space $H^\eta(\mathbf{R}_x^d)$ equipped with norm $\| \cdot \|_{H_x^\eta L_{v,w}^2}$.
By the Cauchy-Schwarz inequality, we can easily observe that the space $H_x^\eta L_{v,w}^2$ is equivalent to the standard Sobolev space $H^\eta(\mathbf{R}_x^d)$.
Regarding the approximate equation (\ref{ALBE}), we establish its local well-posedness.

\begin{lemma} \label{Sol:ALBE}
Suppose that $(\mathcal{V}, w)$ is an isotropic lattice.
Let $\varepsilon>0$ be fixed and $m \in \mathbf{N}$ with $m > d$.
Then, for any $g^{0, \varepsilon} \in H_x^m L_{v,w}^2$, there exist constants $T_0 = T_0(g^{0, \varepsilon}, \mathcal{V}, c_s, m, d) > 0$ and $c_0 = c_0(\mathcal{V}, c_s, m, d) > 0$, such that the approximate LBGK Boltzmann equation (\ref{ALBE}) on $(\mathcal{V}, w)$ with initial data $\Lambda_\varepsilon(g^{0, \varepsilon})$ is locally well-posedness.
In particular, the unique local solution
\begin{align*}
g^\varepsilon \in L^\infty( [0,T_0]; H_x^m L_{v,w}^2 )
\end{align*}
satisfies $g^\varepsilon \bigm|_{t=0} = \Lambda_\varepsilon(g^{0, \varepsilon})$, $g^\varepsilon = \Lambda_\varepsilon(g^\varepsilon)$ and the local energy inequality
\begin{equation} \label{loEE:ALBE}
\begin{split}
&\sup_{t \in [0,T_0]} \| g^\varepsilon(t) \|_{H_x^m L_{v,w}^2}^2 + \frac{1}{\varepsilon^2} \int_0^{T_0} \| \mathcal{G}_{eq}^\varepsilon(s) - g^\varepsilon(s) \|_{H_x^m L_{v,w}^2}^2 \, ds \\
&\ \ \leq \frac{(\nu + \nu^2) \| \Lambda_\varepsilon(g^{0, \varepsilon}) \|_{H_x^m L_{v,w}^2}^2}{\nu - c_0 T_0 \| \Lambda_\varepsilon(g^{0, \varepsilon}) \|_{H_x^m L_{v,w}^2}^2} - \nu \| \Lambda_\varepsilon(g^{0, \varepsilon}) \|_{H_x^m L_{v,w}^2}^2.
\end{split}
\end{equation}
\end{lemma}

The merit of working with the approximate equation (\ref{ALBE}) instead of the original LBGK Boltzmann equation (\ref{LBE}) is that the existence of a local solution to (\ref{ALBE}) can be constructed easily using the classical Picard's method.
Since we are considering the case where $m > d$, the Sobolev space $H^m(\mathbf{R}_x^d)$ is indeed a Banach algebra.
As a result, the local energy estimate (\ref{loEE:ALBE}) can be derived easily by the traditional energy method. 
To show the existence of a local solution, for each $0 \leq i \leq n$, we consider a sequence of functions $\{ f_{i, j}^\varepsilon(t) \}_{j \in \mathbf{N}_0}$ that is constructed inductively by
\begin{align*}
f_{i, j+1}^\varepsilon(t) := f_{i, 0}^\varepsilon - \frac{1}{\varepsilon} \int_0^t v_i \cdot \nabla_x f_{i, j}^\varepsilon(s) \, ds + \frac{1}{\varepsilon^2 \nu} \int_0^t \mathcal{G}_{i,eq,j}^\varepsilon(s) - f_{i,j}^\varepsilon(s) \, ds
\end{align*}
for $j \geq 0$ where $f_{i,0}^\varepsilon = \Lambda_\varepsilon(g_{i,0}^\varepsilon)$,
\begin{align*}
&\mathcal{G}_{i,eq,j}^\varepsilon := \Lambda_\varepsilon(\rho_j^\varepsilon) + \frac{v_i \cdot \Lambda_\varepsilon(u_j^\varepsilon)}{c_s^2} + \frac{\varepsilon}{2 c_s^4} \sum_{\alpha, \beta=1}^d (v_{i,\alpha} v_{i,\beta} - c_s^2 \delta_{\alpha \beta}) \Lambda_\varepsilon\big( \Lambda_\varepsilon(u_{j,\alpha}^\varepsilon) \Lambda_\varepsilon(u_{j,\beta}^\varepsilon) \big),  \\
&\Lambda_\varepsilon(\rho_j^\varepsilon) := \sum_{i=0}^n w_i f_{i,j}^\varepsilon, \quad \Lambda_\varepsilon(u_j^\varepsilon) := \sum_{i=0}^n w_i v_i f_{i,j}^\varepsilon.
\end{align*}
Since $f_{i,j}^\varepsilon$ has compact support in Fourier space, by the Bernstein-type lemma (see e.g. \cite[Lemma 2.1]{BCD}), we can control the $H_x^m$ norm of $v_i \cdot \nabla_x f_{i,j}^\varepsilon$ by the $H_x^m$ norm of $f_{i,j}^\varepsilon$. Although in this case $\varepsilon$ with negative power would appear in the coefficient, this is harmless. Thus, we can construct a local solution using the simple Picard's iteration.
Let $f_j^\varepsilon := ( f_{i,j}^\varepsilon )_{0 \leq i \leq n}$. 
We can prove by induction that the sequence $\{ f_j^\varepsilon(t) \}_{j \in \mathbf{N}_0}$ is Cauchy in $L^\infty( [0,T_\ast^\varepsilon]; H_x^m L_{v,w}^2 )$ with some $T_\ast^\varepsilon > 0$ that depends on $\varepsilon$.
The existence of a local solution can then be concluded by the contraction mapping theorem and the Banach fixed point theorem. 
Since the existence time $T_0$ in the local energy inequality (\ref{loEE:ALBE}) is independent of $\varepsilon$, we can extend the local solution up to time $T_0$.
Moreover, by analogous energy method, the uniqueness for the local solution and continuity with respect to initial data can be argued by deriving the energy inequality for the difference between two local solutions that are well-defined on the same time interval.
This completes the proof of Lemma \ref{Sol:ALBE}.

Having the local solvability for the approximate equation (\ref{ALBE}), we can then follow the idea of Bardos, Golse and Levermore \cite{BGL93} to take the hydrodynamic limit.
The main convergence theorem of this paper reads as follows.

\begin{theorem} \label{MT}
Let $m \in \mathbf{N}$ with $m > d$. For any $(\rho_0, u_0) \in H^m(\mathbf{R}_x^d)$ and $0 < \varepsilon < 1$, we consider the initial data $g^{0, \varepsilon} = (g_0^{0, \varepsilon}, g_1^{0, \varepsilon}, ..., g_n^{0, \varepsilon})$ with
\begin{equation}
g_i^{0, \varepsilon} = \Lambda_\varepsilon(\rho_0) + \frac{v_i}{c_s^2} \cdot \Lambda_\varepsilon(u_0), \quad \forall \; 0 \leq i \leq n
\label{gi0eps}
\end{equation}
to the approximate LBGK Boltzmann equation (\ref{ALBE}) where $\Lambda_\varepsilon$ is the cutoff in Fourier space operator defined by (\ref{Ctf:op}).
Suppose that $\{ \varepsilon_n \}_{n \in \mathbf{N}} \subset (0,1)$ is a sequence which converges to $0$ as $n \to \infty$. For each $n \in \mathbf{N}$, let $g^{\varepsilon_n}$ be the unique local solution to the approximate LBGK Boltzmann equation (\ref{ALBE}) established in Lemma \ref{Sol:ALBE} corresponding to the initial data $g^{0, \varepsilon_n}$.
Then, there exist a subsequence $\{ \varepsilon_{n(k)} \}_{k \in \mathbf{N}}$ and $(\rho, u) \in L^\infty( [0,T_0]; H^m(\mathbf{R}_x^d) )$ such that
\begin{align*}
g_i^{\varepsilon_{n(k)}} \overset{\ast}{\rightharpoonup} \rho + \frac{v_i \cdot u}{c_s^2} \quad \text{as} \quad k \to \infty
\end{align*}
in the weak-$\ast$ topology $\sigma\big( L^\infty( [0,T_0]; H^m(\mathbf{R}_x^d) ); L^1( [0,T_0]; H^{-m}(\mathbf{R}_x^d) ) \big)$.
Moreover, $u$ is a local weak solution to the Cauchy problem of the incompressible Navier-Stokes equations
\begin{equation} \label{InNS:lim}
\left\{
 \begin{aligned}
 \partial_t u - c_s^2 \nu \Delta_x u + \nabla_x \cdot (u \otimes u) + \nabla_x p &= 0,& \\
 \nabla_x \cdot u &= 0,& \\
 u(x,0) &= \mathbb{P}(u_0)&
 \end{aligned}
\right.
\end{equation}
where $\mathbb{P}$ denotes the Helmholtz projection for $L^2(\mathbf{R}_x^d)$. 
Furthermore, $u \in C( [0,T_0]; H^{m-1}(\mathbf{R}_x^d) )$ satisfies the local energy inequality
\begin{align*}
\| u \|_{L^\infty( [0,T_0]; H^m(\mathbf{R}_x^d) )} \lesssim \| \rho_0 \|_{H^m(\mathbf{R}_x^d)} + \| u_0 \|_{H^m(\mathbf{R}_x^d)}.
\end{align*}
\end{theorem}

The formal derivation of the incompressible Navier-Stokes equations is somehow standard.
By taking the inner product of the approximate equation (\ref{ALBE}) with $1$ and $v$ in the $L_{v,w}^2$ sense, we obtain that
\begin{equation} \label{LBGK:1v}
\left\{
 \begin{aligned}
 \partial_t \rho^\varepsilon + \frac{1}{\varepsilon} \nabla_x \cdot u^\varepsilon &=0,& \\
 \partial_t u^\varepsilon + \sum_{i=0}^n \frac{1}{\varepsilon} w_i \nabla_x \cdot (v_i \otimes v_i g_i^\varepsilon) &=0.
 \end{aligned}
\right.
\end{equation}
It can be easily observed that the divergence free condition comes from the first equation of (\ref{LBGK:1v}) in hydrodynamic limit.
To derive the main part of the Navier-Stokes equations, which results from the second equation of (\ref{LBGK:1v}), more effort is needed.
For each $0 \leq i \leq n$, we have to use the matrix $A_i := (v_i \otimes v_i) - c_s^2 I$ to rewrite 
\begin{align*} 
\nabla_x \cdot (v_i \otimes v_i g_i^\varepsilon) = \nabla_x \cdot \big( A_i (g_i^\varepsilon - \mathcal{G}^\varepsilon_{i,eq}) \big) + \nabla_x \cdot (A_i \mathcal{G}^\varepsilon_{i,eq}) + c_s^2 \nabla_x g_i^\varepsilon.
\end{align*}
Then, by making use of the isotropic summation condition (\ref{Sum:wv}), we can show that the transport term $u^\varepsilon \cdot \nabla_x u^\varepsilon$ comes from the summation of $\varepsilon^{-1} w_i \nabla_x \cdot (A_i \mathcal{G}_{i,eq}^\varepsilon)$ in $i$ whereas the diffusion term $\Delta_x u^\varepsilon$ comes from the summation of $\varepsilon^{-1} w_i \nabla_x \cdot \big( A_i (g_i^\varepsilon - \mathcal{G}_{i,eq}^\varepsilon) \big)$ in $i$.
Hence, the second equation of (\ref{LBGK:1v}) can be rewritten in the form of 
\begin{align} \label{For:NS:M}
\partial_t u^\varepsilon - c_s^2 \nu \Delta_x u^\varepsilon + \nabla_x \cdot (u^\varepsilon \otimes u^\varepsilon) - c_s^2 \nu \nabla_x \operatorname{div} u^\varepsilon + \frac{c_s^2}{\varepsilon} \nabla_x \rho^\varepsilon + \mathcal{R}^\varepsilon = 0
\end{align}
where $\mathcal{R}^\varepsilon$ is a remainder term which converges to
zero in the sense of distributions in the hydrodynamic limit.  The
term $\nabla_x \rho^\varepsilon$ in (\ref{For:NS:M}) is a bad term
with high frequency, its coefficient is of order $\varepsilon^{-1}$
which blows up as $\varepsilon \to 0$.  We have to get rid of this
effect from (\ref{For:NS:M}) before we take the limit
$\varepsilon \to 0$.  In order to do so, we apply the $L^2$ Helmholtz
projection $\mathbb{P}$ to both sides of (\ref{For:NS:M}).  The
gradient term $\varepsilon^{-1} \nabla_x \rho^\varepsilon$ is thus
eliminated.  Without terms having high frequency, we can then show
that $\mathbb{P}(u^\varepsilon)$ is equi-continuous in time $t$ by an
energy argument.  Suppressing subsequences, the strong convergence of
$\mathbb{P}(u^\varepsilon)$ to $u$ as $\varepsilon \to 0$ can be
concluded by the Arzel$\grave{\text{a}}$-Ascoli theorem.  Finally, the
convergence of $u^\varepsilon - \mathbb{P}(u^\varepsilon)$ to zero in
the sense of distributions is guaranteed by a compensated compactness
result by Lions and Masmoudi \cite{LioMas}.  It is worth noting that
the strong convergence of $\mathbb{P}(u^\varepsilon)$ (Lemma
\ref{StCo:Puep}) guarantees that the initial data for the macroscopic
fluid velocity is only determined by the first moment of
$g^{0,\varepsilon}$, i.e., $\sum_{i=0}^n w_i v_i
g_i^{0,\varepsilon}$. Hence, when we construct the microscopic
  initial data for the Lattice BGK Boltzmann system from the
  macroscopic fluid initial data, cf.~formula \eqref{gi0eps}, it is
not necessary to take into account the nonlinear effect in
$g_{i,eq}^\varepsilon$ since the matrix $v_i \otimes v_i - c_s^2 I$ is
orthogonal to $1$ and $v_i$ in the $L_{v,w}^2$ sense.

As the end of the analysis part, we give characterizations to isotropic lattices in the case when $d=2,3$ and $c_s = 3^{-\frac{1}{2}}$. Since the combination of an isotropic lattice and speed of sound is scale invariant (see Remark \ref{supp:Xa}), we have to restrict the size of every component of each particle velocity in order to give a characterization. The key idea for characterizing $2$D and $3$D isotropic lattices is as follows. 
If we require the size condition that every component of each particle velocity takes value in the interval $[-1,1]$, then every component of each particle velocity can only take value in the set $\{-1,0,1\}$. 
Moreover, we note that for $1 \leq i \leq 5$, the $i$-th summation condition in (\ref{Sum:wv}) is indeed the expectations of all possible products of $i-1$ components of particle velocities. 
The characterizations for $2$D and $3$D isotropic lattices (when $c_s = 3^{-\frac{1}{2}}$) can be obtained by algebraic manipulations of summation conditions in (\ref{Sum:wv}) in the language of expectations. 
In particular, in the $2$D case with $c_s = 3^{-\frac{1}{2}}$, we show that if we require that each component of every particle velocity to take value in $[-1,1]$, then the D$2$Q$9$ scheme is the only possible isotropic lattice.
In the $3$D case with $c_s = 3^{-\frac{1}{2}}$, we have more than one possible choice for isotropic lattices, simply because of the system is underdetermined.

\subsection{Numerical investigation of the hydrodynamic limit}
\label{Sub:InNum}

In order to illustrate these concepts, in particular, the
  hydrodynamic limit of the Lattice BGK system, we conduct numerical
  simulations of system \eqref{LBE} for different values of
  $\varepsilon$ and compare these solutions to the solution of the
  Navier-Stokes system \eqref{InNS:lim} with the corresponding
  macroscopic initial condition. However, it should be
    emphasized that since we directly solve the partial differential
    system \eqref{LBE}, this approach is distinct from the LBM. For
  simplicity, we focus on the two-dimensional (2D) case ($d = 2$) and
  consider problems defined on a periodic spatial domain (i.e., a 2D
  torus $\mathbf{T}_x^2$). Flows corresponding to two different
  macroscopic initial conditions are studied, namely, the Taylor-Green
  vortex for which the 2D Navier-Stokes system admits a closed-form
  analytic solution and its perturbed version which leads to a
  turbulent-like evolution featuring repeated filamentation of
  vortices resulting in an enstrophy cascade. In both cases we observe
  that
  $\| \nabla_x u(T) - \nabla_x u^\varepsilon(T) \|_{L^2
    (\mathbf{T}_x^2)} = \mathcal{O}(\varepsilon^2)$ at some time $T>0$
  as $\varepsilon \rightarrow 0$. This observation is interesting as
  it complements the analysis presented in this paper which does not
  provide explicit information about the rate with which the
  hydrodynamic limit is attained.

\subsection{Organization of the paper}
\label{Sub:Org}

This article is organized as follows.
Section \ref{Sec:LocSol} is devoted to the local well-posedness of the approximate LBGK Boltzmann equation (\ref{ALBE}).
In Section \ref{Sub:LocEne}, we employ the standard energy method to derive a local energy inequality which holds for both the approximate and the original LBGK Boltzmann equation.
In Section \ref{Sub:LocExi}, we construct a local solution to the approximate equation (\ref{ALBE}) using the classical Picard's method.
In Section \ref{Sub:LocWell}, we consider analogous energy argument as in section \ref{Sub:LocEne} to show the uniqueness and continuity with respect to initial data.
Section \ref{Sec:DerNS} is devoted to the formal derivation of the incompressible Navier-Stokes equations.
In particular, the transport term $u^\varepsilon \cdot \nabla_x u^\varepsilon$ is derived in section \ref{Sub:DeTrans} and the diffusion term is derived in section \ref{Sub:DeDiff}.
In Section \ref{Sec:HydroLi}, we take the hydrodynamic limit and prove the convergence of the unique local solution for the approximate equation (\ref{ALBE}) to a local weak solution for the incompressible Navier-Stokes equations (\ref{InNS:lim}).
In Section \ref{Sub:HelPNS}, we apply the Helmholtz decomposition $\mathbb{P}$ to get rid of the gradient term that has high frequency in the formal equations derived in Section \ref{Sec:DerNS}.
We also prove the strong convergence $\mathbb{P}(u^\varepsilon)$ to $u$ here.
In Section \ref{Sub:ConNS}, we collect all the convergence results established to give a proof to Theorem \ref{MT}.
In Section \ref{Sec:IsoLat}, we give characterizations to $2$D and $3$D isotropic lattices when $c_s = 3^{-\frac{1}{2}}$ and a size condition is imposed on particle velocities.
In Section \ref{Sub:Cuba}, we rewrite the summation conditions in (\ref{Sum:wv}) using expectations in the probability setting.
Section \ref{Sub:2d3diso} is devoted to the characterizations of $2$D and $3$D isotropic lattices.
In Section \ref{Sec:NumSol}, we present the numerical implementation of the D$2$Q$9$ scheme in the $2$D case to provide numerical evidence that justify the convergence behavior we have proved in the analysis part.

Throughout this paper, the notation $A \lesssim B$ will mean that there exists a constant $c$, which is independent of $\varepsilon$ and $\nu$, such that $A \leq c B$.

\section{Local solvability of the approximate equation}
\label{Sec:LocSol}

First of all, we would like to highlight a simple tool that is crucial for norm estimations of nonlinear terms in this paper.

\begin{proposition} \label{Hm:BA}
For $m \in \mathbf{N}$ such that $m > d$, the Sobolev space $H^m(\mathbf{R}_x^d)$ is a Banach algebra.
\end{proposition}
\begin{proof}
For $f, g \in H^m(\mathbf{R}_x^d)$, we can easily observe by the Cauchy-Schwarz inequality that
\begin{equation*}
\| f g \|_{H^m(\mathbf{R}_x^d)}^2 \leq \sum_{\tau \in \mathbf{N}_0^d, \, | \tau |_\mathrm{s} \leq m} ( | \tau |_\mathrm{s} + 1)^d \sum_{\sigma \in \mathbf{N}_0^d, \, \sigma \leq \tau} \| \partial_x^\sigma f \partial_x^{\tau - \sigma} g \|_{L^2(\mathbf{R}_x^d)}^2
\end{equation*}
where the notation $\sigma \leq \tau$ means that $\sigma_i \leq \tau_i$ for all $1 \leq i \leq d$.
Since $m \geq d+1$, we note that either $|\sigma|_\mathrm{s}$ or $|\tau - \sigma|_\mathrm{s} = |\tau|_\mathrm{s} - |\sigma|_\mathrm{s}$ must be less than or equal to $m - [\frac{d}{2}] - 1$ where $[\frac{d}{2}]$ denotes the largest integer less than or equal to $\frac{d}{2}$. 
Without loss of generality, we may assume that $|\sigma|_\mathrm{s} \leq m - [\frac{d}{2}] - 1$.
Then, by the continuous Sobolev embedding $H^{[\frac{d}{2}] + 1}(\mathbf{R}_x^d) \hookrightarrow L^\infty(\mathbf{R}_x^d)$, we deduce that
\begin{align*}
\| \partial_x^\sigma f \partial_x^{\tau - \sigma} g \|_{L^2(\mathbf{R}_x^d)}^2 &\leq \| \partial_x^\sigma f \|_{L^\infty(\mathbf{R}_x^d)}^2 \| \partial_x^{\tau - \sigma} g \|_{L^2(\mathbf{R}_x^d)}^2 \\
&\leq \| f \|_{H^m(\mathbf{R}_x^d)}^2 \| \partial_x^{\tau - \sigma} g \|_{L^2(\mathbf{R}_x^d)}^2.
\end{align*}
Hence,
\begin{align*}
\| f g \|_{H^m(\mathbf{R}_x^d)}^2 &\leq \| f \|_{H^m(\mathbf{R}_x^d)}^2 \sum_{\tau \in \mathbf{N}_0^d, \, | \tau |_\mathrm{s} \leq m} ( | \tau |_\mathrm{s} + 1)^d \| g \|_{H^{| \tau |_\mathrm{s}}(\mathbf{R}_x^d)}^2 \\
&\leq (m+1)^d \| f \|_{H^m(\mathbf{R}_x^d)}^2 \| g \|_{H^m(\mathbf{R}_x^d)}^2 \cdot \sum_{j=0}^m \binom{j+d-1}{d-1} \\
&\leq (m+1)^{d+1} \frac{(m+d-1)^{d-1}}{(d-1)!} \| f \|_{H^m(\mathbf{R}_x^d)}^2 \| g \|_{H^m(\mathbf{R}_x^d)}^2.
\end{align*}
This completes the proof of Proposition \ref{Hm:BA} since $H^m(\mathbf{R}_x^d)$ is certainly a Banach space.
\end{proof}
\begin{remark} \label{Hm:Linf}
The continuous Sobolev embedding $H^{[\frac{d}{2}] + 1}(\mathbf{R}_x^d) \hookrightarrow L^\infty(\mathbf{R}_x^d)$ plays an important role in proving Proposition \ref{Hm:BA}.
We recall that in fact, it holds more generally that $H^s(\mathbf{R}_x^d) \hookrightarrow L^\infty(\mathbf{R}_x^d)$ whenever $s > \frac{d}{2}$ as
\begin{align*}
\| h \|_{L^\infty(\mathbf{R}_x^d)} \leq \| \widehat{h} \|_{L^1(\mathbf{R}_x^d)} \leq \| \langle \xi \rangle^{-s} \|_{L^2(\mathbf{R}_x^d)} \| \langle \xi \rangle^s \widehat{h} \|_{L^2(\mathbf{R}_x^d)} \lesssim \| h \|_{H^s(\mathbf{R}_x^d)}
\end{align*}
where $\langle \xi \rangle := (1 + |\xi|^2)^{\frac{1}{2}}$ and $\widehat{h}$ denotes the Fourier transform of $h$.
\end{remark}

\subsection{Local energy estimate for the lattice Boltzmann equation}
\label{Sub:LocEne}

The local energy inequality can be derived by working directly with the original lattice Boltzmann equation (\ref{LBE}).

\begin{lemma} \label{locEE:LBE}
Suppose that $(\mathcal{V}, w)$ is an isotropic lattice.
Let $\varepsilon>0$ and $m \in \mathbf{N}$ satisfying $m > d$.
Then, for any $g^{0, \varepsilon} \in H_x^m L_{v,w}^2$, there exists $T_0 = T_0(g^{0, \varepsilon}, \mathcal{V}, c_s, m, d)>0$ such that for any $g^\varepsilon \in L^\infty( [0, T_0]; H_x^m L_{v,w}^2 )$ satisfying the LBGK Boltzmann equation (\ref{LBE}) on $(\mathcal{V}, w)$, the local energy estimate
\begin{equation} \label{locEE}
\begin{split}
&\| g^\varepsilon(t) \|_{H_x^m L_{v,w}^2}^2 + \frac{1}{\varepsilon^2} \int_0^t \| g_{eq}^\varepsilon(s) - g^\varepsilon(s) \|_{H_x^m L_{v,w}^2}^2 \, ds \\
&\ \ \leq \frac{(\nu + \nu^2) \| g^{0, \varepsilon} \|_{H_x^m L_{v,w}^2}^2}{\nu - t \cdot C_{LE} \| g^{0, \varepsilon} \|_{H_x^m L_{v,w}^2}^2} - \nu \| g^{0, \varepsilon} \|_{H_x^m L_{v,w}^2}^2
\end{split}
\end{equation}
holds for any $t \in [0,T_0]$ with a constant $C_{LE} = C_{LE}(\mathcal{V}, c_s, m, d) >0$.
\end{lemma}
\begin{proof}
Let $\tau \in \mathbf{N}_0^d$ satisfying $|\tau|_\mathrm{s} \leq m$.
By applying $\partial_x^\tau$ to the lattice Boltzmann equation (\ref{LBE}) and taking its inner product with $\partial_x^\tau g^\varepsilon$ in the $L_x^2 L_{v,w}^2$ sense, we obtain that
\begin{equation} \label{InP:LBEg}
\begin{split}
&\frac{1}{2} \sum_{i=0}^n w_i \frac{\mathrm{d}}{\mathrm{d} t} \| \partial_x^\tau g_i^\varepsilon \|_{L^2(\mathbf{R}_x^d)}^2 +  \frac{1}{\varepsilon^2 \nu} \sum_{i=0}^n w_i \| \partial_x^\tau g_{i,eq}^\varepsilon - \partial_x^\tau g_i^\varepsilon \|_{L^2(\mathbf{R}_x^d)}^2 \\
&\ \ = \frac{1}{\varepsilon^2 \nu} \sum_{i=0}^n w_i \int_{\mathbb{R}^d} (\partial_x^\tau g_{i,eq}^\varepsilon - \partial_x^\tau g_i^\varepsilon) \partial_x^\tau g_{i,eq}^\varepsilon \, dx.
\end{split}
\end{equation}
By (\ref{Def:rhou}) and (\ref{Alt:rhou}), we observe that
\[
\sum_{i=0}^n w_i (g_{i,eq}^\varepsilon - g_i^\varepsilon) = 0, \quad \sum_{i=0}^n w_i v_i (g_{i,eq}^\varepsilon - g_i^\varepsilon) = 0.
\]
Hence, by substituting expression (\ref{Ex:gieq}) for $g_{i,eq}^\varepsilon$ into the right hand side of (\ref{InP:LBEg}), we deduce that
\begin{equation} \label{dEI:rhs1}
\begin{split}
&\frac{1}{\varepsilon^2 \nu} \sum_{i=0}^n w_i \int_{\mathbb{R}^d} (\partial_x^\tau g_{i,eq}^\varepsilon - \partial_x^\tau g_i^\varepsilon) \partial_x^\tau g_{i,eq}^\varepsilon \, dx \\
&\ \ = \frac{1}{2 c_s^4 \varepsilon \nu} \sum_{i=0}^n \sum_{\alpha, \beta=1}^d \sum_{\sigma \in \mathbf{N}_0^d, \, \sigma \leq \tau} w_i ( v_{i,\alpha} v_{i,\beta} - c_s^2 \delta_{\alpha \beta} ) \int_{\mathbb{R}^d} (\partial_x^\tau g_{i,eq}^\varepsilon - \partial_x^\tau g_i^\varepsilon) \partial_x^\sigma u_\alpha^\varepsilon \partial_x^{\tau - \sigma} u_\beta^\varepsilon \, dx \\
&\ \ \leq \frac{\ell_{\mathcal{V}, c_s}}{2 c_s^4 \varepsilon \nu} \sum_{i=0}^n w_i \| \partial_x^\tau g_{i,eq}^\varepsilon - \partial_x^\tau g_i^\varepsilon \|_{L^2(\mathbf{R}_x^d)} \sum_{\alpha, \beta=1}^d  \sum_{\sigma \in \mathbf{N}_0^d, \, \sigma \leq \tau} \| \partial_x^\sigma u_\alpha^\varepsilon \partial_x^{\tau - \sigma} u_\beta^\varepsilon \|_{L^2(\mathbf{R}_x^d)} \\
&\ \ \leq \frac{1}{2 \varepsilon^2 \nu} \sum_{i=0}^n w_i \| \partial_x^\tau g_{i,eq}^\varepsilon - \partial_x^\tau g_i^\varepsilon \|_{L^2(\mathbf{R}_x^d)}^2 + \frac{\ell_{\mathcal{V}, c_s}^2 d^2 | \tau |_\mathrm{s}^d}{8 c_s^8 \nu} \sum_{\alpha, \beta=1}^d \sum_{\sigma \in \mathbf{N}_0^d, \, \sigma \leq \tau} \| \partial_x^\sigma u_\alpha^\varepsilon \partial_x^{\tau - \sigma} u_\beta^\varepsilon \|_{L^2(\mathbf{R}_x^d)}^2,
\end{split}
\end{equation}
where $\ell_{\mathcal{V}, c_s} := \Big( \underset{1 \leq i \leq n}{\mathrm{max}} \, | v_i | \Big)^2 + c_s^2$.
From the proof of Proposition \ref{Hm:BA}, we see that
\begin{align*} 
\sum_{\sigma \in \mathbf{N}_0^d, \, \sigma \leq \tau} \| \partial_x^\sigma u_\alpha^\varepsilon \partial_x^{\tau - \sigma} u_\beta^\varepsilon \|_{L^2(\mathbf{R}_x^d)}^2 &\leq \| u_\alpha^\varepsilon \|_{H^m(\mathbf{R}_x^d)}^2 \| u_\beta^\varepsilon \|_{H^{|\tau|_\mathrm{s}}(\mathbf{R}_x^d)}^2.
\end{align*}
On the other hand, we observe by (\ref{Def:rhou}) that the estimate
\begin{align*}
\| u_\alpha^\varepsilon \|_{H^k(\mathbf{R}_x^d)}^2 \leq \left( \sum_{i=0}^n w_i \| g_i^\varepsilon \|_{H^k(\mathbf{R}_x^d)} \right)^2 \leq \sum_{i=0}^n w_i \| g_i^\varepsilon \|_{H^k(\mathbf{R}_x^d)}^2
\end{align*}
holds for any $1 \leq \alpha \leq d$ and $k \in \mathbf{N}_0$.
Therefore, (\ref{InP:LBEg}) and (\ref{dEI:rhs1}) imply that
\begin{equation} \label{dEI:tau}
\begin{split}
&\frac{1}{2} \frac{\mathrm{d}}{\mathrm{d} t} \sum_{i=0}^n w_i \| \partial_x^\tau g_i^\varepsilon \|_{L^2(\mathbf{R}_x^d)}^2 +  \frac{1}{2 \varepsilon^2 \nu} \sum_{i=0}^n w_i \| \partial_x^\tau g_{i,eq}^\varepsilon - \partial_x^\tau g_i^\varepsilon \|_{L^2(\mathbf{R}_x^d)}^2 \\
&\ \ \leq \frac{\ell_{\mathcal{V}, c_s}^2 d^4 | \tau |_\mathrm{s}^d}{8 c_s^8 \nu} \| g^\varepsilon \|_{H_x^m L_{v,w}^2}^2 \| g^\varepsilon \|_{H_x^{| \tau |_\mathrm{s}} L_{v,w}^2}^2.
\end{split}
\end{equation}

Summing up inequality (\ref{dEI:tau}) over all $\tau \in \mathbf{N}_0^d$ satisfying $|\tau|_\mathrm{s} \leq m$ gives us
\begin{align} \label{dEI}
\frac{1}{2} \frac{\mathrm{d}}{\mathrm{d} t} \| g^\varepsilon \|_{H_x^m L_{v,w}^2}^2 + \frac{1}{2 \varepsilon^2 \nu} \| g_{eq}^\varepsilon - g^\varepsilon \|_{H_x^m L_{v,w}^2}^2 \leq \frac{\ell_{\mathcal{V}, c_s}^2 d^4 m^d}{8 c_s^8 \nu} \| g^\varepsilon \|_{H_x^m L_{v,w}^2}^4.
\end{align}
Applying Gronwall's inequality (see e.g. \cite[Page 362]{MPF}) to
\[
\frac{\mathrm{d}}{\mathrm{d} t} \| g^\varepsilon \|_{H_x^m L_{v,w}^2}^2 \leq \frac{\ell_{\mathcal{V}, c_s}^2 d^4 m^d}{4 c_s^8 \nu} \| g^\varepsilon \|_{H_x^m L_{v,w}^2}^4,
\]
we deduce that
\begin{align} \label{EI:1p}
\| g^\varepsilon(t) \|_{H_x^m L_{v,w}^2}^2 \leq \frac{\nu \| g^{0, \varepsilon} \|_{H_x^m L_{v,w}^2}^2}{\nu - t \cdot C_{L E} \| g^{0, \varepsilon} \|_{H_x^m L_{v,w}^2}^2},
\end{align}
where $C_{L E} = C_{L E}(\mathcal{V}, c_s, m, d) := 4^{-1} c_s^{-8} \ell_{\mathcal{V}, c_s}^2 d^4 m^d$.
Substituting inequality (\ref{EI:1p}) into the inequality
\[
\frac{1}{\varepsilon^2 C_{LE}} \| g_{eq}^\varepsilon - g^\varepsilon \|_{H_x^m L_{v,w}^2}^2 \leq \| g^\varepsilon \|_{H_x^m L_{v,w}^2}^4,
\]
which is another implication of inequality (\ref{dEI}), we can further deduce that
\begin{align*}
\frac{1}{\varepsilon^2} \int_0^t \| g_{eq}^\varepsilon - g^\varepsilon \|_{H_x^m L_{v,w}^2}^2 \, ds \leq \frac{\nu^2 \| g^{0, \varepsilon} \|_{H_x^m L_{v,w}^2}^2}{\nu - t \cdot C_{LE} \| g^{0, \varepsilon} \|_{H_x^m L_{v,w}^2}^2} - \nu \| g^{0, \varepsilon} \|_{H_x^m L_{v.w}^2}^2.
\end{align*}
Finally, we take $T_0 < \nu C_{LE}^{-1} \| g^{0, \varepsilon} \|_{H^m_x L_{v,w}^2}^{-2}$.
This completes the proof of Lemma \ref{locEE:LBE}.
\end{proof}
\begin{remark} \label{Ctf:AdjP}
It is easy to observe that the cutoff operator $\Lambda_\varepsilon$ commutes with the differentiation $\partial_x$.
Moreover, it holds that for any $f,h \in L_x^2 L_{v,w}^2$,
\begin{align*}
\langle \Lambda_\varepsilon(f), h \rangle_{L_x^2 L_{v,w}^2} = \langle f, \Lambda_\varepsilon(h) \rangle_{L_x^2 L_{v,w}^2} \quad \text{and} \quad \Lambda_\varepsilon\big( \Lambda_\varepsilon(f) \big) = \Lambda_\varepsilon(f).
\end{align*}
Making use of these three properties of the cutoff operator $\Lambda_\varepsilon$, we can deduce that Lemma \ref{locEE:LBE} also holds for the approximate LBGK Boltzmann equation (\ref{ALBE}) defined on an isotropic lattice $(\mathcal{V}, w)$, with $g^{0, \varepsilon}$, $g^\varepsilon$ and $g_{eq}^\varepsilon$ in inequality (\ref{locEE}) being replaced respectively by $\Lambda_\varepsilon(g^{0, \varepsilon})$, $\Lambda_\varepsilon(g^\varepsilon)$ and $\mathcal{G}_{eq}^\varepsilon$.
\end{remark}

\subsection{Existence of a local solution for the approximate equation}
\label{Sub:LocExi}

\begin{proof}[Proof of Lemma \ref{Sol:ALBE} (Existence)]
For simplicity of notations, we denote $f^\varepsilon = ( f_i^\varepsilon )_{0 \leq i \leq n} := \Lambda_\varepsilon(g^\varepsilon)$ and $f_0^\varepsilon = ( f_{i, 0}^\varepsilon )_{0 \leq i \leq n} := \Lambda_\varepsilon(g^{0, \varepsilon})$, i.e., 
\[
f_i^\varepsilon = \Lambda_\varepsilon(g_i^\varepsilon) \quad \text{and} \quad f_{i, 0}^\varepsilon = \Lambda_\varepsilon(g_i^{0, \varepsilon}), \quad 0 \leq i \leq n.
\]
Performing integration of the $i$-th component of the approximate LBGK Boltzmann equation (\ref{ALBE}) for each $0 \leq i \leq n$ with respect to the time variable, we have that
\begin{equation} \label{IE:fi}
\begin{split}
f_i^\varepsilon(t) = f_{i,0}^\varepsilon - \int_0^t \frac{1}{\varepsilon} (v_i \cdot \nabla_x f_i^\varepsilon) \, ds + \int_0^t \frac{1}{\varepsilon^2 \nu} (\mathcal{G}_{i,eq}^\varepsilon - f_i^\varepsilon) \, ds.
\end{split}
\end{equation}
Based on the integral equation (\ref{IE:fi}), for each $0 \leq i \leq n$, we define a sequence of functions $\{ f_{i,j}^\varepsilon \}_{j \in \mathbf{N}_0}$ inductively by 
\begin{equation} \label{IE:fij}
f_{i,j+1}^\varepsilon(t) = f_{i,0}^\varepsilon - \int_0^t \frac{1}{\varepsilon} \big( v_i \cdot \nabla_x f_{i,j}^\varepsilon(s) \big) \, ds + \int_0^t \frac{1}{\varepsilon^2 \nu} \big( \mathcal{G}_{i,eq,j}^\varepsilon(s) - f_{i,j}^\varepsilon(s) \big) \, ds
\end{equation}
for $j \in \mathbf{N}_0$ where
\begin{align*}
&\mathcal{G}_{i,eq,j}^\varepsilon := \Lambda_\varepsilon(\rho_j^\varepsilon) + \frac{v_i \cdot \Lambda_\varepsilon(u_j^\varepsilon)}{c_s^2} + \frac{\varepsilon}{2 c_s^4} \sum_{\alpha, \beta=1}^d (v_{i,\alpha} v_{i,\beta} - c_s^2 \delta_{\alpha \beta}) \Lambda_\varepsilon\big( \Lambda_\varepsilon(u_j^\varepsilon)_\alpha \Lambda_\varepsilon(u_j^\varepsilon)_\beta \big),  \\
&\Lambda_\varepsilon(\rho_j^\varepsilon) := \sum_{i=0}^n w_i f_{i,j}^\varepsilon, \quad \Lambda_\varepsilon(u_j^\varepsilon) := \sum_{i=0}^n w_i v_i f_{i,j}^\varepsilon
\end{align*}
with $\Lambda_\varepsilon(u_j^\varepsilon)_\alpha$ denoting the $\alpha$-th component of $\Lambda_\varepsilon(u_j^\varepsilon)$.
For every $j \in \mathbf{N}$, we set $f_j^\varepsilon := ( f_{i,j}^\varepsilon )_{0 \leq i \leq n}$.
The key idea here is to show that the vector sequence $\{ f_j^\varepsilon \}_{j \in \mathbf{N}_0}$ is Cauchy in $L^\infty( [0,T]; H_x^m L_{v,w}^2 )$ for each $0 \leq i \leq n$.
By Minkowski's integral inequality, for any $j \in \mathbf{N}$ and $0 \leq i \leq n$, we have that
\begin{equation} \label{Pica:dif}
\begin{split}
\| f_{i,j+1}^\varepsilon - f_{i,j}^\varepsilon \|_{H^m(\mathbf{R}_x^d)} &\lesssim \int_0^t \frac{1}{\varepsilon} \| v_i \cdot \nabla_x (f_{i,j}^\varepsilon - f_{i,j-1}^\varepsilon) \|_{H^m(\mathbf{R}_x^d)} \, ds \\
&\ \ + \int_0^t \frac{1}{\varepsilon^2 \nu} \| \mathcal{G}_{i,eq,j}^\varepsilon - \mathcal{G}_{i,eq,j-1}^\varepsilon \|_{H^m(\mathbf{R}_x^d)} \, ds \\
&\ \ + \int_0^t \frac{1}{\varepsilon^2 \nu} \| f_{i,j}^\varepsilon - f_{i,j-1}^\varepsilon \|_{H^m(\mathbf{R}_x^d)} \, ds.
\end{split}
\end{equation}

Since the Fourier transform of $f_{i,j}^\varepsilon$ is supported within $\overline{B_{\varepsilon^{-1}}(0)}$ for any $0 \leq i \leq n$ and $j \in \mathbf{N}_0$, by the Bernstein-type lemma (see \cite[Lemma 2.1]{BCD}), we deduce that 
\begin{align} \label{Es:trans}
\| v_i \cdot \nabla_x (f_{i,j}^\varepsilon - f_{i,j-1}^\varepsilon) \|_{H^m(\mathbf{R}_x^d)} \lesssim \| \nabla_x (f_{i,j}^\varepsilon - f_{i,j-1}^\varepsilon) \|_{H^m(\mathbf{R}_x^d)} \lesssim \frac{1}{\varepsilon} \| f_{i,j}^\varepsilon - f_{i,j-1}^\varepsilon \|_{H^m(\mathbf{R}_x^d)}.
\end{align}
To estimate the $H^m$ norm of $\mathcal{G}_{i,eq,j}^\varepsilon -  \mathcal{G}_{i,eq,j-1}^\varepsilon$, we first observe that for any $j \in \mathbf{N}$,
\begin{align} \label{Es:GeqL}
\| \Lambda_\varepsilon(\rho_j^\varepsilon) - \Lambda_\varepsilon(\rho_{j-1}^\varepsilon) \|_{H^m(\mathbf{R}_x^d)} + \big\| v_i \cdot \big( \Lambda_\varepsilon(u_j^\varepsilon) - \Lambda_\varepsilon(u_{j-1}^\varepsilon) \big) \big\|_{H^m(\mathbf{R}_x^d)} \lesssim \| f_j^\varepsilon - f_{j-1}^\varepsilon \|_{H_x^m L_{v,w}^2}.
\end{align}
Since the cutoff operator $\Lambda_\varepsilon$ commutes with the differentiation $\partial_x$, the nonlinear terms in $\mathcal{G}_{i,eq,j}^\varepsilon -  \mathcal{G}_{i,eq,j-1}^\varepsilon$ can be estimated by Plancherel's identity and Proposition \ref{Hm:BA}, i.e., for any $j \in \mathbf{N}$ and $1 \leq \alpha, \beta \leq d$, it holds that
\begin{equation} \label{Es:GeqNL}
\begin{split}
&\| \Lambda_\varepsilon\big( \Lambda_\varepsilon(u_j^\varepsilon)_\alpha \Lambda_\varepsilon(u_j^\varepsilon)_\beta - \Lambda_\varepsilon(u_{j-1}^\varepsilon)_\alpha \Lambda_\varepsilon(u_{j-1}^\varepsilon)_\beta \big) \|_{H^m(\mathbf{R}_x^d)} \\
&\ \ \leq \big\| \big( \Lambda_\varepsilon(u_j^\varepsilon)_\alpha - \Lambda_\varepsilon(u_{j-1}^\varepsilon)_\alpha \big)\Lambda_\varepsilon(u_j^\varepsilon)_\beta \big\|_{H^m(\mathbf{R}_x^d)} \\
&\ \ \ \ + \big\| \Lambda_\varepsilon(u_{j-1}^\varepsilon)_\alpha \big( \Lambda_\varepsilon(u_j^\varepsilon)_\beta - \Lambda_\varepsilon(u_{j-1}^\varepsilon)_\beta \big) \big\|_{H^m(\mathbf{R}_x^d)} \\
&\ \ \lesssim \big( \| \Lambda_\varepsilon(u_j^\varepsilon) \|_{H^m(\mathbf{R}_x^d)} + \| \Lambda_\varepsilon(u_{j-1}^\varepsilon) \|_{H^m(\mathbf{R}_x^d)} \big) \| \Lambda_\varepsilon(u_j^\varepsilon) - \Lambda_\varepsilon(u_{j-1}^\varepsilon) \|_{H^m(\mathbf{R}_x^d)} \\
&\ \ \lesssim (\| f_j^\varepsilon \|_{H_x^m L_{v,w}^2} + \| f_{j-1}^\varepsilon \|_{H_x^m L_{v,w}^2} ) \| f_j^\varepsilon - f_{j-1}^\varepsilon \|_{H_x^m L_{v,w}^2}.
\end{split}
\end{equation}
Hence, by substituting estimates (\ref{Es:trans}), (\ref{Es:GeqL}) and (\ref{Es:GeqNL}) back into inequality (\ref{Pica:dif}), we obtain that for any $j \in \mathbf{N}$ and $0 \leq i \leq n$, 
\begin{equation} \label{Pica:Dif}
\begin{split}
\| f_{i,j+1}^\varepsilon - f_{i,j}^\varepsilon \|_{H^m(\mathbf{R}_x^d)} &\lesssim \int_0^t \frac{1+\nu}{\varepsilon^2 \nu} \| f_{i,j}^\varepsilon - f_{i,j-1}^\varepsilon \|_{H^m(\mathbf{R}_x^d)} \, ds \\
&\ \ + \int_0^t \frac{1}{\varepsilon^2 \nu} \big( 1 + \| f_j^\varepsilon \|_{H_x^m L_{v,w}^2} + \| f_{j-1}^\varepsilon \|_{H_x^m L_{v,w}^2} \big) \| f_j^\varepsilon - f_{j-1}^\varepsilon \|_{H_x^m L_{v,w}^2} \, ds,
\end{split}
\end{equation}
as $\varepsilon \in (0,1)$.

Finally, we shall prove by induction that there exists $T>0$ sufficiently small so that simultaneously, it holds for any $j \in \mathbf{N}$ that
\begin{equation} \label{UniB:fj}
\begin{split}
\| f_j^\varepsilon \|_{L_T^\infty H_x^m L_{v,w}^2} &:= \sup_{t \in [0,T]} \| f_j^\varepsilon(t) \|_{H_x^m L_{v,w}^2} \\
&\lesssim \| f_0^\varepsilon \|_{H_x^m L_{v,w}^2} + \frac{2 T (1+\nu)}{\varepsilon^2 \nu} \big( \| f_0^\varepsilon \|_{H_x^m L_{v,w}^2} + \| f_0^\varepsilon \|_{H_x^m L_{v,w}^2}^2 \big) =: M(f_0^\varepsilon,T)
\end{split}
\end{equation}
and for any $j \in \mathbf{N}$ satisfying $j \geq 2$,
\begin{align} \label{Induc:dif}
\| f_j^\varepsilon - f_{j-1}^\varepsilon \|_{L_T^\infty H_x^m L_{v,w}^2} \leq \frac{1}{2} \| f_{j-1}^\varepsilon - f_{j-2}^\varepsilon  \|_{L_T^\infty H_x^m L_{v,w}^2}.
\end{align}
Let $k \in \mathbf{N}$ with $k \geq 2$. 
Suppose that estimates (\ref{UniB:fj}) and (\ref{Induc:dif}) hold simultaneously for all $1 \leq j \leq k$.
Then, by summing up estimate (\ref{Pica:Dif}) over $0 \leq i \leq n$, we deduce that
\begin{align*}
&\| f_{k+1}^\varepsilon - f_k^\varepsilon \|_{L_T^\infty H_x^m L_{v,w}^2} \\
&\ \ \lesssim \frac{(1+\nu) T}{\varepsilon^2 \nu} \| f_k^\varepsilon - f_{k-1}^\varepsilon \|_{L_T^\infty H_x^m L_{v,w}^2} \big( 1 + \| f_{k-1}^\varepsilon \|_{L_T^\infty H_x^m L_{v,w}^2} + \| f_k^\varepsilon \|_{L_T^\infty H_x^m L_{v,w}^2} \big).
\end{align*}
Using assumption (\ref{UniB:fj}) for cases $j = k-1$ and $j=k$, we have that
\begin{align*}
\| f_{k+1}^\varepsilon - f_k^\varepsilon \|_{L_T^\infty H_x^m L_{v,w}^2} \lesssim \frac{(1+\nu) T}{\varepsilon^2 \nu} \big( 1 + 2M(f_0^\varepsilon,T) \big) \| f_k^\varepsilon - f_{k-1}^\varepsilon \|_{L_T^\infty H_x^m L_{v,w}^2}.
\end{align*}
There exists $T_\ast = T_\ast(f_0^\varepsilon, \varepsilon, \nu) > 0$ such that
\begin{align*}
\frac{(1+\nu) T_\ast}{\varepsilon^2 \nu} \big( 1 + M(f_0^\varepsilon, T_\ast) \big) < \frac{1}{2}.
\end{align*}
Hence,
\begin{align} \label{Ind:Dif:k1}
\| f_{k+1}^\varepsilon - f_k^\varepsilon \|_{L_{T_\ast}^\infty H_x^m L_{v,w}^2} \lesssim \frac{1}{2} \| f_k^\varepsilon - f_{k-1}^\varepsilon \|_{L_{T_\ast}^\infty H_x^m L_{v,w}^2}.
\end{align}
By working directly with equation (\ref{IE:fij}) for $j=0$, we can show by analogous derivations as in the above paragraph that the inequality
\[
\| f_1^\varepsilon - f_0^\varepsilon \|_{L_T^\infty H_x^m L_{v,w}^2} \lesssim \frac{(1+\nu) T}{\varepsilon^2 \nu} ( \| f_0^\varepsilon \|_{L_T^\infty H_x^m L_{v,w}^2} + \| f_0^\varepsilon \|_{L_T^\infty H_x^m L_{v,w}^2}^2 )
\]
holds for any $T>0$.
By assumption (\ref{Induc:dif}) and inequality (\ref{Ind:Dif:k1}), we deduce that for all $j$ that takes integer value from $2$ to $k+1$, it holds that
\[
\| f_j^\varepsilon - f_{j-1}^\varepsilon \|_{L_{T_\ast}^\infty H_x^m L_{v,w}^2} \lesssim \frac{1}{2^{j-1}} \| f_1^\varepsilon - f_0^\varepsilon \|_{L_{T_\ast}^\infty H_x^m L_{v,w}^2}.
\]
As a result, 
\begin{align*}
\| f_{k+1}^\varepsilon \|_{L_{T_\ast}^\infty H_x^m L_{v,w}^2} &\lesssim \| f_0^\varepsilon \|_{H_x^m L_{v,w}^2} + \sum_{j=1}^{k+1} \| f_j^\varepsilon - f_{j-1}^\varepsilon \|_{L_{T_\ast}^\infty H_x^m L_{v,w}^2} \\
&\leq \| f_0^\varepsilon \|_{H_x^m L_{v,w}^2} + \| f_1^\varepsilon - f_0^\varepsilon \|_{L_{T_\ast}^\infty H_x^m L_{v,w}^2} \cdot \sum_{j=1}^{k+1} \frac{1}{2^{j-1}}  \\
&\leq \| f_0^\varepsilon \|_{H_x^m L_{v,w}^2} + 2 \| f_1^\varepsilon - f_0^\varepsilon \|_{L_{T_\ast}^\infty H_x^m L_{v,w}^2} \lesssim M(f_0^\varepsilon, T_\ast).
\end{align*}
This completes the proof of the induction. 
By the contraction mapping principle, the sequence $\{ f_j^\varepsilon \}_{j \in \mathbf{N}_0}$ is indeed Cauchy in $L^\infty( [0,T_\ast]; H_x^m L_{v,w}^2)$.
Taking the limit as $j \to \infty$, we obtain a local solution to the approximate LBGK Boltzmann equation (\ref{ALBE}).
\end{proof}
\begin{remark} \label{CtfeqIt}
In the proof of Lemma \ref{Sol:ALBE}, the existence of a local solution $f^\varepsilon$ to the approximate LBGK Boltzmann equation (\ref{ALBE}) is concluded by the Banach fixed point theorem. Hence, for any $t \in [0,T_\ast]$, it holds that 
\[
f^\varepsilon = \Lambda_\varepsilon(f^\varepsilon).
\]
With the help of the local energy estimate (\ref{loEE:ALBE}) and the standard continuous induction argument, we can extend this local solution $f^\varepsilon$ further to a local solution $\widetilde{f^\varepsilon} \in L^\infty( [0,T_0]; H_x^m L_{v,w}^2 )$ satisfying
\[
\widetilde{f^\varepsilon} = \Lambda_\varepsilon( \widetilde{f^\varepsilon} ), \quad \forall \; t \in [0,T_0],
\]
where the existence time $T_0$ is a constant that only depends on the initial data $g^{0, \varepsilon}$ and $m$.
Hence, without loss of generality, we may always assume that $f^\varepsilon$ satisfies $f^\varepsilon = \Lambda_\varepsilon(f^\varepsilon)$ whenever we consider a local solution $f^\varepsilon$ constructed in Lemma \ref{Sol:ALBE}.
\end{remark}

As a direct application of Lemma \ref{Sol:ALBE}, we have the following implication.

\begin{corollary} \label{WCon:ru}
Let $m \in \mathbf{N}$ with $m > d$ and $\{ \varepsilon_n \}_{n \in \mathbf{N}} \subset (0,1)$ be a sequence that converges to zero as $n \to \infty$.
Let $\{ g^{0, {\varepsilon_n}} \}_{n \in \mathbf{N}}$ be a sequence of initial conditions satisfying $g^{0,{\varepsilon_n}} \in H_x^m L_{v,w}^2$ for any $n \in \mathbf{N}$ and
\begin{align*}
M_\ast := \sup_{n \in \mathbf{N}} \| g^{0, {\varepsilon_n}} \|_{H_x^m L_{v,w}^2} < \infty.
\end{align*}
Then, there exists $T_0^\ast = T_0^\ast(M_\ast,m) > 0$ such that for each $n \in \mathbf{N}$, there exists a unique local solution $g^{\varepsilon_n} \in L^\infty( [0,T_0^\ast]; H_x^m L_{v,w}^2)$ to the approximate lattice Boltzmann equation (\ref{ALBE}) with initial condition $\Lambda_\varepsilon(g^{0, {\varepsilon_n}})$ satisfying $\Lambda_\varepsilon(g^{\varepsilon_n}) = g^{\varepsilon_n}$ and the local energy estimate (\ref{loEE:ALBE}).
Let 
\[
\rho^{\varepsilon_n} = \sum_{i=0}^n w_i g_i^{\varepsilon_n}, \quad u^{\varepsilon_n} = \sum_{i=0}^n w_i v_i g_i^{\varepsilon_n}.
\]
It holds that
\begin{align} \label{LEE:ru}
\sup_{n \in \mathbf{N}} \| \rho^{\varepsilon_n} \|_{L^\infty( [0,T_0^\ast]; H^m(\mathbf{R}_x^d) )} + \sup_{n \in \mathbf{N}} \| u^{\varepsilon_n} \|_{L^\infty( [0,T_0^\ast]; H^m(\mathbf{R}_x^d) )} \leq 3 (2 + \nu)^{\frac{1}{2}} M_\ast.
\end{align}
By suppressing subsequences, there exist 
\[
g \in L^\infty( [0,T_0^\ast]; H_x^m L_{v,w}^2 ) \quad \text{and} \quad (\rho, u) \in L^\infty( [0,T_0^\ast]; H^m(\mathbf{R}_x^d) ) 
\]
such that $g^{\varepsilon_n} \to g$ weak-$\ast$ in time $t$, weakly in $H_x^m L_{v,w}^2$ and $(\rho^{\varepsilon_n}, u^{\varepsilon_n}) \to (\rho, u)$ weak-$\ast$ in time $t$, weakly in $H^m(\mathbf{R}_x^d)$ as $n \to \infty$.
In particular, 
\[
\rho = \sum_{i=0}^n w_i g_i, \quad u = \sum_{i=0}^n w_i v_i g_i
\]
and
\begin{align} \label{L2t:Geqg}
\mathcal{G}_{eq}^{\varepsilon_n} - g^{\varepsilon_n} \to 0 \quad \text{in} \quad L^2( [0,T_0^\ast]; H_x^m L_{v,w}^2 ) \quad \text{as} \quad n \to \infty.
\end{align}
\end{corollary}
\begin{proof}
By Lemma \ref{Sol:ALBE}, we can observe that if we choose 
\begin{align} \label{Bdd:T0s}
T_0^\ast \leq \frac{\nu}{2 C_{LE} M_\ast^2}
\end{align}
where $C_{LE}$ is the constant introduced in Lemma \ref{locEE:LBE}, then for any $n \in \mathbf{N}$, there exists a local solution $g^{\varepsilon_n} \in L^\infty( [0,T_0^\ast]; H_x^m L_{v,w}^2 )$ to the approximate lattice Boltzmann equation (\ref{ALBE}) with initial data $\Lambda_\varepsilon(g^{0, {\varepsilon_n}})$ satisfying $\Lambda_\varepsilon(g^{\varepsilon_n}) = g^{\varepsilon_n}$ and the local energy estimate (\ref{loEE:ALBE}) on the time interval $[0,T_0^\ast]$.
By Plancherel's identity, for each $n \in \mathbf{N}$, we have that $\| \Lambda_\varepsilon(g^{0, {\varepsilon_n}}) \|_{H_x^m L_{v,w}^2} \leq \| g^{0, {\varepsilon_n}} \|_{H_x^m L_{v,w}^2}$.
Together with (\ref{Bdd:T0s}), we can deduce from the local energy estimate (\ref{loEE:ALBE}) that
\begin{align} \label{UnBd:gepn}
\sup_{n \in \mathbf{N}, \, t \in [0,T_0^\ast]} \| g^{\varepsilon_n}(t) \|_{H_x^m L_{v,w}^2} \leq (2 + \nu)^{\frac{1}{2}} M_\ast.
\end{align}
Estimate (\ref{LEE:ru}) is a direct consequence of (\ref{UnBd:gepn}).

Since $H^m(\mathbf{R}_x^d)$ is separable, $H^{-m}(\mathbf{R}_x^d)$ is also separable \cite[Theorem 3.26]{Brezis}.
Since $L^1([0,T_0^\ast])$ is separable, $L^1( [0,T_0^\ast]; H^{-m}(\mathbf{R}_x^d) )$ and $L^1( [0,T_0^\ast]; H_x^{-m} L_{v,w}^2 )$ are both separable.
Since the sequences
\begin{align*}
\{ \| g^{\varepsilon_n} \|_{L^\infty( [0,T_0^\ast]; H_x^m L_{v,w}^2 )} \}_{n \in \mathbf{N}}, \quad \{ \| \rho^{\varepsilon_n} \|_{L^\infty( [0,T_0^\ast]; H^m(\mathbf{R}_x^d) )} \}_{n \in \mathbf{N}}, \quad \{ \| u^{\varepsilon_n} \|_{L^\infty( [0,T_0^\ast]; H^m(\mathbf{R}_x^d) )} \}_{n \in \mathbf{N}}
\end{align*}
are all uniformly bounded, by suppressing subsequences, we conclude by \cite[Corollary 3.30]{Brezis} that there exist $g \in H_x^m L_{v,w}^2$ and $(\rho, u) \in H^m(\mathbf{R}_x^d)$ such that $g^{\varepsilon_n} \to g$ weak-$\ast$ in time $t$, weakly in $H_x^m L_{v,w}^2$ and $(\rho^{\varepsilon_n}, u^{\varepsilon_n}) \to (\rho,u)$ weak-$\ast$ in time $t$, weakly in $H^m(\mathbf{R}_x^d)$ as $n \to \infty$.
Finally, the convergence (\ref{L2t:Geqg}) results from the implication of local energy inequality (\ref{loEE:ALBE}) that
\begin{align*}
\frac{1}{\varepsilon^2} \int_0^{T_0^\ast} \| \mathcal{G}_{eq}^{\varepsilon_n}(s) - g^{\varepsilon_n}(s) \|_{H_x^m L_{v,w}^2}^2  \, ds \leq (2 + \nu) \| g_0^{\varepsilon_n} \|_{H_x^m L_{v,w}^2}^2 \leq (2 + \nu) M_\ast^2.
\end{align*}
This completes the proof of Corollary \ref{WCon:ru}.
\end{proof}
\begin{remark} \label{Bdd:rhou}
Following the assumption made in Corollary \ref{WCon:ru}, since 
\begin{align*}
\rho^{\varepsilon_n} \overset{\ast}{\rightharpoonup} \rho \quad \text{and} \quad u^{\varepsilon_n} \overset{\ast}{\rightharpoonup} u \quad \text{in} \quad \sigma\big( L^\infty( [0,T_0]; H^m(\mathbf{R}_x^d) ), L^1( [0,T_0]; H^{-m}(\mathbf{R}_x^d) ) \big),
\end{align*}
we have by the uniform estimate (\ref{LEE:ru}) that
\begin{align*}
&\| \rho \|_{L^\infty( [0,T_0]; H^m(\mathbf{R}_x^d) )} \leq \underset{n \to \infty}{\operatorname{liminf}} \; \| \rho^{\varepsilon_n} \|_{L^\infty( [0,T_0]; H^m(\mathbf{R}_x^d) )} \lesssim (2+\nu)^{\frac{1}{2}} M_\ast, \\
&\| u \|_{L^\infty( [0,T_0]; H^m(\mathbf{R}_x^d) )} \leq \underset{n \to \infty}{\operatorname{liminf}} \; \| u^{\varepsilon_n} \|_{L^\infty( [0,T_0]; H^m(\mathbf{R}_x^d) )} \lesssim (2+\nu)^{\frac{1}{2}} M_\ast,
\end{align*}
see e.g. \cite[Proposition 3.13]{Brezis}.
\end{remark}

\subsection{Local well-posedness for the approximate equation}
\label{Sub:LocWell}

The uniqueness and continuous dependence on initial data for the local solution can be proved by deriving the energy inequality for the difference between two local solutions for the approximate equation (\ref{ALBE}) that are well-defined on the same time interval.
\begin{proof}[Proof of Lemma \ref{Sol:ALBE} (Uniqueness and continuity with respect to initial data)]
Let $\varepsilon > 0$ be fixed.
Suppose that $f^\varepsilon, g^\varepsilon \in L^\infty( [0,T_0]; H_x^m L_{v,w}^2 )$ are two local solutions for the approximate equation (\ref{ALBE}) on the same isotropic lattice $(\mathcal{V}, w)$ satisfying 
\begin{equation*}
\left\{
 \begin{aligned}
&f^\varepsilon = \Lambda_\varepsilon(f^\varepsilon),& &f^\varepsilon \bigm|_{t=0} = \zeta^{0, \varepsilon}, \\
&g^\varepsilon = \Lambda_\varepsilon(g^\varepsilon),& &g^\varepsilon \bigm|_{t=0} = \eta^{0,\varepsilon}
 \end{aligned}
\right.
\end{equation*}
and the local energy inequality (\ref{loEE:ALBE}) respectively.
Their difference $h^\varepsilon := f^\varepsilon - g^\varepsilon$ satisfies the vector equation
\begin{equation} \label{Dif:ALBE}
\varepsilon \partial_t h^\varepsilon + v \cdot \nabla_x h^\varepsilon = \frac{1}{\varepsilon \nu} (\mathcal{H}_{eq}^\varepsilon - h^\varepsilon), \quad h^\varepsilon \bigm|_{t=0} = \zeta^{0,\varepsilon} - \eta^{0,\varepsilon},
\end{equation}
where for any $0 \leq i \leq n$,
\begin{equation*} 
\mathcal{H}^\varepsilon_{i,eq} := \rho_h^\varepsilon + \frac{v_i \cdot u_h^\varepsilon}{c_s^2} + \frac{\varepsilon}{2 c_s^4} \sum_{\alpha,\beta = 1}^d \Lambda_\varepsilon(u^\varepsilon_{f,\alpha} u^\varepsilon_{f,\beta} - u^\varepsilon_{g,\alpha} u^\varepsilon_{g,\beta}) ( v_{i,\alpha}v_{i,\beta} - c_s^2 \delta_{\alpha \beta} )
\end{equation*}
with notations
\begin{equation*} 
\rho_z^\varepsilon := \sum_{i=0}^n w_i z_i^\varepsilon, \quad u_z^\varepsilon := \sum_{i=0}^n w_i v_i z_i^\varepsilon \quad \text{for} \quad z^\varepsilon \in L^\infty( [0,T_0]; H_x^m L_{v,w}^2 )
\end{equation*}
and $u_{z,\alpha}^\varepsilon$ denotes the $\alpha$-th component of $u_z^\varepsilon$ for any $1 \leq \alpha \leq d$.
Then, by applying $\partial_x^\tau$ with $|\tau|_\mathrm{s} \leq m$ to the difference equation (\ref{Dif:ALBE}) and taking its inner product with $\partial_x^\tau h^\varepsilon$ in the $L_x^2 L_{v,w}^2$ sense, we obtain that
\begin{equation*} 
\begin{split}
&\frac{1}{2} \sum_{i=0}^n w_i \frac{\mathrm{d}}{\mathrm{d} t} \| \partial_x^\tau h_i^\varepsilon \|_{L^2(\mathbf{R}_x^d)}^2 +  \frac{1}{\varepsilon^2 \nu} \sum_{i=0}^n w_i \| \partial_x^\tau \mathcal{H}_{i,eq}^\varepsilon - \partial_x^\tau h_i^\varepsilon \|_{L^2(\mathbf{R}_x^d)}^2 \\
&\ \ = \frac{1}{\varepsilon^2 \nu} \sum_{i=0}^n w_i \int_{\mathbb{R}^d} (\partial_x^\tau \mathcal{H}_{i,eq}^\varepsilon - \partial_x^\tau h_i^\varepsilon) \partial_x^\tau \mathcal{H}_{i,eq}^\varepsilon \, dx.
\end{split}
\end{equation*}
Analogously, in this case, the isotropy of $(\mathcal{V}, w)$  also guarantees that
\[
\sum_{i=0}^n w_i (\mathcal{H}_{i,eq}^\varepsilon - h_i^\varepsilon) = 0, \quad \sum_{i=0}^n w_i v_i (\mathcal{H}_{i,eq}^\varepsilon - h_i^\varepsilon) = 0.
\]
Hence, considering Plancherel's identity, we deduce that
\begin{equation} \label{EnIne:h2}
\begin{split}
&\frac{1}{\varepsilon^2 \nu} \sum_{i=0}^n w_i \int_{\mathbb{R}^d} (\partial_x^\tau \mathcal{H}_{i,eq}^\varepsilon - \partial_x^\tau h_i^\varepsilon) \partial_x^\tau \mathcal{H}_{i,eq}^\varepsilon \, dx \\
&\ \ = \frac{1}{2 c_s^4 \varepsilon \nu} \sum_{i=0}^n \sum_{\alpha, \beta=1}^d \sum_{\sigma \leq \tau} w_i ( v_{i,\alpha} v_{i,\beta} - c_s^2 \delta_{\alpha \beta} ) \\
&\ \ \ \ \times \int_{\mathbb{R}^d} (\partial_x^\tau \mathcal{H}_{i,eq}^\varepsilon - \partial_x^\tau h_i^\varepsilon) \Lambda_\varepsilon(\partial_x^\sigma u_{h,\alpha}^\varepsilon \partial_x^{\tau - \sigma} u_{f,\beta}^\varepsilon - \partial_x^\sigma u_{g,\alpha}^\varepsilon \partial_x^{\tau - \sigma} u_{h,\beta}^\varepsilon) \, dx \\
&\ \ \leq \frac{\ell_{\mathcal{V}, c_s}}{2 c_s^4 \varepsilon \nu} \sum_{i=0}^n w_i \| \partial_x^\tau \mathcal{H}_{i,eq}^\varepsilon - \partial_x^\tau h_i^\varepsilon \|_{L^2(\mathbf{R}_x^d)} \\ 
&\ \ \ \ \times \sum_{\alpha, \beta=1}^d  \sum_{\sigma \leq \tau} ( \| \partial_x^\sigma u_{h,\alpha}^\varepsilon \partial_x^{\tau - \sigma} u_{f,\beta}^\varepsilon \|_{L^2(\mathbf{R}_x^d)} + \| \partial_x^\sigma u_{g,\alpha}^\varepsilon \partial_x^{\tau - \sigma} u_{h,\beta}^\varepsilon \|_{L^2(\mathbf{R}_x^d)} ) \\
&\ \ \leq \frac{1}{2 \varepsilon^2 \nu} \sum_{i=0}^n w_i \| \partial_x^\tau \mathcal{H}_{i,eq}^\varepsilon - \partial_x^\tau h_i^\varepsilon \|_{L^2(\mathbf{R}_x^d)}^2 \\
&\ \ \ \ + \frac{\ell_{\mathcal{V}, c_s}^2 d^2 | \tau |_\mathrm{s}^d}{4 c_s^8 \nu} \sum_{\alpha, \beta=1}^d \sum_{\sigma \leq \tau} ( \| \partial_x^\sigma u_{h,\alpha}^\varepsilon \partial_x^{\tau - \sigma} u_{f,\beta}^\varepsilon \|_{L^2(\mathbf{R}_x^d)}^2 + \| \partial_x^\sigma u_{g,\alpha}^\varepsilon \partial_x^{\tau - \sigma} u_{h,\beta}^\varepsilon \|_{L^2(\mathbf{R}_x^d)}^2),
\end{split}
\end{equation}
where $\ell_{\mathcal{V}, c_s} := \Big( \underset{1 \leq i \leq n}{\mathrm{max}} \, | v_i | \Big)^2 + c_s^2$.
By using Proposition \ref{Hm:BA} to further estimate the right hand side of (\ref{EnIne:h2}), we obtain that
\begin{equation} \label{EnIne:h3}
\begin{split}
&\frac{1}{2} \frac{\mathrm{d}}{\mathrm{d} t} \sum_{i=0}^n w_i \| \partial_x^\tau h_i^\varepsilon \|_{L^2(\mathbf{R}_x^d)}^2 +  \frac{1}{2 \varepsilon^2 \nu} \sum_{i=0}^n w_i \| \partial_x^\tau \mathcal{H}_{i,eq}^\varepsilon - \partial_x^\tau h_i^\varepsilon \|_{L^2(\mathbf{R}_x^d)}^2 \\
&\ \ \leq \frac{\ell_{\mathcal{V}, c_s}^2 d^4 | \tau |_\mathrm{s}^d}{4 c_s^8 \nu} (\| h^\varepsilon \|_{H_x^m L_{v,w}^2}^2 \| f^\varepsilon \|_{H_x^{| \tau |_\mathrm{s}} L_{v,w}^2}^2 + \| g^\varepsilon \|_{H_x^m L_{v,w}^2}^2 \| h^\varepsilon \|_{H_x^{| \tau |_\mathrm{s}} L_{v,w}^2}^2 ).
\end{split}
\end{equation}
Summing up inequality (\ref{EnIne:h3}) over all $\tau \in \mathbf{N}_0^d$ satisfying $|\tau|_\mathrm{s} \leq m$ gives us
\begin{equation} \label{EnIne:h4}
\begin{split}
&\frac{1}{2} \frac{\mathrm{d}}{\mathrm{d} t} \| h^\varepsilon \|_{H_x^m L_{v,w}^2}^2 + \frac{1}{2 \varepsilon^2 \nu} \| \mathcal{H}_{eq}^\varepsilon - h^\varepsilon \|_{H_x^m L_{v,w}^2}^2 \\
&\ \ \leq \frac{\ell_{\mathcal{V}, c_s}^2 d^4 m^d}{4 c_s^8 \nu} \| h^\varepsilon \|_{H_x^m L_{v,w}^2}^2 ( \| f^\varepsilon \|_{H_x^m L_{v,w}^2}^2 + \| g^\varepsilon \|_{H_x^m L_{v,w}^2}^2 ).
\end{split}
\end{equation}
Without loss of generality, we may assume that $T_0$, which depends on sizes of $\zeta^{0, \varepsilon}$ and $\eta^{0,\varepsilon}$, is chosen to be sufficiently small so that simultaneously
\begin{equation*}
\left\{
 \begin{aligned}
 \sup_{t \in [0,T_0]} \| f^\varepsilon(t) \|_{H_x^m L_{v,w}^2}^2 \leq (\nu + 2) \| \zeta^{0, \varepsilon} \|_{H_x^m L_{v,w}^2}^2, \\
 \sup_{t \in [0,T_0]} \| g^\varepsilon(t) \|_{H_x^m L_{v,w}^2}^2 \leq (\nu + 2) \| \eta^{0, \varepsilon} \|_{H_x^m L_{v,w}^2}^2.
 \end{aligned}
\right.
\end{equation*}
Applying the standard Gronwall's inequality to (\ref{EnIne:h4}), we deduce that the estimate
\begin{align*}
\| h^\varepsilon(t) \|_{H_x^m L_{v,w}^2}^2 \leq \| \zeta^{0, \varepsilon} - \eta^{0, \varepsilon} \|_{H_x^m L_{v,w}^2}^2 \mathrm{e}^{t c_1(\zeta^{0, \varepsilon}, \eta^{0, \varepsilon})} 
\end{align*}
holds for any $t \in [0,T_0]$ where
\begin{align*}
c_1(\zeta^{0, \varepsilon}, \eta^{0, \varepsilon}) := c_s^{- 8} (2^{-1} + \nu^{-1}) \ell_{\mathcal{V}, c_s} d^4 m^d \big( \| \zeta^{0, \varepsilon} \|_{H_x^m L_{v,w}^2}^2 + \| \eta^{0, \varepsilon} \|_{H_x^m L_{v,w}^2}^2 \big).
\end{align*}
Combining with the existence of the local solution, the local well-posedness of equation (\ref{ALBE}) can thus be concluded.
\end{proof}

\section{Formal derivation of the incompressible Navier-Stokes equations}
\label{Sec:DerNS}

Suppose that $(\mathcal{V}, w)$ is an isotropic lattice associated with speed of sound $c_s$.
Let $\varepsilon>0$ be fixed, $m \in \mathbf{N}$ with $m > d$ and $g_0 \in H_x^m L_{v,w}^2$.
By Lemma \ref{Sol:ALBE} and Corollary \ref{WCon:ru}, there exists $T_0 = T_0(g_0, \mathcal{V}, c_s, m, d) > 0$ sufficiently small such that the approximate equation (\ref{ALBE}) admits a unique local solution $g^\varepsilon \in L^\infty( [0,T_0]; H_x^m L_{v,w}^2 )$ satisfying $g^\varepsilon = \Lambda_\varepsilon(g^\varepsilon)$, $g^\varepsilon \bigm|_{t=0} = \Lambda_\varepsilon(g_0)$ and the local energy inequality
\begin{align} \label{loEI:Rg0}
\sup_{t \in [0,T_0]} \| g^\varepsilon(t) \|_{H_x^m L_{v,w}^2}^2 + \frac{1}{\varepsilon^2} \int_0^{T_0} \| \mathcal{G}_{eq}^\varepsilon(t) - g^\varepsilon(t) \|_{H_x^m L_{v,w}^2}^2 \, dt \leq (2 + \nu) \| g_0 \|_{H_x^m L_{v,w}^2}^2.
\end{align}
Due to (\ref{Alt:Crhu}), by taking the inner product of the approximate equation (\ref{ALBE}) with $1$ and $v$ in the $L_{v,w}^2$ sense, we obtain that
\begin{equation} \label{Inp:LBE1v}
\left\{
 \begin{aligned}
 \sum_{i=0}^n \varepsilon w_i \partial_{t} g_i^\varepsilon + \sum_{i=0}^n w_i v_i  \cdot \nabla_x g_i^\varepsilon &=& &0&   &=& &\varepsilon \partial_t \rho^\varepsilon + \operatorname{div} u^\varepsilon,& \\
 \sum_{i=0}^n \varepsilon w_i v_i \partial_t g_i^\varepsilon + \sum_{i=0}^n w_i v_i (v_i \cdot \nabla_x g_i^\varepsilon) &=& &0& &=& &\varepsilon \partial_t u^\varepsilon + \sum_{i=0}^n w_i \nabla_x \cdot (v_i \otimes v_i g_i^\varepsilon).&
 \end{aligned}
\right.
\end{equation}
For each $0 \leq i \leq n$, we define the matrix $A_i$ by 
\begin{align*} 
A_i = v_i \otimes v_i - c_s^2 I
\end{align*}
with $I$ denoting the identity matrix.
By making use of the matrix $A_i$, for each $i$ we have that 
\begin{align*} 
\nabla_x \cdot (v_i \otimes v_i g_i^\varepsilon) = \nabla_x \cdot (A_i g_i^\varepsilon) + c_s^2 \nabla_x g_i^\varepsilon.
\end{align*}
Furthermore, we decompose
\begin{align*} 
\nabla_x \cdot (A_i g_i^\varepsilon) = \nabla_x \cdot \big( A_i (g_i^\varepsilon - \mathcal{G}^\varepsilon_{i,eq}) \big) + \nabla_x \cdot (A_i \mathcal{G}^\varepsilon_{i,eq}).
\end{align*}
%

\subsection{Derivation of the transport term $u^\varepsilon \cdot \nabla_x u^\varepsilon$}
\label{Sub:DeTrans}

In this section, we show that the transport term $u^\varepsilon \cdot \nabla_x u^\varepsilon$ can be obtained by rewriting the summation of $\varepsilon^{-1} w_i \nabla_x \cdot (A_i \mathcal{G}_{i,eq}^\varepsilon)$ in $i$.

\begin{lemma} \label{Rewr:Trans} 
It holds that
\begin{align*}
\sum_{i=0}^n \frac{1}{\varepsilon} w_i \nabla_x \cdot (A_i \mathcal{G}^\varepsilon_{i,eq}) = \nabla_x \cdot (u^\varepsilon \otimes u^\varepsilon) + \mathcal{R}_T^\varepsilon(x,t)
\end{align*}
where $\mathcal{R}_T^\varepsilon(x,t)$ is a remainder vector depending on $x$ and $t$ which converges to zero in the sense of distributions, i.e., for any $\Phi \in C_\mathrm{c}^\infty(\mathbf{R}^d \times [0,T_0])$,
\begin{align*}
\left| \int_0^{T_0} \int_{\mathbf{R}^d} \mathcal{R}_T^\varepsilon(x,t) \Phi(x,t) \, dx \, dt \right| \to 0 \quad \text{as} \quad \varepsilon \to 0.
\end{align*}
More explicitly, the remainder vector $\mathcal{R}_T^\varepsilon(x,t)$ is defined by expressions (\ref{Rema:Tran:c}) and (\ref{Vec:RTep}) within the proof of this lemma.
\end{lemma}
\begin{proof}
Substituting expression (\ref{Ctf:gieq}) for $\mathcal{G}_{i,eq}^\varepsilon$ directly, we have that
\begin{equation} \label{SubG:Tra}
\begin{split}
\sum_{i=0}^n \frac{1}{\varepsilon} w_i \nabla_x \cdot (A_i \mathcal{G}^\varepsilon_{i,eq}) &= \sum_{i=0}^n \frac{1}{\varepsilon} \nabla_x \cdot \Big( w_i A_i \rho^\varepsilon + w_i A_i \frac{v_i \cdot u^\varepsilon}{c_s^2} \Big) \\
&\ \ + \frac{1}{2 c_s^4} \sum_{i=0}^n \sum_{\alpha, \beta=1}^d \nabla_x \cdot \big( w_i A_i (v_{i,\alpha} v_{i,\beta} - c_s^2 \delta_{\alpha \beta}) \Lambda_\varepsilon(u_\alpha^\varepsilon u_\beta^\varepsilon) \big).
\end{split}
\end{equation} 
By using the second to the fourth summation condition in (\ref{Sum:wv}), we observe that for any $1 \leq \alpha, \beta \leq d$,
\begin{align*}
\sum_{i=0}^n w_i A_{i,\alpha,\beta} = \sum_{i=0}^n w_i ( v_{i,\alpha} v_{i,\beta} - c_s^2 \delta_{\alpha \beta}) = \sum_{i=0}^n w_i v_{i,\alpha} v_{i,\beta} - c_s^2 \delta_{\alpha \beta} = 0
\end{align*}
and
\begin{align*}
\sum_{i=0}^n \sum_{\gamma=1}^d w_i A_{i,\alpha,\beta} v_{i,\gamma} u_\gamma^\varepsilon &= \sum_{i=0}^n \sum_{\gamma=1}^d w_i (v_{i,\alpha} v_{i,\beta} - c_s^2 \delta_{\alpha \beta}) v_{i,\gamma} u_\gamma^\varepsilon \\
&= \sum_{\gamma=1}^d u_\gamma^\varepsilon \left( \sum_{i=0}^n w_i v_{i,\alpha} v_{i,\beta} v_{i,\gamma} \right) - c_s^2 \delta_{\alpha \beta} \sum_{\gamma=1}^d u_\gamma^\varepsilon \left( \sum_{i=0}^n w_i v_{i,\gamma} \right) = 0,
\end{align*}
where the notation $A_{i,\alpha,\beta}$ represents the $(\alpha,\beta)$-entry of the matrix $A_i$.
Hence, it holds that
\begin{equation} \label{Sum:wAru}
\sum_{i=0}^n \frac{1}{\varepsilon} \nabla_x \cdot \Big( w_i A_i \rho^\varepsilon + w_i A_i \frac{v_i \cdot u^\varepsilon}{c_s^2} \Big) = 0.
\end{equation}
Moreover, by using the third and the fifth summation condition in (\ref{Sum:wv}), we derive that for any $1 \leq \gamma, \zeta \leq d$,
\begin{equation} \label{Sum:w2A}
\begin{split}
\sum_{i=0}^n w_i A_{i,\gamma,\zeta} (v_{i,\alpha} v_{i,\beta} - c_s^2 \delta_{\alpha \beta}) &= \sum_{i=0}^n w_i (v_{i,\gamma} v_{i,\zeta} - c_s^2 \delta_{\gamma \zeta}) (v_{i,\alpha} v_{i,\beta} - c_s^2 \delta_{\alpha \beta}) \\
&= \sum_{i=0}^n w_i v_{i,\alpha} v_{i,\beta} v_{i,\gamma} v_{i,\zeta} - c_s^4 \delta_{\alpha \beta} \delta_{\gamma \zeta} = c_s^4 \delta_{\alpha \gamma} \delta_{\beta \zeta} + c_s^4 \delta_{\alpha \zeta} \delta_{\beta \gamma}.
\end{split}
\end{equation} 
By substituting (\ref{Sum:wAru}) and (\ref{Sum:w2A}) back into the right hand side of (\ref{SubG:Tra}), for any $1 \leq \gamma \leq d$, we obtain that
\begin{align*}
\left\{ \sum_{i=0}^n \frac{1}{\varepsilon} w_i \nabla_x \cdot (A_i \mathcal{G}^\varepsilon_{i,eq}) \right\}_\gamma &= \frac{1}{2} \sum_{\alpha, \beta, \zeta=1}^d (\delta_{\alpha \gamma} \delta_{\beta \zeta} + \delta_{\alpha \zeta} \delta_{\beta \gamma}) \partial_{x_\zeta} \Lambda_\varepsilon(u_\alpha^\varepsilon u_\beta^\varepsilon) \\
&= \frac{1}{2} \sum_{\beta, \zeta = 1}^d \big( \delta_{\beta \zeta} \partial_{x_\zeta} \Lambda_\varepsilon(u_\gamma^\varepsilon u_\beta^\varepsilon) + \delta_{\beta \gamma} \partial_{x_\zeta} \Lambda_\varepsilon(u_\zeta^\varepsilon u_\beta^\varepsilon) \big) \\
&= \sum_{\zeta=1}^d \partial_{x_\zeta} \Lambda_\varepsilon(u_\zeta^\varepsilon u_\gamma^\varepsilon) = \big\{ \nabla_x \cdot (u^\varepsilon \otimes u^\varepsilon) \big\}_\gamma + \mathcal{R}_T^\varepsilon(x,t)_\gamma
\end{align*}
where $\mathcal{R}_T^\varepsilon(x,t)_\gamma$ is the remainder function
\begin{align} \label{Rema:Tran:c}
\mathcal{R}_T^\varepsilon(x,t)_\gamma := \sum_{\zeta=1}^d \partial_{x_\zeta} \big( \Lambda_\varepsilon(u_\zeta^\varepsilon u_\gamma^\varepsilon) - u_\zeta^\varepsilon u_\gamma^\varepsilon \big),
\end{align}
which converges to zero in the sense of distributions.
Indeed, for any $1 \leq \gamma \leq d$ and $\Phi(x,t) \in C_\mathrm{c}^\infty(\mathbf{R}^d \times [0,T_0])$, we observe by Remark \ref{Ctf:AdjP} that
\begin{align*}
\int_0^{T_0} \int_{\mathbf{R}^d} \mathcal{R}_T^\varepsilon(x,t)_\gamma \Phi(x,t) \, dx \, dt &= - \int_0^{T_0} \int_{\mathbf{R}^d} \big( \Lambda_\varepsilon(u_\gamma^\varepsilon u^\varepsilon) - u_\gamma^\varepsilon u^\varepsilon \big) \cdot \nabla_x \Phi \, dx \, dt  \\
&= - \int_0^{T_0} \int_{\mathbf{R}^d} u_\gamma^\varepsilon u^\varepsilon \cdot \big( \Lambda_\varepsilon(\nabla_x \Phi) - \nabla_x \Phi \big) \, dx \, dt.
\end{align*}
By the Plancheral's identity, we deduce that
\begin{equation*} 
\begin{split}
\| \Lambda_\varepsilon(\nabla_x \Phi) - \nabla_x \Phi \|_{H^{-m}(\mathbf{R}_x^2)}^2 &= 4 \pi^2 \int_{|\xi| \geq \frac{1}{\varepsilon}} (1 + |\xi|^2)^{-m} |\xi|^2 \big| \widehat{\Phi}(\xi,t) \big|^2 \, d\xi \\
&\leq 4 \pi^2 \int_{|\xi| \geq \frac{1}{\varepsilon}} \big| \widehat{\Phi}(\xi,t) \big|^2 \, d\xi \to 0
\end{split}
\end{equation*}
as $\varepsilon \to 0$ for any $t \in [0,T_0]$.
Using Proposition \ref{Hm:BA} and inequality (\ref{LEE:ru}) to control $u_\gamma^\varepsilon u^\varepsilon$, we deduce by the dominated convergence theorem that
\begin{align*}
&\left| \int_0^{T_0} \int_{\mathbf{R}^d} u_\gamma^\varepsilon u^\varepsilon \cdot \big( \Lambda_\varepsilon(\nabla_x \Phi) - \nabla_x \Phi \big) \, dx \, dt \right| \\
&\ \ \leq \int_0^{T_0} \| u_\gamma^\varepsilon u^\varepsilon \|_{H^m(\mathbf{R}_x^d)} \| \Lambda_\varepsilon(\nabla_x \Phi) - \nabla_x \Phi \|_{H^{-m}(\mathbf{R}_x^d)} \, dt \\
&\ \ \lesssim \| g_0 \|_{L^\infty( [0,T_0]; H_x^m L_{v,w}^2 )}^2 \int_0^{T_0} \int_{|\xi| \geq \frac{1}{\varepsilon}} \big| \widehat{\Phi}(\xi,t) \big|^2 \, d\xi \, dt \to 0 \quad \text{as} \quad \varepsilon \to 0.
\end{align*}
Finally, we set the vector field
\begin{align} \label{Vec:RTep}
\mathcal{R}_T^\varepsilon(x,t) := \big( \mathcal{R}_T^\varepsilon(x,t)_1, \mathcal{R}_T^\varepsilon(x,t)_2 \big).
\end{align}
The proof of Lemma \ref{Rewr:Trans} is then completed.
\end{proof}

\subsection{Derivation of the diffusion term $\Delta_x u^\varepsilon$}
\label{Sub:DeDiff}

In this section, we show that the diffusion term $\Delta_x u^\varepsilon$ can be obtained by rewriting the summation of $\varepsilon^{-1} w_i \nabla_x \cdot \big( A_i (g_i^\varepsilon - \mathcal{G}_{i,eq}^\varepsilon) \big)$ in $i$.
For simplicity of notation, in the following we shall denote $g^\varepsilon - \mathcal{G}_{eq}^\varepsilon$ by $\mathcal{F}^\varepsilon$ and $g_i^\varepsilon - \mathcal{G}_{i,eq}^\varepsilon$ by $\mathcal{F}_i^\varepsilon$ for each $0 \leq i \leq n$.

\begin{lemma} \label{Rewr:Diff} 
It holds that
\begin{align*}
\sum_{i=0}^n \frac{1}{\varepsilon} w_i \nabla_x \cdot ( A_i \mathcal{F}_i^\varepsilon ) = - c_s^2 \nu \Delta_x u^\varepsilon - c_s^2 \nu \nabla_x \operatorname{div} u^\varepsilon + \mathcal{R}_D^\varepsilon(x,t),
\end{align*}
where $\mathcal{R}_D^\varepsilon(x,t)$ is a remainder vector depending on $x$ and $t$ which converges to zero in the sense of distributions, i.e., for any $\Phi \in C_\mathrm{c}^\infty(\mathbf{R}^d \times [0,T_0])$, 
\begin{align*}
\left| \int_0^{T_0} \int_{\mathbf{R}^d} \mathcal{R}_D^\varepsilon(x,t) \Phi(x,t) \, dx \, dt \right| \to 0 \quad \text{as} \quad \varepsilon \to 0.
\end{align*}
More explicitly, the remainder vector $\mathcal{R}_D^\varepsilon(x,t)$ is defined by expression (\ref{Remain:Diff}) within the proof of this lemma.
\end{lemma}
\begin{proof}
By rearranging the approximate equation (\ref{ALBE}), we note that
\begin{equation} \label{gi:mi:Gieq}
\mathcal{F}_i^\varepsilon = - \varepsilon^2 \nu \partial_t g^\varepsilon_i - \varepsilon \nu v_i \cdot \nabla_x g^\varepsilon_i.
\end{equation} 
for any $0 \leq i \leq n$.
By substituting expression (\ref{gi:mi:Gieq}) for $\mathcal{F}_i^\varepsilon$ for each $i$, we rewrite
\begin{equation*}
\sum_{i=0}^n \frac{1}{\varepsilon} w_i \nabla_x \cdot ( A_i \mathcal{F}_i^\varepsilon ) = - \varepsilon \nu \sum_{i=0}^n w_i \nabla_x \cdot (A_i \partial_t g_i^\varepsilon) - \nu \sum_{i=0}^n w_i \nabla_x \cdot \big( A_i (v_i \cdot \nabla_x g_i^\varepsilon) \big). 
\end{equation*} 
Furthermore, we rewrite again
\begin{align*}
\sum_{i=0}^n w_i \nabla_x \cdot \big( A_i (v_i \cdot \nabla_x g_i^\varepsilon) \big) &= \sum_{i=0}^n w_i \nabla_x \cdot \big( A_i (v_i \cdot \nabla_x \mathcal{F}_i^\varepsilon) \big) + \sum_{i=0}^n w_i \nabla_x \cdot \big( A_i (v_i \cdot \nabla_x \mathcal{G}_{i,eq}^\varepsilon) \big).
\end{align*}
What we are about to show is that the diffusion term $\Delta_x u^\varepsilon$ comes from 
\begin{align} \label{Ori:Diff}
\sum_{i=0}^n w_i \nabla_x \cdot \big( A_i (v_i \cdot \nabla_x \mathcal{G}_{i,eq}^\varepsilon) \big)
\end{align}
whereas
\[
\varepsilon \sum_{i=0}^n w_i \nabla_x \cdot (A_i \partial_t g_i^\varepsilon) + \sum_{i=0}^n w_i \nabla_x \cdot \big( A_i (v_i \cdot \nabla_x \mathcal{F}_i^\varepsilon) \big) \to 0
\]
in the sense of distributions as $\varepsilon \to 0$.

By substituting expression (\ref{Ctf:gieq}) for $\mathcal{G}_{i,eq}^\varepsilon$ into (\ref{Ori:Diff}), we observe that
\begin{align*} 
\sum_{i=0}^n w_i \nabla_x \cdot \big( A_i (v_i \cdot \nabla_x \mathcal{G}_{i,eq}^\varepsilon) \big) &= \sum_{i=0}^n w_i \nabla_x \cdot \big( A_i (v_i \cdot \nabla_x \rho^\varepsilon) \big) + \frac{1}{c_s^2} \sum_{i=0}^n w_i \nabla_x \cdot \Big( A_i \big( v_i \cdot \nabla_x (v_i \cdot u^\varepsilon) \big) \Big) \\
&\ \ + \frac{\varepsilon}{2 c_s^4} \sum_{i=0}^n w_i \nabla_x \cdot \Big( \sum_{\alpha, \beta=1}^d A_{i,\alpha,\beta} A_i \big( v_i \cdot \nabla_x \Lambda_\varepsilon(u_\alpha^\varepsilon u_\beta^\varepsilon) \big) \Big).
\end{align*}
For $1 \leq \beta \leq d$, we can deduce from the second and the fourth summation condition in (\ref{Sum:wv}) that
\begin{align*}
\left\{ \sum_{i=0}^n w_i \nabla_x \cdot \big( A_i (v_i \cdot \nabla_x \rho^\varepsilon) \big) \right\}_\beta &= \sum_{i=0}^n \sum_{\alpha, \gamma=1}^d w_i (v_{i,\alpha} v_{i,\beta} - c_s^2 \delta_{\alpha \beta}) v_{i,\gamma} \partial_{x_\alpha} \partial_{x_\gamma} \rho^\varepsilon \\
&= \sum_{\alpha, \gamma=1}^d \partial_{x_\alpha} \partial_{x_\gamma} \rho^\varepsilon \left( \sum_{i=0}^n  w_i v_{i,\alpha} v_{i,\beta} v_{i,\gamma} - c_s^2 \delta_{\alpha \beta} \sum_{i=0}^n w_i v_{i,\gamma} \right) = 0.
\end{align*}
For $0 \leq i \leq n$, we observe that
\begin{align*}
h_i^\varepsilon := v_i \cdot \nabla_x (v_i \cdot u^\varepsilon) = \sum_{\alpha=1}^d v_{i,\alpha} \partial_{x_\alpha} (v_i \cdot u^\varepsilon) = \sum_{\alpha, \beta=1}^d v_{i,\alpha} v_{i,\beta} \partial_{x_\alpha} u_\beta^\varepsilon.
\end{align*}
Thus, for $1 \leq j \leq d$, we have that
\begin{equation} \label{j:Sw:divAh}
\begin{split}
&\left\{ \sum_{i=0}^n w_i \nabla_x \cdot (A_i h_i^\varepsilon) \right\}_j = \sum_{i=0}^n \sum_{k=1}^d w_i A_{i,k,j} \partial_{x_k} h_i^\varepsilon = \sum_{i=0}^n \sum_{k,\alpha,\beta=1}^d w_i A_{i,k,j} v_{i,\alpha} v_{i,\beta} \partial_{x_k} \partial_{x_\alpha} u_\beta^\varepsilon \\
&\ \ = \sum_{i=0}^n \sum_{k,\alpha,\beta=1}^d w_i v_{i,k} v_{i,j} v_{i,\alpha} v_{i,\beta} \partial_{x_k} \partial_{x_\alpha} u_\beta^\varepsilon - c_s^2 \sum_{i=0}^n \sum_{k,\alpha,\beta=1}^d w_i \delta_{kj} v_{i,\alpha} v_{i,\beta} \partial_{x_k} \partial_{x_\alpha} u_\beta^\varepsilon.
\end{split}
\end{equation}
Again, by the third summation condition in (\ref{Sum:wv}), we deduce that
\begin{equation} \label{Swd:v2:p2u}
\begin{split}
\sum_{i=0}^n \sum_{k,\alpha,\beta=1}^d w_i \delta_{kj} v_{i,\alpha} v_{i,\beta} \partial_{x_k} \partial_{x_\alpha} u_\beta^\varepsilon &= \sum_{i=0}^n \sum_{\alpha,\beta=1}^d w_i v_{i,\alpha} v_{i,\beta} \partial_{x_j} \partial_{x_\alpha} u_\beta^\varepsilon \\
&= \sum_{\alpha,\beta=1}^d \left( \sum_{i=0}^n w_i v_{i,\alpha} v_{i,\beta} \right) \partial_{x_j} \partial_{x_\alpha} u_\beta^\varepsilon = c_s^2 \sum_{\alpha=1}^d \partial_{x_j} \partial_{x_\alpha} u_\alpha^\varepsilon.
\end{split}
\end{equation}
On the other hand, we deduce by the fifth summation condition in (\ref{Sum:wv}) that
\begin{equation} \label{Sw:v4:p2u}
\begin{split}
&\sum_{i=0}^n \sum_{k,\alpha,\beta=1}^d w_i v_{i,\alpha} v_{i,\beta} v_{i,k} v_{i,j} \partial_{x_k} \partial_{x_\alpha} u_\beta^\varepsilon = c_s^4 \sum_{k,\alpha,\beta=1}^d \big( \delta_{\alpha \beta} \delta_{kj} + \delta_{\alpha k} \delta_{\beta j} + \delta_{\alpha j} \delta_{\beta k} \big) \partial_{x_k} \partial_{x_\alpha} u_\beta^\varepsilon \\
&\ \ = c_s^4 \left\{ \sum_{\alpha,\beta=1}^d \delta_{\alpha \beta} \partial_{x_j} \partial_{x_\alpha} u_\beta^\varepsilon + \sum_{k,\alpha=1}^d \delta_{\alpha k} \partial_{x_k} \partial_{x_\alpha} u_j^\varepsilon + \sum_{k,\beta=1}^d \delta_{\beta k} \partial_{x_k} \partial_{x_j} u_\beta^\varepsilon \right\} \\
&\ \ = c_s^4 \left\{ \sum_{\alpha=1}^d \partial_{x_j} \partial_{x_\alpha} u_\alpha^\varepsilon + \sum_{\alpha=1}^d \partial_{x_\alpha}^2 u_j^\varepsilon + \sum_{\beta=1}^d \partial_{x_\beta} \partial_{x_j} u_\beta^\varepsilon \right\}.
\end{split}
\end{equation}
Substitute (\ref{Swd:v2:p2u}) and (\ref{Sw:v4:p2u}) back into (\ref{j:Sw:divAh}), we obtain that
\begin{align*}
\left\{ \sum_{i=0}^n w_i \nabla_x \cdot (A_i h_i^\varepsilon) \right\}_j = c_s^4 \Delta_x u_j^\varepsilon + c_s^4 \partial_{x_j} \operatorname{div} u^\varepsilon
\end{align*}
for any $1 \leq j \leq d$.

We set 
\begin{equation} \label{Remain:Diff}
\begin{split}
\mathcal{R}_D^\varepsilon(x,t) &:= - \varepsilon \nu \sum_{i=0}^n w_i \nabla_x \cdot (A_i \partial_t g_i^\varepsilon) - \nu \sum_{i=0}^n w_i \nabla_x \cdot \big( A_i (v_i \cdot \nabla_x \mathcal{F}_i^\varepsilon) \big) \\
&\ \ - \frac{\varepsilon \nu}{2 c_s^4} \sum_{i=0}^n w_i \nabla_x \cdot \Big( \sum_{\alpha,\beta=1}^d A_{i,\alpha,\beta} A_i \big( v_i \cdot \nabla_x \Lambda_\varepsilon(u_\alpha^\varepsilon u_\beta^\varepsilon) \big) \Big).
\end{split}
\end{equation}
It is not hard to see that $\mathcal{R}_D^\varepsilon$ converges to zero in the sense of distributions as $\varepsilon \to 0$.
Indeed, we take $\Phi \in C_\mathrm{c}^\infty(\mathbf{R}^d \times [0,T_0])$.
By the local energy estimate (\ref{loEI:Rg0}), for any $1 \leq \beta \leq d$, we have that
\begin{align*}
&\left| \sum_{i=0}^n w_i \left\{ \int_0^{T_0} \int_{\mathbf{R}^d} \nabla_x \cdot \big( A_i (v_i \cdot \nabla_x \mathcal{F}_i^\varepsilon) \big) \Phi \, dx \, dt \right\}_\beta \right| \\
&\ \ = \left| \sum_{i=0}^n \sum_{\alpha, \gamma=1}^d w_i A_{i,\alpha,\beta} v_{i,\gamma} \int_0^{T_0} \int_{\mathbf{R}^d} \mathcal{F}_i^\varepsilon \partial_{x_\gamma} \partial_{x_\alpha} \Phi \, dx \, dt \right| \\
&\ \ \lesssim T_0^{\frac{1}{2}} \| \mathcal{F}^\varepsilon \|_{L^2( [0,T_0]; H_x^m L_{v,w}^2 )} \| \nabla^2 \Phi \|_{L^\infty( [0,T_0]; H^{-m}(\mathbf{R}_x^d) )} \\
&\ \ \lesssim \varepsilon T_0^{\frac{1}{2}} \| g_0 \|_{H_x^m L_{v,w}^2} \| \nabla^2 \Phi \|_{L^\infty( [0,T_0]; H^{-m}(\mathbf{R}_x^d) )} \to 0 \quad \text{as} \quad \varepsilon \to 0.
\end{align*}
On the other hand, by estimate (\ref{UnBd:gepn}), we have that
\begin{align*}
&\left| \sum_{i=0}^n w_i \int_0^{T_0} \int_{\mathbf{R}^d} \Phi \nabla_x \cdot (A_i \partial_t g_i^\varepsilon) \, dx \, dt \right| = \left| \sum_{i=0}^n w_i \int_0^{T_0} \int_{\mathbf{R}^d} (A_i^\mathrm{T} \partial_t g_i^\varepsilon) \cdot \nabla_x \Phi \, dx \, dt \right|  \\
&\ \ \lesssim \sum_{i=0}^n w_i \int_{\mathbf{R}^d}  |g_{i,0}^\varepsilon| \big| \big( \nabla_x \Phi \big) (x,0) \big| \, dx + \sum_{i=0}^n w_i \int_{\mathbf{R}^d} |g_i^\varepsilon(T_0)| \big| \big( \nabla_x \Phi \big) (x,T_0) \big| \, dx \\
&\ \ \ \ + \sum_{i=0}^n w_i \int_0^{T_0} \int_{\mathbf{R}^d}  |g_i^\varepsilon| |\partial_t \nabla_x \Phi| \, dx \, dt \\
&\ \ \lesssim \| g_0 \|_{L^\infty( [0,T_0]; H_x^m L_{v,w}^2 )} \Big( \| \nabla_x \Phi \|_{L^\infty( [0,T_0]; H^{-m}(\mathbf{R}_x^d) )} + T_0 \| \partial_t \nabla_x \Phi \|_{L^\infty( [0,T_0]; H^{-m}(\mathbf{R}_x^d) )} \Big) < \infty.
\end{align*}
Analogously, for any $1 \leq j \leq d$, by using Proposition \ref{Hm:BA}, Plancherel's identity and estimate (\ref{UnBd:gepn}), we deduce that
\begin{align*}
&\left| \sum_{i=0}^n w_i \left\{ \int_0^{T_0} \int_{\mathbf{R}^d} \Phi \nabla_x \cdot \Big( \sum_{\alpha,\beta=1}^d A_{i,\alpha,\beta} A_i \big( v_i \cdot \nabla_x \Lambda_\varepsilon(u_\alpha^\varepsilon u_\beta^\varepsilon) \big) \Big) \, dx \, dt \right\}_j \right| \\
&\ \ = \left| \sum_{i=0}^n \sum_{\alpha,\beta,\gamma,k=1}^d w_i A_{i,\alpha,\beta} A_{i,k,j} v_{i,\gamma} \int_0^{T_0} \int_{\mathbf{R}^d} \Lambda_\varepsilon(u_\alpha^\varepsilon u_\beta^\varepsilon) \partial_{x_k} \partial_{x_\gamma} \Phi \, dx \, dt \right| \\
&\ \ \lesssim T_0 \| u^\varepsilon \|_{L^\infty([0,T_0]; H^m(\mathbf{R}_x^d) )}^2 \| \nabla_x^2 \Phi \|_{L^\infty([0,T_0]; H^{-m}(\mathbf{R}_x^d) )} \\
&\ \ \lesssim T_0 \| g_0 \|_{L^\infty([0,T_0]; H_x^m L_{v,w}^2)}^2 \| \nabla_x^2 \Phi \|_{L^\infty([0,T_0]; H^{-m}(\mathbf{R}_x^d) )} < \infty.
\end{align*}
This completes the proof of Lemma \ref{Rewr:Diff}.
\end{proof}

\section{The hydrodynamic limit}
\label{Sec:HydroLi}

In this section, we further stick to the setting in Section \ref{Sec:DerNS}.
Combining Lemma \ref{Rewr:Trans} and Lemma \ref{Rewr:Diff}, we rewrite  system (\ref{Inp:LBE1v}) as
\begin{equation} \label{NS:ep}
\left\{
 \begin{aligned}
 \partial_t u^\varepsilon - c_s^2 \nu \Delta_x u^\varepsilon + \nabla_x \cdot (u^\varepsilon \otimes u^\varepsilon) - c_s^2 \nu \nabla_x \operatorname{div} u^\varepsilon + \frac{c_s^2}{\varepsilon} \nabla_x \rho^\varepsilon + \mathcal{R}^\varepsilon(x,t) &= 0,& \\
 \operatorname{div} u^\varepsilon + \varepsilon \partial_t \rho^\varepsilon &= 0,&
 \end{aligned}
\right.
\end{equation}
where $\mathcal{R}^\varepsilon(x,t) := \mathcal{R}_T^\varepsilon(x,t) + \mathcal{R}_D^\varepsilon(x,t)$.
By the second equation of (\ref{NS:ep}) and estimate (\ref{LEE:ru}), it is easy to observe that $\operatorname{div} u^\varepsilon \to 0$ in the sense of distributions as $\varepsilon \to 0$. 
More precisely, we have that
\begin{equation} \label{Conv:div}
\begin{split}
&\left| \int_0^{T_0} \int_{\mathbf{R}^d} \psi \operatorname{div} u^\varepsilon \, dx \, dt \right| \\
&\ \ \leq \varepsilon \left| \int_{\mathbf{R}^d} \rho^\varepsilon(T_0) \psi(T_0) \, dx \right| + \varepsilon \left| \int_{\mathbf{R}^d} \rho^\varepsilon(0) \psi(0) \, dx \right| + \varepsilon \left| \int_0^{T_0} \int_{\mathbf{R}^d} \rho^\varepsilon \partial_t \psi \, dx \, dt \right| \\
&\ \ \lesssim \varepsilon \| g_0 \|_{L^\infty([0,T_0]; H_x^m L_{v,w}^2)} \Big\{ \| \psi \|_{L^\infty( [0,T_0]; H^{-m}(\mathbf{R}_x^d) )} + T_0^{\frac{1}{2}} \| \partial_t \psi \|_{L^2( [0,T_0]; H^{-m}(\mathbf{R}_x^d) )} \Big\}
\end{split}
\end{equation}
for any $\psi \in L^\infty( [0,T_0]; H^{-m}(\mathbf{R}_x^d) )$ satisfying $\partial_t \psi \in L^2( [0,T_0]; H^{-m}(\mathbf{R}_x^d) )$.
Since $u^\varepsilon \to u$ weak-$\ast$ in $t$ and weakly in $H^m(\mathbf{R}_x^d)$, we have that
\begin{align} \label{divu:cv}
\int_0^{T_0} \int_{\mathbf{R}^d} \psi \operatorname{div} (u^\varepsilon - u) \, dx \, dt = - \int_0^{T_0} \int_{\mathbf{R}^d} (u^\varepsilon - u) \cdot \nabla_x \psi \, dx \, dt \to 0
\end{align}
as $\varepsilon \to 0$ for any $\psi \in L^1( [0,T_0]; H^1(\mathbf{R}_x^d) )$ as $H^1(\mathbf{R}_x^d) \hookrightarrow H^{-m}(\mathbf{R}_x^d)$.
Since $T_0$ is finite, by combining (\ref{Conv:div}) and (\ref{divu:cv}), we obtain that 
\begin{align} \label{divu:0}
\int_0^{T_0} \int_{\mathbf{R}^d} \psi \operatorname{div} u \, dx \, dt = 0 = \int_0^{T_0} \int_{\mathbf{R}^d} u \cdot \nabla_x \psi \, dx \, dt
\end{align}
for any $\psi \in L^\infty( [0,T_0]; H^1(\mathbf{R}_x^d) )$ that satisfies $\partial_t \psi \in L^2( [0,T_0]; H^1(\mathbf{R}_x^d) )$.

\subsection{Application of the $L^2$ Helmholtz projection to equation \eqref{NS:ep}} 
\label{Sub:HelPNS}

In order to get rid of the term $\nabla_x \rho^\varepsilon$ whose coefficient is $\varepsilon^{-1}$ in the first equation of (\ref{NS:ep}), we consider the Helmholtz decomposition for $L^2(\mathbf{R}_x^d)$. 
Let us recall that for any $h \in L^2(\mathbf{R}_x^d)$, there exists a unique decomposition $h = h_0 + \nabla_x \pi$ such that
\begin{align*}
h_0 \in L_\sigma^2(\mathbf{R}_x^d) &:= \{ f \in L^2(\mathbf{R}_x^d) \bigm| \operatorname{div} f = 0 \}, \\
\nabla_x \pi \in G^2(\mathbf{R}_x^d) &:= \{ \nabla_x \pi \in L^2(\mathbf{R}_x^d) \bigm| \pi \in L_{loc}^2(\mathbf{R}_x^d) \}.
\end{align*}
Moreover, it holds that
\begin{align} \label{E:L2Helm}
\| h_0 \|_{L^2(\mathbf{R}_x^d)} + \| \nabla_x \pi \|_{L^2(\mathbf{R}_x^d)} \leq 2 \| h \|_{L^2(\mathbf{R}_x^d)}.
\end{align}
The Helmholtz projection, denoted by $\mathbb{P}$, is the projection that maps $h$ to $h_0$, i.e., we have that $\mathbb{P}(h) = h_0$ for $h \in L^2(\mathbf{R}_x^d)$.
Furthermore, the Helmholtz projection $\mathbb{P}$ satisfies the properties
\begin{align} \label{Helm:Prop}
\mathbb{P}(h_0) = h_0 \quad \forall \; h_0 \in L_\sigma^2(\mathbf{R}_x^d) \quad \text{and} \quad \mathbb{P}(\nabla_x \pi) = 0 \quad \forall \; \nabla_x \pi \in G^2(\mathbf{R}_x^d).
\end{align}
Let $\mathbb{Q} := \mathbb{I} - \mathbb{P}$ where $\mathbb{I}$ denotes the identity operator.
Applying the Helmholtz projection $\mathbb{P}$ to the first equation of (\ref{NS:ep}), we observe by the second property of (\ref{Helm:Prop}) that 
\begin{align*} 
\partial_t \mathbb{P}(u^\varepsilon) - c_s^2 \nu \Delta_x \mathbb{P}(u^\varepsilon) + \mathbb{P}\big( \nabla_x \cdot (u^\varepsilon \otimes u^\varepsilon) \big) + \mathbb{P}(\mathcal{R}^\varepsilon) = 0.
\end{align*}
\begin{proposition} \label{Helm:Hk}
Let $h \in H^k(\mathbf{R}_x^d)$ with $k \in \mathbf{N}$ satisfying $k \geq 1$. Let $h = h_0 + \nabla_x \pi$ be the Helmholtz decomposition of $h$ in $L^2(\mathbf{R}_x^d)$. 
For any $\tau \in \mathbf{N}_0^d$ with $|\tau|_\mathrm{s} \leq k$, the unique Helmholtz decomposition of $\partial_x^\tau h$ in $L^2(\mathbf{R}_x^d)$ is given by $\partial_x^\tau h = \partial_x^\tau h_0 + \nabla_x (\partial_x^\tau \pi)$.
Moreover, it holds that
\begin{align*} 
\| \partial_x^\tau h_0 \|_{L^2(\mathbf{R}_x^d)} + \| \nabla_x (\partial_x^\tau \pi) \|_{L^2(\mathbf{R}_x^d)} \leq 2 \| \partial_x^\tau h \|_{L^2(\mathbf{R}_x^d)}.
\end{align*}
\end{proposition}
\begin{proof}
Let $\tau \in \mathbf{N}_0^d$ satisfying $|\tau|_\mathrm{s} \leq k$.
Since we are working in the whole space $\mathbf{R}^d$, the $L^2$
Helmholtz projection $\mathbb{P}$ is explicitly defined, i.e., for $1
\leq i,j \leq d$, the $(i,j)$-entry of $\mathbb{P}$ is given by
$\mathbb{P}_{ij} := \delta_{ij} + R_i R_j$ where $R_i$ and $R_j$
denote the $i$-th and $j$-th component of the Riesz transform, respectively.
Hence, it can be easily observed that the differentiation $\partial_x$ commutes with $\mathbb{P}$ and $\partial_x^\tau h_0 = \partial_x^\tau \mathbb{P}(h) = \mathbb{P}(\partial_x^\tau h)$. 
The $L^2$ boundedness of $\mathbb{P}$ can be proved by working with the explicit formula for $\mathbb{P}$.
As a result, by estimate (\ref{E:L2Helm}) we have that
\begin{align*}
\| \partial_x^\tau h_0 \|_{L^2(\mathbf{R}_x^d)} = \| \mathbb{P}(\partial_x^\tau h) \|_{L^2(\mathbf{R}_x^d)} \lesssim \|\partial_x^\tau h \|_{L^2(\mathbf{R}_x^d)},
\end{align*}
i.e., $\partial_x^\tau h_0 \in L^2(\mathbf{R}_x^d)$. 
Note that $\partial_x^\tau h_0$ is divergence free.
Analogously, by the $L^2$ boundedness of $\mathbb{I} - \mathbb{P}$ and estimate (\ref{E:L2Helm}), we have that 
\begin{align*}
\| \partial_x^\tau (\nabla_x \pi) \|_{L^2(\mathbf{R}^d)} = \| \partial_x^\tau h - \mathbb{P}(\partial_x^\tau h) \|_{L^2(\mathbf{R}_x^d)} \lesssim \| \partial_x^\tau h \|_{L^2(\mathbf{R}_x^d)},
\end{align*}
i.e., $\nabla_x (\partial_x^\tau \pi) \in L^2(\mathbf{R}_x^d)$.
By noting that $\partial_x^\tau \pi \in L^2(\mathbf{R}_x^d)$ if $\tau \in \mathbf{N}_0^d$ satisfies $1 \leq |\tau|_\mathrm{s} \leq k$, we conclude that $\partial_x^\tau h = \partial_x^\tau h_0 + \nabla_x (\partial_x^\tau \pi)$ is indeed the unique Helmholtz decomposition of $\partial_x^\tau h$ in $L^2(\mathbf{R}^d)$.
\end{proof}

Combining Proposition \ref{Helm:Hk} with Corollary \ref{WCon:ru}, we observe that the sequence $\{ \mathbb{P}(u^\varepsilon) \}_\varepsilon$ is uniformly bounded in $L^\infty([0,T_0]; H^m(\mathbf{R}_x^d) )$.
Hence, by suppressing subsequences again, it can be concluded that there exists $\widetilde{u} \in L^\infty([0,T_0]; H^m(\mathbf{R}_x^d) )$ such that $\mathbb{P}(u^\varepsilon) \to \widetilde{u}$ where the convergence is weak-$\ast$ in time $t$ and weakly in $H^m(\mathbf{R}_x^d)$, i.e., for any $\Psi \in L^1([0,T_0]; H^{-m}(\mathbf{R}_x^d) )$, it holds that
\begin{align} \label{Cv:tout}
\int_0^{T_0} \int_{\mathbf{R}^d} \mathbb{P}(u^\varepsilon) \cdot \Psi \, dx \, dt \to \int_0^{T_0} \int_{\mathbf{R}^d} \widetilde{u} \cdot \Psi \, dx \, dt \quad \text{as} \quad \varepsilon \to 0.
\end{align}
On the other hand, for any $\Psi \in L^1([0,T_0]; L^2(\mathbf{R}_x^d) )$, it holds that
\begin{align} \label{Cv:toPu}
\int_0^{T_0} \int_{\mathbf{R}^d} \mathbb{P}(u^\varepsilon - u) \cdot \Psi \, dx \, dt = \int_0^{T_0} \int_{\mathbf{R}^d} (u^\varepsilon - u) \cdot \mathbb{P}(\Psi) \, dx \, dt \to 0 \quad \text{as} \quad \varepsilon \to 0.
\end{align}
Combining convergence statements (\ref{Cv:tout}) and (\ref{Cv:toPu}), we deduce that
\begin{align} \label{Pu:utilde}
\int_0^{T_0} \int_{\mathbf{R}^d} \widetilde{u} \cdot \Psi \, dx \, dt = \int_0^{T_0} \int_{\mathbf{R}^d} \mathbb{P}(u) \cdot \Psi \, dx \, dt, \quad \forall \; \Psi \in L^1( [0,T_0]; L^2(\mathbf{R}_x^d) ).
\end{align}
We next prove that with property (\ref{divu:0}), the limit point $u$ is actually divergence free, i.e., in this case we have $\mathbb{P}(u) = u$ a.e..

\begin{lemma} \label{Weak:divu0}
Suppose that $u \in L^\infty( [0,T_0]; L^2(\mathbf{R}_x^d) )$ satisfies 
\begin{align*}
\int_0^{T_0} \int_{\mathbf{R}^d} u \cdot \nabla_x \psi \, dx \, dt = 0, \quad \forall \; \psi \in W^{1,\infty}( [0,T_0]; H^1(\mathbf{R}_x^d) ).
\end{align*}
Then, for any $\Psi \in C^\infty( [0,T_0]; C_\mathrm{c}^\infty(\mathbf{R}_x^d) )$, it holds that
\begin{align*} 
\int_0^{T_0} \int_{\mathbf{R}^d} \mathbb{P}(u) \cdot \Psi \, dx \, dt = \int_0^{T_0} \int_{\mathbf{R}^d} u \cdot \Psi \, dx \, dt.
\end{align*}
\end{lemma}
\begin{proof}
Let $\Psi \in C^\infty( [0,T_0]; C_\mathrm{c}^\infty(\mathbf{R}_x^d) )$.
For each $t \in [0,T_0]$, let $\Psi(x,t) = \Psi_0(x,t) + \big( \nabla_x \Pi \big) (x,t)$ be the Helmholtz decomposition of $\Psi(\cdot, t)$ in $L^2(\mathbf{R}_x^d)$ where $\Pi(\cdot,t) = \Delta^{-1} \operatorname{div} \Psi(\cdot, t)$.
By integration by parts, we can easily observe that for each $t \in [0,T_0]$, $\Pi(\cdot, t)$ is indeed a Riesz potential of $\Psi(\cdot, t)$.
Since we are considering test function $\Psi$, there exists a constant $C(2,d)$, depending only on $2$ and $d$, such that
\begin{align} \label{L2E:PreHel}
\| \Pi(\cdot, t) \|_{L^2(\mathbf{R}_x^d)} \leq C(2,d) \| \Psi(\cdot, t) \|_{L^{\frac{2d}{d+2}}(\mathbf{R}_x^d)} < \infty,
\end{align}
see e.g. \cite[Chapter V Theorem 1]{Stein}.
Combining (\ref{L2E:PreHel}) with estimate (\ref{E:L2Helm}) for the Helmholtz decomposition, we deduce that $\Pi \in L^\infty( [0,T_0]; H^1(\mathbf{R}_x^d) )$.
Moreover, since the operator $\Delta^{-1} \operatorname{div}$ is linear,  we have that 
\begin{align} \label{pa:t:Psi}
\frac{\Pi(x, t+h) - \Pi(x,t)}{h} = \Delta^{-1} \operatorname{div} \Big( \frac{\Psi(x, t+h) - \Psi(x,t)}{h} \Big)
\end{align}
for any $t, t + h \in (0,T_0)$. We next define
\begin{align*}
\varphi_\ast(x) := \sup_{t \in [0,T_0]} \big| \big( \partial_t \Psi \big) (x,t) \big|, \quad \forall \; x \in \mathbf{R}^d.
\end{align*}
By the mean value theorem, for any $t, t+h \in (0, T_0)$, it holds that
\begin{align} \label{Mean:Psi}
\big| \Psi(x,t+h) - \Psi(x,t) \big| \leq h \varphi_\ast(x).
\end{align}
By integration by parts and (\ref{Mean:Psi}), we can deduce that for any $x \in \mathbf{R}^d$ and $t, t+h \in (0,T_0)$,
\begin{align*}
\Big| \Delta^{-1} \operatorname{div} \Big( \frac{\Psi(x, t+h) - \Psi(x,t)}{h} \Big) \Big| &\lesssim \int_{\mathbf{R}^d} \frac{1}{|x-y|^{d-1}} \Big| \frac{\Psi(y, t+h) - \Psi(y,t)}{h} \Big| \, dy \\
&\leq \int_{\mathbf{R}^d} \frac{1}{|x-y|^{d-1}} |\varphi_\ast(y)| \, dy =: I_1\big( |\varphi_\ast| \big) (x).
\end{align*}
Considering that $\Psi$ is smooth with respect to both $x$ and $t$ and has compact support in $\mathbf{R}^d$ with respect to $x$, it can be easily observed that $\varphi_\ast \in L^2(\mathbf{R}_x^d) \cap L^\infty(\mathbf{R}_x^d)$ also has compact support.
Hence, by analogous reason as (\ref{L2E:PreHel}), we have that $I_1 (|\varphi_\ast|) \in L^2(\mathbf{R}_x^d)$.
Then, by taking the limit $h \to 0$ for equality (\ref{pa:t:Psi}), we can conclude by the dominated convergence theorem that 
\begin{align*}
\big( \partial_t \Pi \big) (x,t) = \Delta^{-1} \operatorname{div} (\partial_t \Psi), \quad \forall \; (x,t) \in \mathbf{R}^d \times [0,T_0].
\end{align*}
Furthermore, since each component of $\nabla \Delta^{-1}
\operatorname{div}$ is nothing but a linear combination of $R_i R_j$ ($1 \leq i,j \leq d$), we observe that $\nabla \Delta^{-1} \operatorname{div}$ is bounded in $L^2(\mathbf{R}_x^d)$.
Thus, we obtain that $\partial_t \Pi \in L^\infty( [0,T_0]; H^1(\mathbf{R}_x^d) )$.
Using (\ref{divu:0}), we can finally deduce that for any $\Psi \in C^\infty( [0,T_0]; C_\mathrm{c}^\infty(\mathbf{R}_x^d) )$, it holds that
\begin{align*}
\int_0^{T_0} \int_{\mathbf{R}^d} \mathbb{P}(u) \cdot \Psi \, dx \, dt = \int_0^{T_0} \int_{\mathbf{R}^d} \mathbb{P}(u) \cdot \Psi_0 \, dx \, dt = \int_0^{T_0} \int_{\mathbf{R}^d} u \cdot \Psi_0 \, dx \, dt = \int_0^{T_0} \int_{\mathbf{R}^d} u \cdot \Psi \, dx \, dt.
\end{align*}
This completes the proof of Lemma \ref{Weak:divu0}.
\end{proof}

Combining Lemma \ref{Weak:divu0} with the convergence statement (\ref{Cv:tout}) and equality (\ref{Pu:utilde}), we establish the convergence of $\mathbb{P}(u^\varepsilon)$ to $u$ in the sense of distributions.
In fact, the convergence of $\mathbb{P}(u^\varepsilon)$ to $u$ can be proved to be stronger than this weak sense.

\begin{lemma} \label{StCo:Puep}
Suppressing subsequences, $\mathbb{P}(u^\varepsilon) \to u$ strongly in $C( [0,T_0]; H^{m-1}(\mathbf{R}_x^d) )$, i.e., it holds that
\begin{align*}
\| \mathbb{P}(u^\varepsilon) - u \|_{L^\infty( [0,T_0]; H^{m-1}(\mathbf{R}_x^d) )} \to 0 \quad \text{as} \quad \varepsilon \to 0.
\end{align*}
\end{lemma}
\begin{proof}
The second equation of (\ref{Inp:LBE1v}) can be written as
\begin{align} \label{Re:NSu}
\partial_t u^\varepsilon + \frac{1}{\varepsilon} \nabla_x \cdot \Big( \sum_{i=0}^n w_i A_i \mathcal{F}_i^\varepsilon \Big) + \frac{c_s^2}{\varepsilon} \nabla_x \rho^\varepsilon = - \nabla_x \cdot (u^\varepsilon \otimes u^\varepsilon) - \mathcal{R}_T^\varepsilon.
\end{align}
where $\mathcal{F}_i^\varepsilon = g_i^\varepsilon - \mathcal{G}_{i,eq}^\varepsilon$ for every $0 \leq i \leq n$.
Let $\tau \in \mathbf{N}_0^d$ with $|\tau|_\mathrm{s} \leq m-1$.
Applying $\partial_x^\tau$ and the Helmholtz projection $\mathbb{P}$ to equation (\ref{Re:NSu}), we obtain 
\begin{equation} \label{P:pxNS}
\begin{split}
&\partial_t \mathbb{P}(\partial_x^\tau u^\varepsilon) + \frac{1}{\varepsilon} \mathbb{P}\left( \nabla_x \cdot \Big( \sum_{i=0}^n w_i A_i \partial_x^\tau \mathcal{F}_i^\varepsilon \Big) \right) \\
&\ \ = - \sum_{\sigma \leq \tau} \mathbb{P}\Big( (\partial_x^\sigma u^\varepsilon \cdot \nabla_x) \partial_x^{\tau-\sigma} u^\varepsilon + \partial_x^\sigma u^\varepsilon \operatorname{div} (\partial_x^{\tau - \sigma} u^\varepsilon) \Big) - \mathbb{P}(\partial_x^\tau \mathcal{R}_T^\varepsilon).
\end{split}
\end{equation}
Integrating equation (\ref{P:pxNS}) over the time interval $[t_1, t_2] \subseteq [0,T_0]$ and then taking its inner product with $\mathbb{P}\big( \partial_x^\tau u^\varepsilon (t_2) \big) - \mathbb{P}\big( \partial_x^\tau u^\varepsilon (t_1) \big)$ in the $L_x^2$ sense, we have that
\begin{equation} \label{EE:ArAs}
\begin{split}
&\big\| \mathbb{P}\big( \partial_x^\tau u^\varepsilon (t_2) \big) - \mathbb{P}\big( \partial_x^\tau u^\varepsilon (t_1) \big) \big\|_{L^2(\mathbf{R}_x^d)} \\
&\ \ = - \frac{1}{\varepsilon} \int_{t_1}^{t_2} \int_{\mathbf{R}^d} \left( \nabla_x \cdot \Big( \sum_{i=0}^n w_i A_i \partial_x^\tau \mathcal{F}_i^\varepsilon \Big) \right) \cdot \mathbb{P}\Big( \partial_x^\tau u^\varepsilon (t_2) - \partial_x^\tau u^\varepsilon (t_1) \Big) \, dx \, dt \\
&\ \ \quad - \sum_{\sigma \leq \tau} \int_{t_1}^{t_2} \int_{\mathbf{R}^d} (\partial_x^\sigma u^\varepsilon \cdot \nabla_x) \partial_x^{\tau-\sigma} u^\varepsilon \cdot \mathbb{P}\Big( \partial_x^\tau u^\varepsilon (t_2) - \partial_x^\tau u^\varepsilon (t_1) \Big) \, dx \, dt \\
&\ \ \quad - \sum_{\sigma \leq \tau} \int_{t_1}^{t_2} \int_{\mathbf{R}^d} \operatorname{div} (\partial_x^{\tau - \sigma} u^\varepsilon) \partial_x^\sigma u^\varepsilon \cdot \mathbb{P}\Big( \partial_x^\tau u^\varepsilon (t_2) - \partial_x^\tau u^\varepsilon (t_1) \Big) \, dx \, dt \\
&\ \ \quad - \int_{t_1}^{t_2} \int_{\mathbf{R}^d} \partial_x^\tau \mathcal{R}_T^\varepsilon \cdot \mathbb{P}\Big( \partial_x^\tau u^\varepsilon (t_2) - \partial_x^\tau u^\varepsilon (t_1) \Big) \, dx \, dt = (\mathrm{I}) + (\mathrm{II}) + (\mathrm{III}) + (\mathrm{IV}).
\end{split}
\end{equation}

For any $t \in [t_1, t_2]$, we can deduce simply by H{\"o}lder's inequality that
\begin{align*}
&\left| \sum_{\alpha, \beta=1}^d \int_{\mathbf{R}^d} \left( \sum_{i=0}^n w_i A_{i,\alpha,\beta} \partial_{x_\alpha} \partial_x^\tau \mathcal{F}_i^\varepsilon \right) \mathbb{P}\Big( \partial_x^\tau u^\varepsilon(t_2) - \partial_x^\tau u^\varepsilon(t_1) \Big)_\beta \, dx \right| \\
&\ \ \lesssim \| \nabla_x \partial_x^\tau \mathcal{F}^\varepsilon (t) \|_{L^2(\mathbf{R}_x^d)} \| \partial_x^\tau u^\varepsilon \|_{L^\infty( [0,T_0]; L^2(\mathbf{R}_x^d) )}.
\end{align*}
Hence, by the local energy inequality (\ref{loEI:Rg0}), we have that
\begin{align*}
| (\mathrm{I}) | &\lesssim \frac{1}{\varepsilon} \left( \int_{t_1}^{t_2} \| \nabla_x \partial_x^\tau \mathcal{F}^\varepsilon (t) \|_{L^2(\mathbf{R}_x^d)}^2 \, dt \right)^{\frac{1}{2}} (t_2 - t_1)^{\frac{1}{2}} \| \partial_x^\tau u^\varepsilon \|_{L^\infty( [0,T_0]; L^2(\mathbf{R}_x^d) )} \\
&\lesssim (2 + \nu)^{\frac{1}{2}} (t_2 - t_1)^{\frac{1}{2}} \| g_0 \|_{H_x^m L_{v,w}^2} \| \partial_x^\tau g^\varepsilon \|_{L^\infty( [0,T_0]; L_x^2 L_{v,w}^2 )}.
\end{align*}
As we have already seen in the proof of Proposition \ref{Hm:BA}, for $m > d$ and $\sigma \in \mathbf{N}_0^d$ satisfying $\sigma \leq \tau$, either $|\sigma|_\mathrm{s}$ or $|\tau|_\mathrm{s} - |\sigma|_\mathrm{s} +1$ is less than or equal to $m-[\frac{d}{2}]-1$.
Hence, we may assume without loss of generality that $|\sigma|_\mathrm{s} \leq m-[\frac{d}{2}]-1$.
In this case, for $t \in [t_1, t_2]$, we have that
\begin{align*}
&\left| \sum_{\alpha, \beta=1}^d \int_{\mathbf{R}^d} (\partial_x^\sigma u_\alpha^\varepsilon) (\partial_{x_\alpha} \partial_x^{\tau - \sigma} u_\beta^\varepsilon) \mathbb{P}\Big( \partial_x^\tau u^\varepsilon(t_2) - \partial_x^\tau u^\varepsilon(t_1) \Big)_\beta \, dx \right| \\
&\ \ \lesssim \| u^\varepsilon \|_{L^\infty( [0,T_0]; H^m(\mathbf{R}_x^d) )}^2 \| \partial_x^\tau u^\varepsilon \|_{L^\infty( [0,T_0]; L^2(\mathbf{R}_x^d) )}.
\end{align*}
Hence, by the local energy inequality (\ref{loEI:Rg0}), we deduce that
\begin{align} \label{Es:II:StCo}
| (\mathrm{II}) | \lesssim (2+\nu) (t_2 - t_1) \| g_0 \|_{H_x^m L_{v,w}^2}^2 \| \partial_x^\tau g^\varepsilon \|_{L^\infty( [0,T_0]; L_x^2 L_{v,w}^2 )}.
\end{align}
Note that $| (\mathrm{III}) |$ can be estimated in exactly the same way as $| (\mathrm{II}) |$, i.e., estimate (\ref{Es:II:StCo}) holds for $| (\mathrm{III}) |$.
Since
\begin{align*}
&\left| \sum_{\alpha, \beta=1}^d \int_{\mathbf{R}^d} \partial_x^\tau \partial_{x_\beta} \big( \Lambda_\varepsilon(u_\alpha^\varepsilon u_\beta^\varepsilon) - u_\alpha^\varepsilon u_\beta^\varepsilon \big) \mathbb{P}\Big( \partial_x^\tau u^\varepsilon(t_2) - \partial_x^\tau u^\varepsilon(t_1) \Big)_\alpha \, dx \right| \\
&\ \ \lesssim \big\| \mathbb{P}\big( \partial_x^\tau u^\varepsilon(t_2) - \partial_x^\tau u^\varepsilon(t_1) \big) \big\|_{L^2(\mathbf{R}_x^d)} \sum_{\beta=1}^d \Big\{ \| \partial_x^\tau \partial_{x_\beta} \Lambda_\varepsilon(u^\varepsilon u_\beta^\varepsilon) \|_{L^2(\mathbf{R}_x^d)} + \| \partial_x^\tau \partial_{x_\beta} (u^\varepsilon u_\beta^\varepsilon) \|_{L^2(\mathbf{R}_x^d)} \Big\} \\
&\ \ \lesssim \| u^\varepsilon \|_{L^\infty( [0,T_0]; H^m(\mathbf{R}_x^d) )}^2 \| \partial_x^\tau u^\varepsilon \|_{L^\infty( [0,T_0]; L^2(\mathbf{R}_x^d) )}
\end{align*}
for any $t \in [t_1, t_2]$, estimate (\ref{Es:II:StCo}) holds for $| (\mathrm{IV}) |$ analogously.

By summing up (\ref{EE:ArAs}) over all $\tau \in \mathbf{N}_0^d$ with $|\tau|_\mathrm{s} \leq m-1$, we deduce by the local energy inequality (\ref{loEI:Rg0}) that
\begin{align*}
\big\| \mathbb{P}\big( u^\varepsilon(t_2) \big) - \mathbb{P}\big( u^\varepsilon(t_1) \big) \big\|_{H^{m-1}(\mathbf{R}_x^d)} \lesssim (2 + \nu)^{\frac{3}{2}} (1 + \sqrt{T_0}) (t_2 - t_1)^{\frac{1}{2}} \| g_0 \|_{H_x^m L_{v,w}^2}^3.
\end{align*}
Therefore, this shows that $\{ \mathbb{P}(u^\varepsilon) \}_\varepsilon \subset C( [0,T_0]; H^{m-1}(\mathbf{R}_x^d) )$ and $\big\{ \big\| \mathbb{P}\big( u^\varepsilon \big) (t) \big\|_{H^{m-1}(\mathbf{R}_x^d)} \big\}_\varepsilon$ is equi-continuous in time $t$. 
By the Arzel$\grave{\text{a}}$-Ascoli theorem, we conclude by suppressing subsequences that there exists $u_\ast \in C( [0,T_0]; H^{m-1}(\mathbf{R}_x^d) )$ such that
\begin{align*} 
\| \mathbb{P}(u^\varepsilon) - u_\ast \|_{L^\infty( [0,T_0]; H^{m-1}(\mathbf{R}_x^d) )} \to 0 \quad \text{as} \quad \varepsilon \to 0.
\end{align*}
By splitting $u_\ast - u$ into $( u_\ast - \mathbb{P}(u^\varepsilon) ) + ( \mathbb{P}(u^\varepsilon) - u )$, we can deduce by Lemma \ref{Weak:divu0} that
\begin{align*}
\int_0^{T_0} \int_{\mathbf{R}^d} (u_\ast - u) \cdot \Psi \, dx \, dt = 0, \quad \forall \; \Psi \in C^\infty( [0,T_0]; C_\mathrm{c}^\infty(\mathbf{R}_x^d) ).
\end{align*}
Finally, by the fundamental lemma of the calculus of variations, we conclude that $u_\ast = u$ a.e. in $\mathbf{R}^d \times [0,T_0]$.
This completes the proof of Lemma \ref{StCo:Puep}.
\end{proof}

\subsection{Convergence to the incompressible Navier-Stokes equations}
\label{Sub:ConNS}

Finally, we decompose 
\begin{align*}
\mathbb{P}\big( \nabla_x \cdot (u^\varepsilon \otimes u^\varepsilon) \big) = \mathbb{P}\Big( \nabla_x \cdot \big( \mathbb{P}(u^\varepsilon) \otimes \mathbb{P}(u^\varepsilon) \big) \Big) + \mathcal{R}_\mathbb{P}(u^\varepsilon)
\end{align*}
with
\begin{align*}
\mathcal{R}_\mathbb{P}(u^\varepsilon) := \mathbb{P}\Big( \nabla_x \cdot \big( \mathbb{P}(u^\varepsilon) \otimes \mathbb{Q}(u^\varepsilon) \big) \Big) + \mathbb{P}\Big( \nabla_x \cdot \big( \mathbb{Q}(u^\varepsilon) \otimes \mathbb{P}(u^\varepsilon) \big) \Big) + \mathbb{P}\Big( \nabla_x \cdot \big( \mathbb{Q}(u^\varepsilon) \otimes \mathbb{Q}(u^\varepsilon) \big) \Big).
\end{align*}
Let $C_{\mathrm{c}, \sigma}^\infty(\mathbf{R}^d) := \{ w \in C_\mathrm{c}^\infty(\mathbf{R}^d) \bigm| \operatorname{div} w = 0 \; \; \text{in} \; \; \mathbf{R}^d \}$.

\begin{lemma} \label{RemHel0}
For any $\Phi \in C^\infty( [0,T_0]; C_{\mathrm{c}, \sigma}^\infty(\mathbf{R}_x^d) )$, it holds that
\begin{align*}
\left| \int_0^{T_0} \int_{\mathbf{R}^d} \mathcal{R}_\mathbb{P}(u^\varepsilon) \cdot \Phi \, dx \, dt \right| \to 0 \quad \text{as} \quad \varepsilon \to 0.
\end{align*}
\end{lemma}
\begin{proof}
Since the test function $\Phi$ we consider is already divergence free,
\begin{align*}
\int_0^{T_0} \int_{\mathbf{R}^d} \mathbb{P}\Big( \nabla_x \cdot \big( \mathbb{P}(u^\varepsilon) \otimes \mathbb{Q}(u^\varepsilon) \big) \Big) \cdot \Phi \, dx \, dt = \int_0^{T_0} \int_{\mathbf{R}^d} \mathbb{P}(u^\varepsilon) \otimes \mathbb{Q}(u^\varepsilon) : \nabla_x \Phi \,  dx \, dt.
\end{align*}
We further decompose 
\begin{align*}
\int_0^{T_0} \int_{\mathbf{R}^d} \mathbb{P}(u^\varepsilon) \otimes \mathbb{Q}(u^\varepsilon) : \nabla_x \Phi \,  dx \, dt &= \int_0^{T_0} \int_{\mathbf{R}^d} \big( \mathbb{P}(u^\varepsilon) - u \big) \otimes \mathbb{Q}(u^\varepsilon) : \nabla_x \Phi \, dx \, dt \\
&\ \ + \int_0^{T_0} \int_{\mathbf{R}^d} u \otimes \mathbb{Q}(u^\varepsilon) : \nabla_x \Phi \, dx \, dt.
\end{align*}
By Lemma \ref{StCo:Puep}, we deduce that
\begin{align*}
&\left| \int_0^{T_0} \int_{\mathbf{R}^d} \big( \mathbb{P}(u^\varepsilon) - u \big) \otimes \mathbb{Q}(u^\varepsilon) : \nabla_x \Phi \, dx \, dt \right| \\
&\ \ \leq \int_0^{T_0} \big\| \mathbb{P}\big( u^\varepsilon \big) (t) - u(t) \|_{L^\infty(\mathbf{R}_x^d)} \big\| \mathbb{Q}\big( u^\varepsilon \big) (t) \big\|_{L^2(\mathbf{R}_x^d)} \| \nabla_x \Phi \|_{L^2(\mathbf{R}_x^d)} \, dt \\
&\ \ \lesssim T_0 \| \mathbb{P}(u^\varepsilon) - u \|_{L^\infty( [0,T_0]; H^2(\mathbf{R}_x^d) )} \| g_0 \|_{H_x^m L_{v,w}^2} \| \nabla_x \Phi \|_{L^\infty( [0,T_0]; L^2(\mathbf{R}_x^d) )} \to 0
\end{align*}
as $\varepsilon \to 0$.
Since Lemma \ref{Weak:divu0} says that $\mathbb{Q}(u) = 0$ a.e. in $\mathbf{R}^d \times [0,T_0]$, we have that
\begin{align*}
\int_0^{T_0} \int_{\mathbf{R}^d} u \otimes \mathbb{Q}(u^\varepsilon) : \nabla_x \Phi \, dx \, dt &= \int_0^{T_0} \int_{\mathbf{R}^d} u \otimes \mathbb{Q}(u^\varepsilon - u) : \nabla_x \Phi \, dx \, dt \\
&= \int_0^{T_0} \int_{\mathbf{R}^d} (u^\varepsilon - u) \cdot \mathbb{Q}(u \cdot \nabla_x \Phi) \, dx \, dt.
\end{align*}
Noting that $u \cdot \nabla_x \Phi \in L^\infty( [0,T_0]; H^m(\mathbf{R}_x^d) )$ and $H^m(\mathbf{R}_x^d) \hookrightarrow H^{-m}(\mathbf{R}_x^d)$, we can then deduce from the weak convergence of $u^\varepsilon$ to $u$ in $L^\infty( [0,T_0]; H^m(\mathbf{R}_x^d) )$ that
\begin{align*}
\int_0^{T_0} \int_{\mathbf{R}^d} (u^\varepsilon - u) \cdot \mathbb{Q}(u \cdot \nabla_x \Phi) \, dx \, dt \to 0 \quad \text{as} \quad \varepsilon \to 0.
\end{align*}

On the other hand, by using Lemma \ref{Rewr:Trans}, we can further deduce from system (\ref{Inp:LBE1v}) that
\begin{equation} \label{WeCon:QQ}
\left\{
 \begin{aligned}
 &\varepsilon \partial_t \rho^\varepsilon + \operatorname{div} \mathbb{Q}(u^\varepsilon) =& &0,& \\
 &\varepsilon \partial_t \mathbb{Q}(u^\varepsilon) + c_s^2 \nabla_x \rho^\varepsilon =& &- \mathbb{Q}\bigg( \nabla_x \cdot \Big( \sum_{i=0}^n w_i A_i \mathcal{F}_i^\varepsilon \Big) \bigg) - \varepsilon \mathbb{Q}\big( \nabla_x \cdot (u^\varepsilon \otimes u^\varepsilon) \big) - \varepsilon \mathbb{Q}(\mathcal{R}_T^\varepsilon),&
 \end{aligned}
\right.
\end{equation}
where the second equation is obtained by simply applying the projection $\mathbb{Q}$ to equation (\ref{Re:NSu}).
Combining Corollary \ref{WCon:ru} with Proposition \ref{Helm:Hk}, we observe that $\{ \rho^\varepsilon \}_\varepsilon$ and $\{ \mathbb{Q}(u^\varepsilon) \}_\varepsilon$ are both uniformly bounded in $L^\infty( [0,T_0]; H^m(\mathbf{R}_x^d) )$.
By Proposition \ref{Helm:Hk} and the local energy inequality (\ref{loEI:Rg0}), we have that
\begin{align*}
\int_0^{T_0} \bigg\| \mathbb{Q}\bigg( \nabla_x \cdot \Big( \sum_{i=0}^n w_i A_i \mathcal{F}_i^\varepsilon \Big) \bigg) \bigg\|_{H^{m-1}(\mathbf{R}_x^d)} \, dt &\lesssim T_0^{\frac{1}{2}} \left( \int_0^{T_0} \| \mathcal{F}^\varepsilon(t) \|_{H_x^m L_{v,w}^2}^2 \, dt \right)^{\frac{1}{2}} \\
&\lesssim \varepsilon (2 + \nu)^{\frac{1}{2}} \| g_0 \|_{H_x^m L_{v,w}^2}.
\end{align*}
Combining Proposition \ref{Hm:BA}, Proposition \ref{Helm:Hk} and the local energy inequality (\ref{loEI:Rg0}), we can deduce that
\begin{equation} \label{Es:QuGu}
\begin{split}
\int_0^{T_0} \big\| \mathbb{Q}\big( \nabla_x \cdot (u^\varepsilon \otimes u^\varepsilon) \big) \big\|_{H^{m-1}(\mathbf{R}_x^d)} \, dt &\lesssim \int_0^{T_0} \| u^\varepsilon(t) \|_{H^m(\mathbf{R}_x^d)}^2 \, dt \\
&\lesssim T_0 (2+\nu) \| g_0 \|_{H_x^m L_{v,w}^2}^2.
\end{split}
\end{equation}
Moreover, by the Plancherel's identity, we observe that the $L_t^1 H_x^{m-1}$ norm of $\mathbb{Q}\big( \mathcal{R}_T^\varepsilon \big)$ follows estimate (\ref{Es:QuGu}) as well.
Therefore, we show that the right hand side of the second equation in (\ref{WeCon:QQ}) converges to zero in the strong sense in $L^1( [0,T_0]; H^{m-1}(\mathbf{R}_x^d) )$ as $\varepsilon \to 0$.
By a compensated compactness result due to Lions and Masmoudi \cite{LioMas} (see also \cite[Theorem A.2]{GSR09}), which basically says that fast oscillating acoustic waves do not contribute to the macroscopic dynamics in the incompressible limits, we can conclude that
\begin{align*}
\mathbb{P}\Big( \nabla_x \cdot \big( \mathbb{Q}(u^\varepsilon) \otimes \mathbb{Q}(u^\varepsilon) \big) \Big) \to 0
\end{align*}
in the sense of distributions as $\varepsilon \to 0$.
This completes the proof of Lemma \ref{RemHel0}.
\end{proof}
\begin{lemma} \label{HyLim:trans}
For any $\Phi \in C^\infty( [0,T_0]; C_{\mathrm{c},\sigma}^\infty(\mathbf{R}_x^d) )$, it holds that
\begin{align*}
\int_0^{T_0} \int_{\mathbf{R}^d} \mathbb{P}\big( \nabla_x \cdot (u^\varepsilon \otimes u^\varepsilon) \big) \cdot \Phi \, dx \, dt \to \int_0^{T_0} \int_{\mathbf{R}^d} \big( \nabla_x \cdot (u \otimes u) \big) \cdot \Phi \, dx \, dt
\end{align*}
as $\varepsilon \to 0$.
\end{lemma}
\begin{proof}
We decompose 
\begin{align*}
&\int_0^{T_0} \int_{\mathbf{R}^d} \Big( \nabla_x \cdot \big( \mathbb{P}(u^\varepsilon) \otimes \mathbb{P}(u^\varepsilon) \big) \Big) \cdot \Phi \, dx \, dt \\
&= \int_0^{T_0} \int_{\mathbf{R}^d} \big( \mathbb{P}(u^\varepsilon) - u \big) \cdot \nabla_x \Phi \cdot \mathbb{P}(u^\varepsilon) \, dx \, dt + \int_0^{T_0} \int_{\mathbf{R}^d} u \cdot \nabla_x \Phi \cdot \big( \mathbb{P}(u^\varepsilon) - u \big) \, dx \, dt \\
&\ \ + \int_0^{T_0} \int_{\mathbf{R}^d} \big( \nabla_x \cdot (u \otimes u) \big) \cdot \Phi \, dx \, dt.
\end{align*}
Combining Lemma \ref{StCo:Puep} with Proposition \ref{Helm:Hk}, Corollary \ref{WCon:ru} and Remark \ref{Bdd:rhou}, we deduce that
\begin{align*}
&\left| \int_0^{T_0} \int_{\mathbf{R}^d} \big( \mathbb{P}(u^\varepsilon) - u \big) \cdot \nabla_x \Phi \cdot \mathbb{P}(u^\varepsilon) \, dx \, dt \right| + \left| \int_0^{T_0} \int_{\mathbf{R}^d} u \cdot \nabla_x \Phi \cdot \big( \mathbb{P}(u^\varepsilon) - u \big) \, dx \, dt \right| \\
&\ \ \lesssim T_0 (2+\nu)^{\frac{1}{2}} \| \mathbb{P}(u^\varepsilon) - u \|_{L^\infty( [0,T_0]; L^2(\mathbf{R}_x^d) )} \| g_0 \|_{L^\infty( [0,T_0]; H_x^m L_{v,w}^2 )} \| \nabla_x \Phi \|_{L^\infty( [0,T_0]; L^\infty(\mathbf{R}_x^d) )} \\
&\ \ \to 0 \quad \text{as} \quad \varepsilon \to 0.
\end{align*}
The convergence of the remainder $\mathcal{R}_\mathbb{P}(u^\varepsilon)$ to zero in the sense of distributions is guaranteed by Lemma \ref{RemHel0}.
We thus obtain Lemma \ref{HyLim:trans}.
\end{proof}

Summarizing all the convergence results that we have derived in this paper, we have proved that 
\begin{align*}
&\int_0^{T_0} \int_{\mathbf{R}^d} \Big\{ \partial_t \mathbb{P}(u^\varepsilon) - c_s^2 \nu \Delta_x \mathbb{P}(u^\varepsilon) + \mathbb{P}\big( \nabla_x \cdot (u^\varepsilon \otimes u^\varepsilon) \big) \Big\} \cdot \Phi \, dx \, dt \\
&\ \ \to \int_{\mathbf{R}^d} u_0 \cdot \Phi(x,0) \, dx - \left\{ \int_0^{T_0} \int_{\mathbf{R}^d} u \cdot \partial_t \Phi + (u \otimes u) : \nabla_x \Phi - c_s^2 \nu u \cdot \Delta_x \Phi \, dx \, dt \right\}
\end{align*}
as $\varepsilon \to 0$ for any $\Phi(x,t) \in C_\mathrm{c}^\infty( [0,T_0); C_{\mathrm{c}, \sigma}^\infty(\mathbf{R}_x^d) )$, i.e., 
\begin{align*}
u \in C( [0,T_0]; H^{m-1}(\mathbf{R}_x^d) ) \cap L^\infty( [0,T_0]; H^m(\mathbf{R}_x^d) )
\end{align*}
is a local weak solution to the incompressible Navier-Stokes equations
\begin{equation*}
\left\{
 \begin{aligned}
 \partial_t u - c_s^2 \nu \Delta_x u + \nabla_x \cdot (u \otimes u) + \nabla_x p &= 0,& \\
 \nabla_x \cdot u &= 0&
 \end{aligned}
\right.
\end{equation*}
with initial data $u(x,0) = \mathbb{P}(u_0)$.

\section{Characterization for isotropic lattices}
\label{Sec:IsoLat}

The purpose of this section is to characterize the $2$D and
$3$D isotropic lattices associated with the speed of sound $c_s = 3^{-\frac{1}{2}}$ under a restriction condition on the particle velocity set $\mathcal{V}$.

\subsection{Relation to the cubature formula}
\label{Sub:Cuba}

Given a lattice $(\mathcal{V}, w)$, by considering $\mathbf{n}_0 := \{0, 1, 2, ..., n\}$ as the state space, we can define a corresponding discrete and finite measure space $(\mathbf{n}_0, 2^{\mathbf{n}_0}, P)$ and measurable functions $X_\alpha$ ($\alpha \in \{1, 2, ..., d\}$) by setting 
\begin{align*}
P(\{ i \}) := w_i \quad \text{and} \quad X_\alpha(i) := v_{i, \alpha} \quad (\forall \; i \in \mathbf{n}_0).
\end{align*}
On the other hand, suppose that we are given a measure space $(\Omega, 2^\Omega, P)$ with $\Omega$ being a finite set and the probability measure $P$ having full support, i.e, this means that $P(\{ y_i \}) > 0$ for each $y_i \in \Omega$, and   measurable functions $X_\alpha$ ($\alpha \in \{1, 2, ..., d\}$) on this measure space.
Then, we can construct a lattice by setting 
\begin{align*}
v_{i \alpha} := X_\alpha(y_i) \quad \text{and} \quad w_i := P(\{ y_i \}) \quad (\forall \; y_i \in \Omega)
\end{align*}
and $v_i := (v_{i 1}, v_{i, 2}, ..., v_{i d})$.
Hence, there is a one-to-one mapping between lattices and finite measure spaces with full support combined with $d$ measurable functions on these measure spaces. 
Let $c_s > 0$ and $Y_\alpha := c_s X_\alpha$ for any $\alpha$.
Then, the definition of an isotropic lattice associated with the speed of sound $c_s$ can be rephrased as the following.

\begin{definition} \label{Def:isLa2}
A lattice $(\mathcal{V}, w)$ is isotropic if the following conditions hold for its corresponding measure space $(\Omega, 2^\Omega, P)$ and the measurable functions $Y_\alpha$ ($\alpha \in \{1, 2, ..., d\}$):
\begin{enumerate} 
\item $(\Omega, 2^\Omega, P)$ is a probability space. 
\item $E[Y_\alpha] = 0$ for any $1 \leq \alpha \leq d$. 
\item $E[Y_\alpha Y_\beta] = \delta_{\alpha \beta}$ for any $1 \leq \alpha, \beta \leq d$.
\item $E[Y_\alpha Y_\beta Y_\gamma] = 0$ for any $1 \leq \alpha, \beta, \gamma \leq d$.
\item $E[Y_\alpha Y_\beta Y_\gamma Y_\zeta] = \delta_{\alpha \beta} \delta_{\gamma \zeta} + \delta_{\alpha \gamma} \delta_{\beta \zeta} + \delta_{\alpha \zeta} \delta_{\beta \gamma}$ for any $1 \leq \alpha, \beta, \gamma, \zeta \leq d$.
\end{enumerate}
Here $E[ \cdot ]$ represents the expectation of $``\cdot"$, i.e., the integral of $``\cdot"$ with respect to measure $P$.
\end{definition}

Let $Z_\alpha$ ($\alpha \in \{1, 2, ..., d\}$) be independent Gaussian
random variables with $E[Z_\alpha] = 0$ and $E[Z_\alpha^2] = 1$ for $\alpha \in \{1, 2, ..., d\}$.
Then, conditions for $\mathbb{Y} = (Y_1, Y_2, ..., Y_d)$ in Definition \ref{Def:isLa2} are summarized as 
\begin{align*}
E[f(\mathbb{Y})] = E[f(\mathbb{Z})]
\end{align*}
for any polynomial $f$ of $d$ variables with degree less than or equal to $4$ where $\mathbb{Z} = (Z_1, Z_2, ..., Z_d)$.
To find such $\mathbb{Y} = (Y_1, Y_2, ..., Y_d)$ taking a finite number of values is a well-studied problem in the context of cubature formulas, see e.g., \cite{OSS}. 
For example, 
\begin{itemize}
\item For $d=2$, there is a solution with $n=7$, i.e., the D$2$Q$7$ scheme exists,
\item For $d=3$, there is a solution with $n=13$, i.e., the D$3$Q$13$ scheme exists,
\item For $d=4$, there is a solution with $n=22$,
\item For $d=5$, there is a solution with $n=33$,
\item For $d=6$, there is a solution with $n=44$,
\item For $d=7$, there is a solution with $n=57$,
\item For $d \geq 8$, there is a solution with $n = d^2 + 3d + 3$.
\end{itemize}

\subsection{$2$D and $3$D isotropic lattices}
\label{Sub:2d3diso}

In this section, we always consider $c_s = 3^{-\frac{1}{2}}$. Without
this causing of any ambiguity, when we refer to isotropic
lattices, we omit the words ``associated with the speed of sound $3^{-\frac{1}{2}}$''.

\begin{proposition} \label{Prob:Xa}
Suppose that the lattice $(\mathcal{V}, w)$ is isotropic and the support of $X_\alpha$ is in $[-1,1]$ for any $\alpha \in \{1, 2, ..., d\}$.
Then, 
\begin{align*}
P(X_\alpha = 0) = \frac{2}{3}, \quad P(X_\alpha = 1) = P(X_\alpha = -1) = \frac{1}{6}
\end{align*}
for any $\alpha \in \{1, 2, ..., d\}$.
\end{proposition}
\begin{proof}
Fix $\alpha \in \{1, 2, ..., d\}$.
Since the support of $X_\alpha$ is in $[-1,1]$, it certainly holds that $X_\alpha^2 \geq X_\alpha^4$, i.e., $P(X_\alpha^2 \geq X_\alpha^4) = 1$.
Since the algebraic conditions for an isotropic lattice say that $E[X_\alpha^2] = E[X_\alpha^4] = \frac{1}{3}$, we must have $P(X_\alpha^2 > X_\alpha^4) = 0$, i.e., $P(X_\alpha^2 = X_\alpha^4) = 1$.
That means that $P(X_\alpha \in \{-1, 0, 1\}) = 1$.
Then, $E[X_\alpha^2] = \frac{1}{3}$ implies that $P(X_\alpha = 0) = \frac{2}{3}$ and $E[X_\alpha] = 0$ implies that $P(X_\alpha = 1) = P(X_\alpha = -1) = \frac{1}{6}$.
\end{proof}
\begin{proposition} \label{Prob:XaXb}
Suppose that the lattice $(\mathcal{V}, w)$ is isotropic and the support of $X_\alpha$ is in $[-1,1]$ for any $\alpha \in \{1, 2, ..., d\}$.
Then, for $\alpha, \beta \in \{1, 2, ..., d\}$ satisfying $\alpha \neq \beta$, it holds that
\begin{itemize}
\item $P(X_\alpha = X_\beta = 0) = \frac{4}{9}$,
\item $P(X_\alpha = \sigma_\alpha, X_\beta = 0) = \frac{1}{9}$ for $\sigma_\alpha \in \{-1,1\}$,
\item $P(X_\alpha = \sigma_\alpha, X_\beta = \sigma_\beta) = \frac{1}{36}$ for $\sigma_\alpha, \sigma_\beta \in \{-1,1\}$.
\end{itemize}
\end{proposition}
\begin{proof}
Let $\alpha, \beta \in \{1, 2, ..., d\}$ with $\alpha \neq \beta$.
For any $k, j \in \{-1, 0, 1\}$, for simplicity of notations we denote $P(X_\alpha = k, X_\beta = j)$ by $p_{k,j}$.
Note that $E[X_\alpha X_\beta] = 0$ implies that
\begin{align} \label{E:XaXb}
p_{1,1} + p_{-1,-1} - p_{1,-1} - p_{-1,1} = 0,
\end{align}
$E[X_\alpha^2 X_\beta] = 0$ implies that
\begin{align} \label{E:Xa2Xb}
p_{1,1} - p_{-1,-1} - p_{1,-1} + p_{-1,1} = 0
\end{align}
and $E[X_\alpha^2 X_\beta^2] = 0$ implies that
\begin{align} \label{E:Xa2Xb2}
p_{1,1} + p_{-1,-1} + p_{1,-1} + p_{-1,1} = \frac{1}{9}.
\end{align}
Adding (\ref{E:XaXb}) and (\ref{E:Xa2Xb}) together, we obtain that $p_{1,1} = p_{1,-1}$ and $p_{-1,-1} = p_{-1,1}$. 
Adding (\ref{E:XaXb}) and (\ref{E:Xa2Xb2}) together, we obtain that
\begin{align} \label{p11pmm}
p_{1,1} + p_{-1,-1} = \frac{1}{18} = p_{1,-1} + p_{-1,1}.
\end{align}
Moreover, we have that
\begin{equation} \label{Decom:Pa}
P(X_\alpha = 1) = \frac{1}{6} = p_{1,1} + p_{1,0} + p_{1,-1}, \quad P(X_\alpha = -1) = \frac{1}{6} = p_{-1,1} + p_{-1,0} + p_{-1,-1}.
\end{equation}
By adding the two equations in (\ref{Decom:Pa}) together and making use of (\ref{p11pmm}), we deduce that
\begin{align} \label{p10pm0}
p_{1,0} + p_{-1,0} = \frac{2}{9}.
\end{align}
By Proposition \ref{Prob:Xa}, we have that $\frac{2}{3} = P(X_\beta = 0) = p_{1,0} + p_{0,0} + p_{-1,0}$. Hence, equality (\ref{p10pm0}) implies that $p_{0,0} = \frac{4}{9}$.
In addition, by the linearity of expectation, we can deduce from $E[X_\alpha + X_\beta] = 0$ that
\begin{equation} \label{XapXb}
 2p_{1,1} + p_{1,0} + p_{0,1} = p_{0,-1} + p_{-1,0} + 2p_{-1,-1}
\end{equation}
and from $E[X_\alpha - X_\beta] = 0$ that
\begin{equation} \label{XamXb}
p_{1,0} + 2p_{1,-1} + p_{0,-1} = p_{0,1} + 2 p_{-1,1} + p_{-1,0}.
\end{equation}
Subtracting (\ref{XamXb}) from (\ref{XapXb}) and noting that $p_{1,1} = p_{1,-1}$ and $p_{-1,-1} = p_{-1,1}$, we deduce that $p_{0,1} = p_{0,-1}$. 
Using Proposition \ref{Prob:Xa} again, we have that $\frac{2}{3} = P(X_\alpha = 0) = p_{0,1} + p_{0,0} + p_{0,-1}$, i.e., it holds that 
\begin{align} \label{Val:p01}
p_{0,1} = p_{0,-1} = \frac{1}{9}.
\end{align}
Since (\ref{Val:p01}) holds for any $\alpha, \beta \in \{1, 2, ..., d\}$ satisfying $\alpha \neq \beta$, by interchanging values of $\alpha$ and $\beta$, we deduce that $p_{1,0} = p_{-1,0} = \frac{1}{9}$.
Finally, by substituting values of $p_{1,0}$ and $p_{-1,0}$ back into (\ref{Decom:Pa}), we obtain that 
\begin{align*}
p_{1,1} = p_{1,-1} = \frac{1}{36} = p_{-1,-1} = p_{-1,1}.
\end{align*}
This completes the proof of Proposition \ref{Prob:XaXb}.
\end{proof}
\begin{corollary} \label{Iso:2D}
Let $d=2$ and $(\mathcal{V}, w)$ be a lattice.
Suppose that the support of $X_\alpha$ is in $[-1,1]$ for any $\alpha = 1,2$.
Then, the lattice $(\mathcal{V}, w)$ is isotropic if and only if the scheme is D$2$Q$9$.
\end{corollary}
\begin{proof}
Sufficiency is an outcome of Proposition \ref{Prob:Xa} and Proposition \ref{Prob:XaXb}.
Necessity is a well-known fact which can be easily checked by direct calculations, see e.g. \cite{LBM6}.
\end{proof}
\begin{proposition} \label{Pb:XaXbXg}
Suppose that the lattice $(\mathcal{V},w)$ is isotropic and the support of $X_\alpha$ is in $[-1,1]$ for any $\alpha \in \{1, ..., d\}$.
Suppose that $\alpha, \beta, \gamma \in \{1, ..., d\}$ satisfying $\alpha \neq \beta$, $\alpha \neq \gamma$ and $\beta \neq \gamma$. 
Then, there exist constants $c \in [0, \frac{1}{72}]$ and $c_\alpha$, $c_\beta$, $c_\gamma \in [- \frac{1}{72}, \frac{1}{72}]$ such that 
\begin{itemize}
\item $P(X_\alpha = X_\beta = X_\gamma = 0) = \frac{1}{3} - 8c$,
\item $P(X_\alpha = \sigma_\alpha, X_\beta = X_\gamma = 0) =
  \frac{1}{18} + 4c + 4 \sigma_\alpha c_\alpha$ \quad for \quad $\sigma_\alpha \in \{-1, 1\}$,
\item $P(X_\alpha = \sigma_\alpha, X_\beta = \sigma_\beta, X_\gamma = 0) = \frac{1}{36} - 2c - 2 \sigma_\alpha c_\alpha - 2 \sigma_\beta c_\beta$ \quad for \quad $\sigma_\alpha, \sigma_\beta \in \{-1, 1\}$,
\item $P(X_\alpha = \sigma_\alpha, X_\beta = \sigma_\beta, X_\gamma = \sigma_\gamma) = c + \sigma_\alpha c_\alpha + \sigma_\beta c_\beta + \sigma_\gamma c_\gamma$ \quad for \quad $\sigma_\alpha, \sigma_\beta, \sigma_\gamma \in \{-1, 1\}$.
\end{itemize}
Moreover, the constants $c, c_\alpha, c_\beta, c_\gamma$ satisfy the inequalities
\begin{equation*}
\left\{
 \begin{aligned}
 |c_\alpha| + |c_\beta| + |c_\gamma| &\leq& &c,& & & \\
 c + |c_\zeta| + |c_\eta| &\leq& &\frac{1}{72}& &\text{for any} \quad \zeta, \eta \in \{\alpha, \beta, \gamma\} \quad \text{with} \quad \zeta \neq \eta.&
 \end{aligned}
\right.
\end{equation*}
\end{proposition}
\begin{proof}
Let $\alpha, \beta, \gamma \in \{1, 2, ..., d\}$ satisfy the relations $\alpha \neq \beta$, $\beta \neq \gamma$ and $\alpha \neq \gamma$.
For any $i, j, k \in \{-1, 0, 1\}$, for simplicity of notations we denote $P(X_\alpha = i, X_\beta = j, X_\gamma=k)$ by $p_{i,j,k}$.
Since $E[X_\alpha X_\beta X_\gamma] = 0$ and $E[X_\alpha^2 X_\beta X_\gamma] = 0$, it holds that
\begin{equation} \label{apa2:bg}
 0 = E[(X_\alpha + X_\alpha^2) X_\beta X_\gamma] = 2 p_{1,1,1} - 2 p_{1,1,-1} -2 p_{1,-1,1} + 2 p_{1,-1,-1}
\end{equation}
and
\begin{equation} \label{ama2:bg}
0 = E[(X_\alpha - X_\alpha^2) X_\beta X_\gamma] = - 2 p_{-1,1,1} + 2 p_{-1,1,-1} + 2 p_{-1,-1,1} - 2 p_{-1,-1,-1}.
\end{equation}
By (\ref{apa2:bg}), we have that 
\begin{equation} \label{1:m1:1mm1}
p_{1,1,1} - p_{1,1,-1} = p_{1,-1,1} - p_{1,-1,-1}. 
\end{equation}
Replacing $(\alpha,\beta)$ by $(\beta, \alpha)$, we deduce from (\ref{1:m1:1mm1}) that
\begin{equation} \label{R:ab:1mm1}
\begin{split}
&P(X_\alpha = 1, X_\beta = - 1, X_\gamma = 1) - P(X_\alpha = 1, X_\beta = - 1, X_\gamma = - 1) \\
&\ \ = P(X_\alpha = 1, X_\beta = 1, X_\gamma = 1) - P(X_\alpha = 1, X_\beta = 1, X_\gamma = -1) \\
&\ \ = P(X_\beta = 1, X_\alpha = 1, X_\gamma = 1) - P(X_\beta = 1, X_\alpha = 1, X_\gamma = -1) \\
&\ \ = P(X_\beta = 1, X_\alpha = - 1, X_\gamma = 1) - P(X_\beta = 1, X_\alpha = - 1, X_\gamma = -1),
\end{split}
\end{equation}
i.e., we obtain that $p_{1,-1,1} - p_{1,-1,-1} = p_{-1,1,1} - p_{-1,1,-1}$.
In addition, (\ref{ama2:bg}) implies that
\begin{equation} \label{m1m11mm1}
p_{-1,1,1} - p_{-1,1,-1} = p_{-1,-1,1} - p_{-1,-1,-1}.
\end{equation}
Combining (\ref{1:m1:1mm1}), (\ref{R:ab:1mm1}) and (\ref{m1m11mm1}),
we can conclude that the quantity $p_{\sigma_\alpha, \sigma_\beta, 1} - p_{\sigma_\alpha, \sigma_\beta, -1}$ is independent of $\sigma_\alpha, \sigma_\beta \in \{-1, 1\}$. 
Hence, without causing any ambiguity, we may set $2 c_\gamma := p_{\sigma_\alpha, \sigma_\beta, 1} - p_{\sigma_\alpha, \sigma_\beta, -1}$ for $\sigma_\alpha, \sigma_\beta \in \{-1, 1\}$. 
By applying analogous argument to consider 
\begin{equation*}
 \begin{aligned}
 E[X_\alpha (X_\beta + X_\beta^2) X_\gamma] &=& 0 &=& E[X_\alpha (X_\beta - X_\beta^2) X_\gamma],& \\
 E[X_\alpha X_\beta (X_\gamma + X_\gamma^2)] &=& 0 &=& E[X_\alpha X_\beta (X_\gamma - X_\gamma^2)],& 
 \end{aligned}
\end{equation*}
we can deduce that $p_{1, \sigma_\beta, \sigma_\gamma} - p_{-1, \sigma_\beta, \sigma_\gamma} =: 2 c_\alpha$ is independent of $\sigma_\beta, \sigma_\gamma \in \{-1, 1\}$ and $p_{\sigma_\alpha, 1, \sigma_\gamma} - p_{\sigma_\alpha, - 1, \sigma_\gamma} =: 2 c_\beta$ is independent of $\sigma_\alpha, \sigma_\gamma \in \{-1, 1\}$.

Let $c := p_{1,1,1} - c_\alpha - c_\beta - c_\gamma$.
Then, 
\begin{equation} \label{p3:1:m1}
 \begin{aligned}
 p_{1,1,-1} &=& p_{1,1,1} - 2 c_\gamma &\implies& p_{1,1,-1} &=& &c + c_\alpha + c_\beta - c_\gamma, \\
 p_{-1,1,1} &=& p_{1,1,1} - 2 c_\alpha &\implies& p_{-1,1,1} &=& &c - c_\alpha + c_\beta + c_\gamma, \\
 p_{1,-1,1} &=& p_{1,1,1} - 2 c_\beta &\implies& p_{1,-1,1} &=& &c + c_\alpha - c_\beta + c_\gamma. \\
 \end{aligned}
\end{equation}
Using (\ref{p3:1:m1}), we further deduce that
\begin{equation*}
 \begin{aligned}
 p_{-1,1,-1} &=& p_{-1,1,1} - 2 c_\gamma &\implies& p_{-1,1,-1} &=& &c - c_\alpha + c_\beta - c_\gamma, \\
 p_{1,-1,-1} &=& p_{1,-1,1} - 2 c_\gamma &\implies& p_{1,-1,-1} &=& &c + c_\alpha - c_\beta - c_\gamma, \\
 p_{-1,-1,1} &=& p_{-1,1,1} - 2 c_\beta &\implies& p_{-1,-1,1} &=& &c - c_\alpha - c_\beta + c_\gamma, \\
 p_{-1,-1,-1} &=& p_{-1,-1,1} - 2 c_\gamma &\implies& p_{-1,-1,-1} &=& &c - c_\alpha - c_\beta - c_\gamma.
 \end{aligned}
\end{equation*}
Summarizing what we have obtained, it holds that
\begin{equation} \label{p3:1m1}
p_{\sigma_\alpha, \sigma_\beta, \sigma_\gamma} = c + \sigma_\alpha c_\alpha + \sigma_\beta c_\beta + \sigma_\gamma c_\gamma
\end{equation}
for any $\sigma_\alpha, \sigma_\beta, \sigma_\gamma \in \{-1, 1\}$ with $c := p_{1,1,1} - c_\alpha - c_\beta - c_\gamma$.

By the fourth bullet point of Proposition \ref{Prob:XaXb} and (\ref{p3:1m1}), we can deduce that for $\sigma_\beta, \sigma_\gamma \in \{-1, 1\}$
\begin{equation*}
p_{0,\sigma_\beta,\sigma_\gamma} = p_{\sigma_\beta, \sigma_\gamma} - p_{1,\sigma_\beta,\sigma_\gamma} - p_{-1,\sigma_\beta,\sigma_\gamma} = \frac{1}{36} - 2c - 2 \sigma_\beta c_\beta - 2 \sigma_\gamma c_\gamma.
\end{equation*}
Analogously, for $\sigma_\alpha, \sigma_\gamma \in \{-1, 1\}$, it holds that
\begin{equation} \label{p3:b0}
p_{\sigma_\alpha,0,\sigma_\gamma} = \frac{1}{36} - 2c - 2 \sigma_\alpha c_\alpha - 2 \sigma_\gamma c_\gamma
\end{equation}
and for $\sigma_\alpha, \sigma_\beta \in \{-1, 1\}$, it holds that
\begin{equation*}
p_{\sigma_\alpha,\sigma_\beta,0} = \frac{1}{36} - 2c - 2 \sigma_\alpha c_\alpha - 2 \sigma_\beta c_\beta.
\end{equation*}
Furthermore, for $\sigma_\gamma \in \{-1, 1\}$, we deduce from the third bullet point of Proposition \ref{Prob:XaXb} and (\ref{p3:b0}) that
\begin{align} \label{p3:a0b0}
p_{0,0,\sigma_\gamma} = p_{0,\sigma_\gamma} - p_{1,0,\sigma_\gamma} - p_{-1,0,\sigma_\gamma} = \frac{1}{18} + 4c + 4 \sigma_\gamma c_\gamma. 
\end{align}
Similarly, for $\sigma_\alpha \in \{-1, 1\}$, it holds that
\begin{align*}
p_{\sigma_\alpha,0,0} = \frac{1}{18} + 4c + 4 \sigma_\alpha c_\alpha
\end{align*}
and for $\sigma_\beta \in \{-1, 1\}$, it holds that
\begin{align*}
p_{0,\sigma_\beta,0} = \frac{1}{18} + 4c + 4 \sigma_\beta c_\beta.
\end{align*}
Using the first bullet point of Proposition \ref{Prob:XaXb} and (\ref{p3:a0b0}), we deduce that
\begin{align*}
p_{0,0,0} = p_{0,0} - p_{0,0,1} - p_{0,0,-1} = \frac{1}{3} - 8c.
\end{align*}

Finally, we let
\begin{equation*}
\widetilde{\operatorname{sgn}} (x) = \left\{
 \begin{aligned}
 &1,& &\text{if}& &x \geq 0, \\
 &-1,& &\text{if}& &x < 0.
 \end{aligned}
\right.
\end{equation*}
By considering (\ref{p3:1m1}), we can deduce from 
\begin{align*}
P(X_\alpha = - \widetilde{\operatorname{sgn}}(c_\alpha), X_\beta = - \widetilde{\operatorname{sgn}}(c_\beta), X_\gamma = - \widetilde{\operatorname{sgn}}(c_\gamma)) \geq 0
\end{align*}
that
\begin{align*} 
0 \leq |c_\alpha| + |c_\beta| + |c_\gamma| \leq c.
\end{align*}
Since $P(X_\alpha = \widetilde{\operatorname{sgn}}(c_\alpha), X_\beta = \widetilde{\operatorname{sgn}}(c_\beta), X_\gamma = 0) \geq 0$, we have that 
\begin{align*}
c + |c_\alpha| + |c_\beta| \leq \frac{1}{72}.
\end{align*}
Similarly, by considering $P(X_\alpha = \widetilde{\operatorname{sgn}}(c_\alpha), X_\beta = 0, X_\gamma = \widetilde{\operatorname{sgn}}(c_\gamma)) \geq 0$ and $P(X_\alpha = 0, X_\beta = \widetilde{\operatorname{sgn}}(c_\beta), X_\gamma = \widetilde{\operatorname{sgn}}(c_\gamma)) \geq 0$, we also have that
\begin{align*}
c + |c_\alpha| + |c_\gamma| \leq \frac{1}{72} \quad \text{and} \quad c + |c_\beta| + |c_\gamma| \leq \frac{1}{72}.
\end{align*}
This completes the proof of Proposition \ref{Pb:XaXbXg}.
\end{proof}
\begin{proposition} \label{Iso:3D}
Let $d=3$ and $(\mathcal{V}, w)$ be a lattice.
Suppose that the support of $X_\alpha$ is in $\{-1,0,1\}$ for any $\alpha = 1,2,3$.
For any $1 \leq \alpha, \beta, \gamma \leq 3$ satisfying $\alpha \neq \beta$, $\alpha \neq \gamma$ and $\beta \neq \gamma$, if there exist constants $c \in [0, \frac{1}{72}]$ and $c_\alpha$, $c_\beta$, $c_\gamma \in [- \frac{1}{72}, \frac{1}{72}]$ such that 
\begin{itemize}
\item $P(X_\alpha = 0) = \frac{2}{3}$ and $P(X_\alpha = 1) = P(X_\alpha = -1) = \frac{1}{6}$, 
\item $P(X_\alpha = X_\beta = 0) = \frac{4}{9}$,
\item $P(X_\alpha = \sigma_\alpha, X_\beta = 0) = \frac{1}{9}$ \quad  for \quad $\sigma_\alpha \in \{-1,1\}$,
\item $P(X_\alpha = \sigma_\alpha, X_\beta = \sigma_\beta) =  \frac{1}{36}$ \quad for \quad $\sigma_\alpha, \sigma_\beta \in \{-1,1\}$,
\item $P(X_\alpha = X_\beta = X_\gamma = 0) = \frac{1}{3} - 8c$,
\item $P(X_\alpha = \sigma_\alpha, X_\beta = X_\gamma = 0) =  \frac{1}{18} + 4c + 4 \sigma_\alpha c_\alpha$ \quad  for \quad $\sigma_\alpha \in \{-1, 1\}$,
\item $P(X_\alpha = \sigma_\alpha, X_\beta = \sigma_\beta, X_\gamma =  0) = \frac{1}{36} - 2c - 2 \sigma_\alpha c_\alpha - 2 \sigma_\beta  c_\beta$ \quad  for \quad $\sigma_\alpha, \sigma_\beta \in \{-1, 1\}$,
\item $P(X_\alpha = \sigma_\alpha, X_\beta = \sigma_\beta, X_\gamma = \sigma_\gamma) = c + \sigma_\alpha c_\alpha + \sigma_\beta c_\beta + \sigma_\gamma c_\gamma$ \quad for \quad$\sigma_\alpha, \sigma_\beta, \sigma_\gamma \in \{-1, 1\}$,
\end{itemize}
then $(\mathcal{V}, w)$ is isotropic.
\end{proposition}
\begin{proof}
This proof contains nothing but direct calculations. 
It can be easily verified that 
\begin{equation*}
\left\{
 \begin{aligned}
 E[X_\alpha] &= 0, \\
 E[X_\alpha^2] &= \frac{1}{3}, \\
 E[X_\alpha^3] &= 0, \\
 E[X_\alpha^4] &= \frac{1}{3}, \\
 \sum_{\sigma_\alpha \in \{-1,0,1\}} P(X_\alpha = \sigma_\alpha) &= 1, \\
 \sum_{\sigma_\alpha, \sigma_\beta \in \{-1,0,1\}} P(X_\alpha = \sigma_\alpha, X_\beta = \sigma_\beta) &= 1, \\
 \sum_{\sigma_\alpha, \sigma_\beta, \sigma_\gamma \in \{-1,0,1\}} P(X_\alpha = \sigma_\alpha, X_\beta = \sigma_\beta, X_\gamma = \sigma_\gamma) &= 1.
 \end{aligned}
\right.
\end{equation*}
For $\alpha, \beta \in \{1,2,3\}$ with $\alpha \neq \beta$, we have that
\begin{equation*}
\left\{
 \begin{aligned}
 E[X_\alpha X_\beta] &= \frac{1}{36} \cdot \{ 1^2 + 1 \cdot (-1) + (-1) \cdot 1 + (-1)^2 \} = 0, \\
 E[X_\alpha^2 X_\beta] &= \frac{1}{36} \cdot \{ 1^3 + 1^2 \cdot (-1) + (-1)^2 \cdot 1 + (-1)^2 \cdot (-1) \} = 0, \\
 E[X_\alpha^3 X_\beta] &= \frac{1}{36} \cdot \{ 1^4 + 1^3 \cdot (-1) + (-1)^3 \cdot 1 + (-1)^3 \cdot (-1) \} = 0, \\
 E[X_\alpha^2 X_\beta^2] &= \frac{1}{36} \cdot \{ 1^4 + 1^2 \cdot (-1)^2 + (-1)^2 \cdot 1^2 + (-1)^2 \cdot (-1)^2 \} = \frac{1}{9}.
 \end{aligned}
\right.
\end{equation*}
For $\alpha, \beta, \gamma \in \{1,2,3\}$ with $\alpha \neq \beta$, $\alpha \neq \gamma$ and $\beta \neq \gamma$, we have that
\begin{align*}
E[X_\alpha X_\beta X_\gamma] &=  1^3 \cdot (c + \sigma_\alpha + \sigma_\beta + \sigma_\gamma) + 1^2 \cdot (-1) \cdot (c + \sigma_\alpha + \sigma_\beta - \sigma_\gamma) \\
&\ \ + 1 \cdot (-1) \cdot 1 \cdot (c + \sigma_\alpha - \sigma_\beta + \sigma_\gamma) + 1 \cdot (-1) \cdot (-1) \cdot (c + \sigma_\alpha - \sigma_\beta - \sigma_\gamma) \\
&\ \ + (-1) \cdot 1^2 \cdot (c - \sigma_\alpha + \sigma_\beta + \sigma_\gamma) + (-1) \cdot 1 \cdot (-1) \cdot (c - \sigma_\alpha + \sigma_\beta - \sigma_\gamma) \\
 &\ \ + (-1) \cdot (-1) \cdot 1 \cdot (c - \sigma_\alpha - \sigma_\beta + \sigma_\gamma) + (-1) \cdot (-1) \cdot (-1) \cdot (c - \sigma_\alpha - \sigma_\beta - \sigma_\gamma) \\
&= 0
\end{align*}
and
\begin{align*}
E[X_\alpha^2 X_\beta X_\gamma] &=  1^4 \cdot (c + \sigma_\alpha + \sigma_\beta + \sigma_\gamma) + 1^3 \cdot (-1) \cdot (c + \sigma_\alpha + \sigma_\beta - \sigma_\gamma) \\
&\ \ + 1^2 \cdot (-1) \cdot 1 \cdot (c + \sigma_\alpha - \sigma_\beta + \sigma_\gamma) + 1^2 \cdot (-1) \cdot (-1) \cdot (c + \sigma_\alpha - \sigma_\beta - \sigma_\gamma) \\
&\ \ + (-1)^2 \cdot 1^2 \cdot (c - \sigma_\alpha + \sigma_\beta + \sigma_\gamma) + (-1)^2 \cdot 1 \cdot (-1) \cdot (c - \sigma_\alpha + \sigma_\beta - \sigma_\gamma) \\
 &\ \ + (-1)^2 \cdot (-1) \cdot 1 \cdot (c - \sigma_\alpha - \sigma_\beta + \sigma_\gamma) + (-1)^2 \cdot (-1) \cdot (-1) \cdot (c - \sigma_\alpha - \sigma_\beta - \sigma_\gamma) \\
&= 0.
\end{align*}
We thus obtain Proposition \ref{Iso:3D}.
\end{proof}
\begin{remark} \label{Cha:Iso3}
Combining Proposition \ref{Pb:XaXbXg} and Proposition \ref{Iso:3D}, we obtain a characterization for $3$D isotropic lattices.
Different from the $2$D case, $3$D isotropic lattice associated with the speed of sound $3^{-\frac{1}{2}}$ is not unique even if we require $X_\alpha$ to be supported in $[-1,1]$.
For example, 
\begin{itemize}
\item If $c = \frac{1}{72}$ and $c_\alpha = c_\beta = c_\gamma = 0$, this corresponds to the D$3$Q$15$ scheme,
\item If $c = c_\alpha = c_\beta = c_\gamma = 0$, this corresponds to the D$3$Q$19$ scheme,
\item If $c = \frac{1}{216}$ and $c_\alpha = c_\beta = c_\gamma = 0$, this corresponds to the D$3$Q$27$ scheme,
\end{itemize}
see e.g. \cite{LBM6}.
\end{remark}
\begin{remark} \label{Sharp:cs}
The characterizations of $2$D and $3$D isotropic lattices (Corollary \ref{Iso:2D}, Proposition \ref{Pb:XaXbXg} and \ref{Iso:3D}, respectively) in this section rely sharply on the assumption $c_s = 3^{-\frac{1}{2}}$.
Indeed, suppose that $c_s \neq 3^{-\frac{1}{2}}$. Then, even if we
assume that $X_\alpha$ is supported in $[-1,1]$, it is possible for
$X_\alpha$ to be supported on values other than $-1,0,1$.
For example, in the $2$D case, if we consider $\mathcal{V} = \{v_0, v_1, v_2, ..., v_6\} \subset \mathbf{R}^2$ where
\begin{equation*}
\left\{
\begin{aligned}
v_0 &= 0,& \\
v_j &= \Big( \cos{\frac{2 \pi j}{6}}, \, \sin{\frac{2 \pi j}{6}} \Big), \quad j = 1, ..., 6&
\end{aligned}
\right.
\end{equation*}
and $w = (w_0, w_1, w_2, ..., w_6) \in \mathbf{R}_+^7$ where
\begin{equation*}
\left\{
\begin{aligned}
w_0 &= \frac{1}{2},& \\
w_j &= \frac{1}{12}, \quad j = 1, ..., 6,&
\end{aligned}
\right.
\end{equation*}
then it can be easily checked that the combination $(\mathcal{V}, w)$ is an isotropic lattice associated with the speed of sound $c_s = \frac{1}{2}$, which corresponds to the D$2$Q$7$ scheme.
If we apply the proof of Proposition \ref{Prob:Xa} to the setting of $(\mathcal{V},w)$, then we shall face the problem that $E[X_\alpha^2] \neq E[X_\alpha^4]$, which further induces the possibility that $P(X_\alpha^2 > X_\alpha^4) > 0$.
Hence, it is not necessary for $X_\alpha$ to support in $\{-1,0,1\}$.
\end{remark}
\begin{remark} \label{supp:Xa}
It is worth noting that requiring the random variable $X_\alpha$ to be
supported in $[-1,1]$ is not necessary for a lattice to be
isotropic. If $(\mathcal{V},w)$ is an isotropic lattice associated
with the speed of sound $c_s$, then for any real number $\lambda > 0$, the lattice $(\lambda \mathcal{V}, w)$ is also an isotropic lattice associated with speed of sound $\frac{c_s}{\lambda}$ where $\lambda \mathcal{V} := \{ \lambda v_i \bigm| 0 \leq i \leq n \}$.
\end{remark}

\section{Numerical investigation of the hydrodynamic limit of the Lattice BGK problem}
\label{Sec:NumSol}

In this section we conduct numerical computations designed to
  illustrate the hydrodynamic limit investigated in Section
  \ref{Sec:HydroLi}. This is done by comparing the macroscopic
  velocity fields derived from the solutions of the Lattice BGK system
  \eqref{LBE} which are obtained for a sequence of decreasing values
  of $\varepsilon$ to the solutions of the Navier-Stokes system
  \eqref{InNS:lim} where the initial condition for the former system,
  $g^{0,\varepsilon}$, is chosen to be consistent with the initial
  condition for the latter, $u_0$, cf.~Theorem \ref{MT}. For
  simplicity and to fix attention, we focus here on the 2D setting
  ($d = 2$) and consider both the Lattice BGK and the Navier-Stokes
  systems \eqref{LBE} and \eqref{InNS:lim} on a torus
  $\mathbf{T}_x^2 := [0,1]^2$ rather than on an unbounded domain
  $\mathbf{R}_x^2$, which is motivated by computational
  considerations. Defining vorticity
  $\omega := \nabla_x^\perp \cdot u$, where
  $\nabla_x^\perp := (-\partial_{x_2},\partial_{x_1})$, it is
  convenient to rewrite the Navier-Stokes system in the vorticity form
  obtained applying the operator $\left( \nabla_x^\perp \cdot \right)$
  to the first equation in \eqref{InNS:lim}
\begin{equation} \label{2DNSvort}
\left\{
 \begin{aligned}
   \partial_t \omega + (u\cdot \nabla_x)\omega - c_s^2 \nu \Delta_x \omega &= 0,& \\
 u &=  \nabla_x^\perp \Delta_x^{-1} \omega,& \\
 \omega(x,0) &= \omega_0 := \nabla_x^\perp \cdot u_0,&
 \end{aligned}
\right.
\end{equation}
where $\Delta_x^{-1}$ is the inverse Laplacian equipped with the
periodic boundary conditions on $\mathbf{T}_x^2$. Since solutions of
the 2D Navier-Stokes system \eqref{2DNSvort} conserve the mean
vorticity, $(d/dt) \, \int_{\mathbf{T}_x^2} \omega(x,t) \, dx = 0$,
$t \ge 0$, it is necessary to restrict the admissible initial
conditions $\omega_0$ to have a zero mean,
$\int_{\mathbf{T}_x^2} \omega_0(x) \, dx = 0$, as this will ensure the
Laplacian can be inverted at all times (cf.~the second equation in
\eqref{2DNSvort}).  For the LBGK problem we consider the D2Q9 lattice
given by \eqref{D2Q9w}--\eqref{D2Q9v} and the speed of sound
$c_s = \frac{1}{\sqrt{3}}$.

We then focus on the following two macroscopic initial conditions
which for convenience are expressed here in the vorticity form:
\begin{itemize}
\item the Taylor-Green vortex
  \begin{equation}
\omega_0(x_1,x_2) = 10 \sin(2 \pi a x_1) \sin(2 \pi b x_2),
    \label{TG}
  \end{equation}
  where $a=b=2$, which leads to a solution of system \eqref{2DNSvort} where the
  nonlinear term $(u\cdot \nabla_x)\omega$ vanishes identically for
  all times $t \ge 0$, such that $\omega$ effectively solves the heat
  equation $\partial_t \omega - c_s^2 \nu \Delta_x \omega = 0$ on
  $\mathbf{T}_x^2$ and has therefore the form \cite{wmz06}
   \begin{equation}
\omega(x_1,x_2,t) = \omega_0(x_1,x_2) \ \exp\left[- 4 \pi^2\left( a^2 + b^2 \right) c_s^2 \nu \, t\right],
\label{TGsol}
\end{equation}

\item the perturbed Taylor-Green vortex
  \begin{equation}
\omega_0(x_1,x_2) = - \sin(2 \pi x_1) \sin(2 \pi x_2) + \exp\left[ - 50
  \left((x_1 - 1/2)^2 + (x_2 - 1/2)^2\right)\right] + C,
    \label{TGvort}
  \end{equation}
  where the constant $C$ is chosen to ensure that this initial data
  satisfies the zero-mean condition; solutions of the Navier-Stokes
  system \eqref{2DNSvort} subject to the initial condition
  \eqref{TGvort} need to be found numerically as described below in
  Section \ref{numerics}; since our spatial domain is the 2D
    torus, the function in \eqref{TGvort} should be interpreted as a
    periodic extension of a function restricted to
    $0 \le x_1,x_2 \le 1$ and since the second term on the right-hand
    side is not a periodic function, this initial condition has in
    fact discontinuous derivatives; however, due to the rapid decay of
    the exponential function away from the point $(1/2,1/2)$, the
    magnitude of this discontinuity is negligible and does not affect
  the numerical results.
\end{itemize}
 The microscopic initial condition
  $g^{0,\varepsilon}$ for the Lattice BGK system is then obtained from
  \eqref{TG} or \eqref{TGvort} using relation \eqref{gi0eps} in which
  we set $u_0 = \nabla_x^\perp \Delta_x^{-1} \omega_0$ and, without
  the loss of generality, $\rho_0(x_1,x_2) = 1$,
  $\forall (x_1,x_2) \in \mathbf{T}_x^2$.

\subsection{Numerical solution of the Lattice BGK and Navier-Stokes
  equations}
\label{numerics}

The Lattice BGK and Navier-Stokes systems \eqref{LBE} and
\eqref{2DNSvort} are approximated numerically using a standard
pseudo-spectral approach where the dependence of the solution on the
space variable $x$ is represented in terms of a truncated Fourier
series and the use of this ansatz can be interpreted as
  application of the cutoff function $\Lambda_\varepsilon$ introduced
  in \eqref{Ctf:op}. Derivatives are then evaluated exactly in the
Fourier (spectral) space whereas all product terms are computed in the
physical space with dealiasing \cite{Canuto2006book}.  Discrete
Fourier transforms relating the two representations are evaluated
using the FFT algorithm. The system of coupled ordinary differential
equations (ODEs) resulting from this approximation is then integrated
in time using the standard fourth-order Runge-Kutta method (RK4). This
is an explicit approach and hence only conditionally stable, so care
must be exercised to ensure the stability of these computations by
using a sufficiently small time step $\Delta t$. In fact, the Lattice
BGK problem is ``stiff'', in the sense that the maximum allowable time
step decreases as $\varepsilon \rightarrow 0$. For both problems
introduced above, the stability and accuracy of computations was
carefully verified by performing these computations with different
time steps $\Delta t$ and different numerical resolutions $N$ (where
$N$ is the number of grid points used to discretize the domain
$\mathbf{T}_x^2$ in each direction). Moreover, the numerical solution
of the Navier-Stokes system \eqref{2DNSvort} can be validated by
considering the problem with the initial condition \eqref{TG}, where
an exact solution is available, cf.~\eqref{TGsol}. This is in fact a
demanding test as it requires the vanishing of the nonlinear term in
the approximate solution. This approach has been numerically
implemented in MATLAB and the code is available on Github
\cite{MatharuGitBoltzmann}.

In the numerical solutions of the Lattice BGK and the Navier-Stokes
system \eqref{LBE} and \eqref{2DNSvort} discussed below the spatial
resolution is $N = 128$. In the latter case the time step is
$\Delta t= 2\cdot 10^{-6}$, whereas in the former it is in the range
$\Delta t \in [2\cdot 10^{-6}, 2\cdot 10^{-4}]$ depending on the value
of $\varepsilon$ ($\Delta t$ is smaller for decreasing
$\varepsilon$). Due to this limitation, we did not consider
  values of $\varepsilon$ smaller than $10^{-1}$.

\subsection{Results}
\label{results}

Before examining the dependence of the difference between the
macroscopic solutions
$\omega^\varepsilon := \nabla_x^\perp \cdot u^\varepsilon$ of the Lattice
BGK system \eqref{LBE} and the solutions $\omega$ of the Navier-Stokes
system \eqref{2DNSvort} on $\varepsilon$, we provide some additional
information about the flows considered. In both cases, the kinematic
viscosity (or the relaxation time in case of the Lattice BGK system)
is $\nu = 10^{-4}$.

As is evident from its exact solution in \eqref{TGsol}, due to the
absence of the nonlinear effects, the structure of the Navier-Stokes
flow corresponding to the initial condition \eqref{TG} remains
unchanged in time with only its magnitude vanishing
exponentially. Snapshots of the vorticity field $\omega(x,t)$ at times
$t=0,8,16,32$ during the evolution of the Navier-Stokes flow with the
initial condition \eqref{TGvort} are shown in Figure
\ref{fig:vort4}. Stretching of the vortices present in the initial
field $\omega_0$ is evident, giving rise to the formation of thin
elongated filaments which is the main dynamic mechanism sustaining the
enstrophy cascade in 2D turbulence \cite{Lesieur1993book}. To further
characterize this flow, we define the enstrophy and palinstrophy as
\begin{align}
  \mathcal{E}(t) & := \frac{1}{2} \int_{\mathbf{T}_x^2} \omega^2 \,dx,  \label{Et} \\
  \mathcal{P}(t) & := \frac{1}{2} \int_{\mathbf{T}_x^2} |\nabla\omega|^2 \,dx,  \label{Pt} 
\end{align}
which are equivalent to, respectively, the $L^2$ norm and the $H^1$
seminorm of the vorticity. It can be shown that, for smooth solution
of the Navier-Stokes system \eqref{2DNSvort}, the time evolution of
these quantities is governed by the equations \cite{ap14}
\begin{align}
  \frac{d\mathcal{E}}{dt} & = - 2 \nu \mathcal{P},  \label{dEdt} \\
  \frac{d\mathcal{P}}{dt} & =  \int_{\mathbf{T}_x^2} (u\cdot\nabla_x)\omega \Delta_x \omega \,dx
                           - \nu \int_{\mathbf{T}_x^2} ( \Delta_x \omega)^2 \,dx. \label{dPdt} 
\end{align}
We see that, while the enstrophy $\mathcal{E}(t)$ is a
nonincreasing function of time, the evolution of the palinstrophy is a
result of the competition between the stretching of vorticity
filaments and viscous dissipation represented, respectively, by the
cubic and the negative-definite quadratic term in \eqref{dPdt}. This
behavior is indeed confirmed by the data shown in Figures
\ref{fig:EtPt}a and \ref{fig:EtPt}b, where, in particular, we observe
that the palinstrophy $\mathcal{P}(t)$ initially increases by a factor
of about three. This demonstrates that the Navier-Stokes flow
corresponding to the initial condition \eqref{TGvort} is at least for
the times considered in Figures \ref{fig:vort4} and \ref{fig:EtPt}
dominated by nonlinear effects. An animated version of these figures
is available online on YouTube \cite{GHMPSY24a_Supp1}.

\begin{figure}
  \centering
 \includegraphics[width=0.8\textwidth]{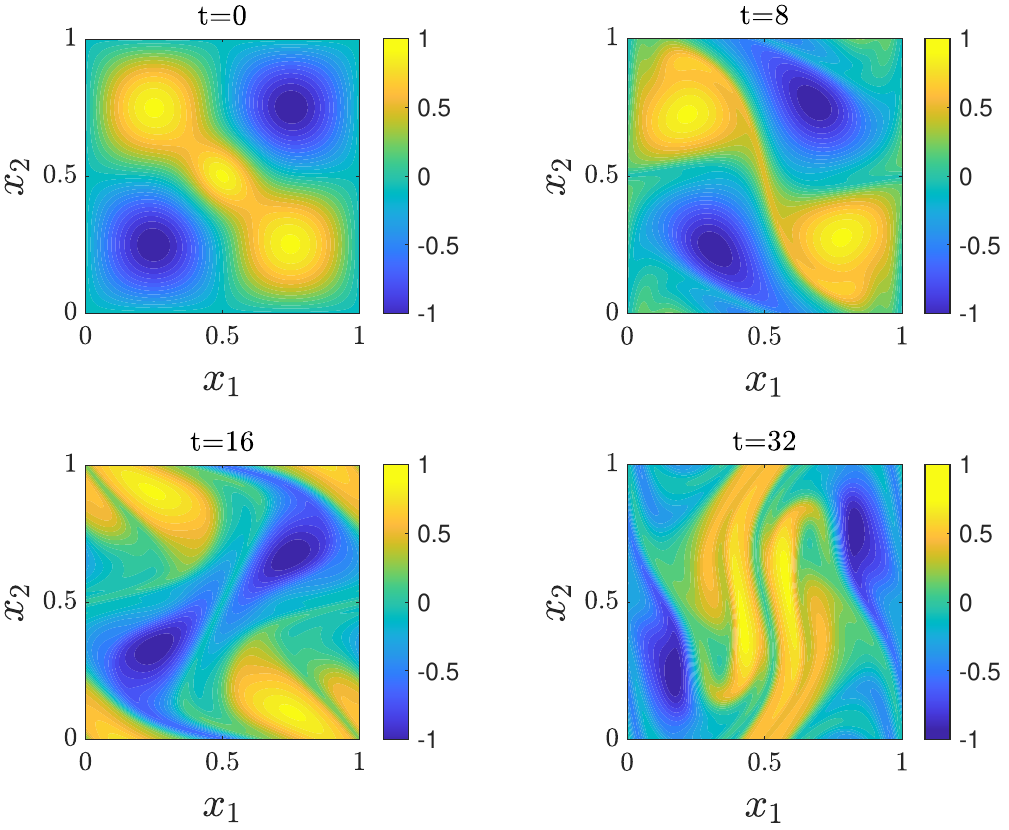}
 \caption{Snapshots of the vorticity field $\omega(x_1,x_2,t)$ at
   times $t=0,8,16,32$ during the evolution of the Navier-Stokes flow
   with the initial condition \eqref{TGvort}. }
  \label{fig:vort4}
\end{figure}

\begin{figure}
  \centering
 \mbox{
     \subfigure[]{\includegraphics[width=0.45\textwidth]{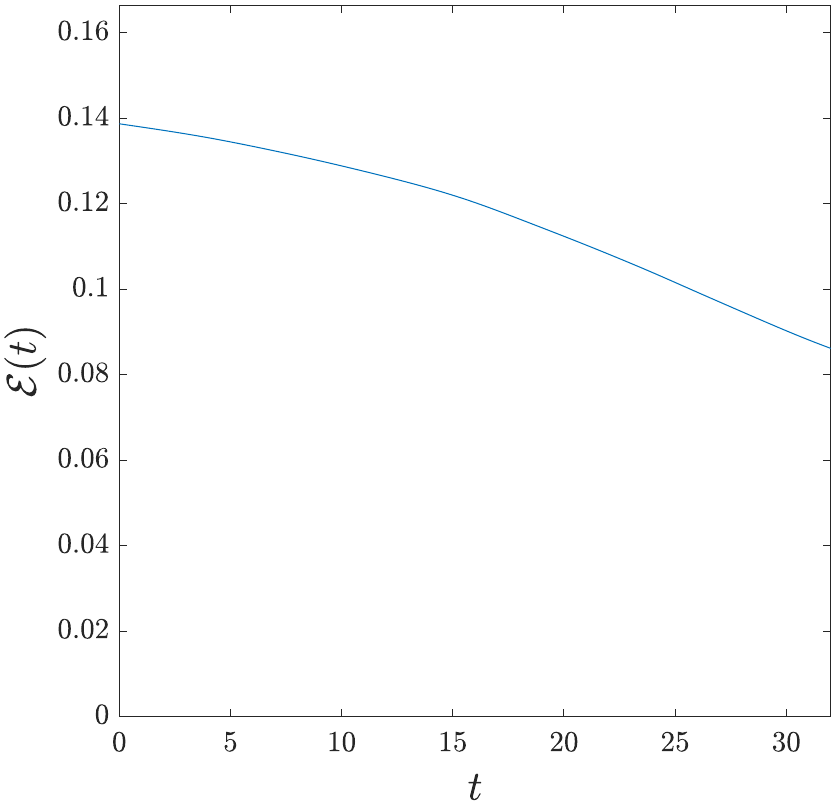}}\qquad
     \subfigure[]{\includegraphics[width=0.45\textwidth]{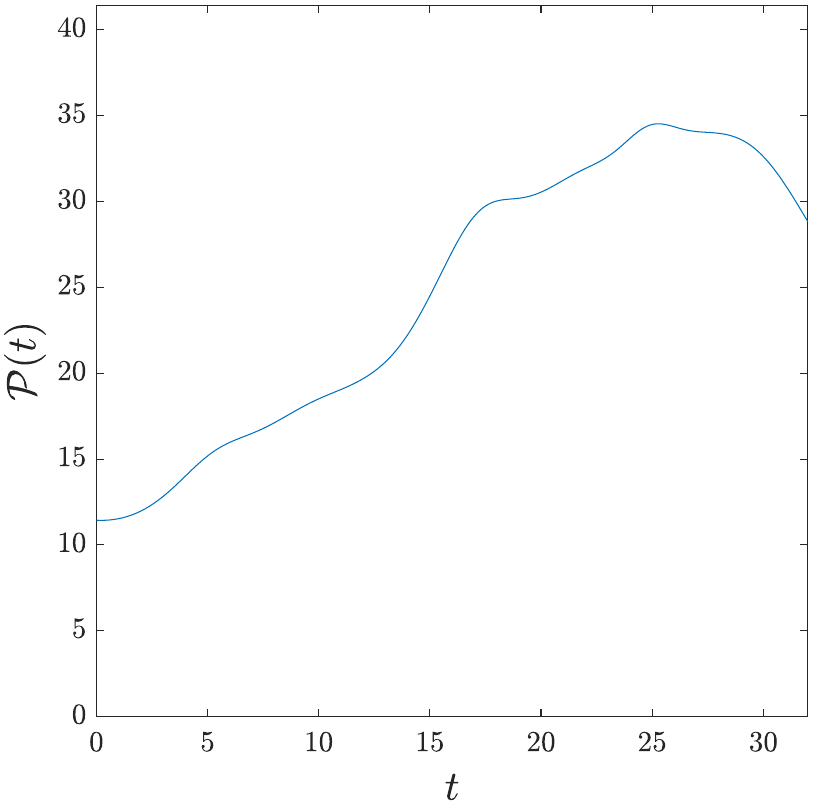}}}
   \caption{The time evolution of (a) the enstrophy $\mathcal{E}(t)$
     and (b) the palinstrophy $\mathcal{P}(t)$,
     cf.~\eqref{Et}--\eqref{Pt}, in the Navier-Stokes flow with the
     initial condition \eqref{TGvort}. }
  \label{fig:EtPt}
\end{figure}

Finally, our main results are shown in Figure
\ref{fig:convergence_eps2} where we plot the (normalized) norms of the
difference between the macroscopic vorticity $\omega^\varepsilon$ in
the Lattice BGK solution and the vorticity $\omega$ in the
Navier-Stokes flow at a certain time $t=T$ as functions of
$\varepsilon$.  The time when this difference is evaluated is chosen
as $T = 1$ and $T= 32$ for the problems with the initial conditions
\eqref{TG} and \eqref{TGvort}. This latter time corresponds to an
instance shortly after the palinstrophy $\mathcal{P}(t)$ has peaked,
cf.~Figure \ref{fig:EtPt}b.  It is clear from Figure
\ref{fig:convergence_eps2} that the norm of the difference can in both
cases be represented very accurately by a power-law relation obtained
via a least-squared fit
\begin{equation}
  \begin{aligned}
\frac{\| \omega(T) - \omega^\varepsilon(T) \|_{L^2
    (\mathbf{T}_x^2)}}{\| \omega(T)  \|_{L^2 (\mathbf{T}_x^2)}} 
& \approx \begin{cases}
  1.0458\cdot 10^{-6} \, \varepsilon^{2.0022} & \text{for the flow with initial condition \eqref{TG}} \\
7.2178\cdot 10^{-4} \, \varepsilon^{2.0001} & \text{for the flow with initial condition \eqref{TGvort}}
\end{cases}  \\
& = \mathcal{O}(\varepsilon^2).
\end{aligned}
\label{eps2}
\end{equation}
Essentially the same power-law behaviour (except for a different
prefactor) was also observed in both cases for different values of
$T$. We note that in the light of the identity
$\| \nabla_x^\perp \cdot u \|_{L^2 (\mathbf{T}_x^2)} = \| \nabla_x u
\|_{L^2 (\mathbf{T}_x^2)}$, relation \eqref{eps2} provides an
indication about the rate with which the hydrodynamic limit is
achieved in 2D in terms of the $H^1$ norm (at the level of velocity).

\begin{figure}
  \centering
 \includegraphics[width=0.6\textwidth]{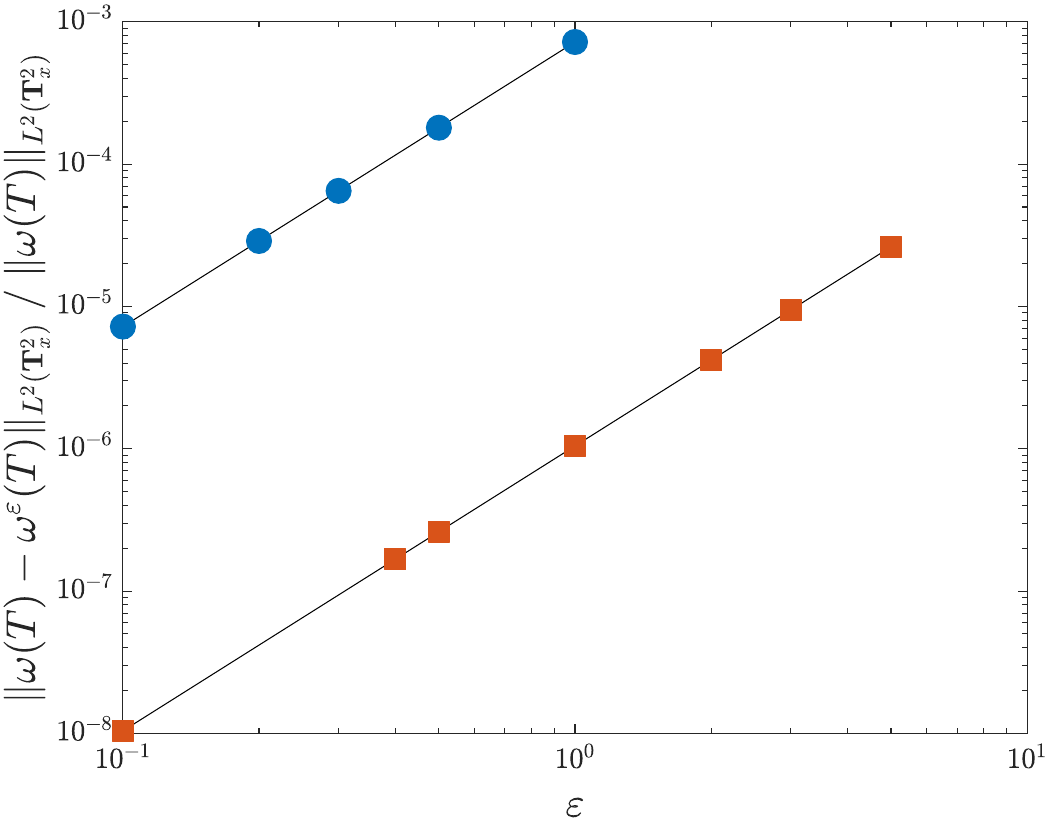}
  \caption{Dependence of $\| \omega(T) - \omega^\varepsilon(T) \|_{L^2
    (\mathbf{T}_x^2)} / \| \omega(T)  \|_{L^2 (\mathbf{T}_x^2)}$ on
  $\varepsilon$ for the Lattice BGK and Navier-Stokes flows with the
  initial conditions \eqref{TG} (red squares) and \eqref{TGvort} (blue
  circles). Solid lines represent the power-law fits \eqref{eps2}.}
  \label{fig:convergence_eps2}
\end{figure}

\section{Summary and Conclusions}

Here are some open questions and points for further discussion:
\begin{itemize}
\item Does the exponent of 2 in \eqref{eps2} depend on the spatial
  dimension ($d = 2$ vs.~$d = 3$) and the type of lattice (D2Q7
  vs.~D2Q9 in 2D)?
\item Can one establish this rate of convergence to the hydrodynamic
  limit rigorously?
\item How does the analysis and numerical results change when bounded domains are considered?
\end{itemize}

\section*{Acknowledgements}

BP was partially supported through an NSERC (Canada) Discovery Grant.
PM thanks the Japan Society for the Promotion of Science and MITACS
for awarding them the Mitacs-JSPS research fellowship to conduct this
work.
Research of TY was partly supported by the JSPS Grants-in-Aid for Scientific Research 24H00186.

\FloatBarrier



\end{document}